\pdfoutput=1
\documentclass[11pt]{amsart}
\usepackage{pinlabel}
\usepackage{mathrsfs}
\usepackage{amsmath} 
\usepackage{cjhebrew}
\usepackage{graphicx} 
\usepackage{subcaption}
\usepackage{tikz-cd}
\usepackage[all,cmtip]{xy}
\setkeys{Gin}{keepaspectratio}
\usepackage{hyperref}
\usepackage{url}
\usepackage{bm}
\usepackage{mathabx}
\usepackage[left=1.25in,top=1in,right=1.25in,bottom=1in,head=.1in]{geometry}
\usepackage{xcolor}
\usepackage{tikz}
\usepackage{adjustbox}
\usepackage[normalem]{ulem} %
\usepackage{enumerate}
\usepackage{hyperref}

\usepackage{verbatim}
\newcommand{\items}{\begin{itemize}[leftmargin=25pt,rightmargin=15pt]
  \setlength\itemsep{2pt}}

\usepackage{hyperref}

\usepackage{fancyhdr}

\newtheorem{thm}{Theorem} 

\newtheorem{theorem}{Theorem}[section] 
\newtheorem*{theorem*}{Theorem}
\newtheorem{lemma}[theorem]{Lemma}

\newtheorem{question}[theorem]{Question}
\newtheorem*{conjecture*}{Conjecture}
\newtheorem*{question*}{Question}
\newtheorem*{lemma*}{Lemma}
\newtheorem{proposition}[theorem]{Proposition}
\newtheorem{corollary}[theorem]{Corollary}
\newtheorem*{corollary*}{Corollary}
\theoremstyle{definition}
\newtheorem{definition}[theorem]{Definition}

\newtheorem{remark}[theorem]{Remark}

\newtheorem*{example*}{Example}
\newtheorem*{remark*}{Remark}
\newtheorem*{claim}{Claim}
\newtheorem*{remarks*}{Remarks}
\newtheorem*{addenda*}{Addenda}
\newtheorem*{construction*}{Construction}

\renewcommand{\phi}{\varphi}

\newcommand{\unred}[1]{ \ignorespaces}  

\usepackage{stmaryrd}

\title[Configurations of Lagrangian spheres in $K3$ surfaces]{Configurations of Lagrangian spheres in $K3$ surfaces}

\author{Juan Muñoz-Echániz}
\address{Simons Center for Geometry and Physics, State University of New York, Stony Brook, 11794, USA}
\email{jmunozechaniz@scgp.stonybrook.edu}

\begin{document}

\setlength{\headheight}{12.0pt}

\begin{abstract}
We study Dehn--Seidel twists on configurations of Lagrangian spheres in symplectic $K3$ surfaces, using tools from Seiberg--Witten theory. In the case of $ADE$ configurations of Lagrangian spheres, we prove that a naturally associated representation of the generalised Braid group in the symplectic mapping class group is always faithful after abelianising, in a suitable sense. More generally, we prove that squared Dehn--Seidel twists on homologically-distinct Lagrangian spheres are algebraically independent in the abelianisation of the smoothly-trivial symplectic mapping class group, and deduce from this new infinite-generation results. Beyond symplectic $K3$ surfaces, we also establish analogues of these results at the level of the fundamental group of the space of symplectic forms. 
%
\end{abstract}

\maketitle

\section{Introduction}


For a closed symplectic $4$-manifold $(X, \omega)$, a fundamental problem is to understand the structure of the \textit{symplectic mapping class group} $\pi_0 \mathrm{Symp}(X, \omega )$ together with the \textit{smoothly-trivial symplectic mapping class group} 
\[
\pi_0 \mathrm{Symp}_0 (X, \omega ) :=  \mathrm{ker}\big( \pi_0 \mathrm{Symp}(X, \omega ) \to \pi_0 \mathrm{Diff}(X) \big).
\]

Beyond a handful examples, a complete description of these groups seems currently beyond reach, and a more reasonable goal is to probe subgroups generated by distinguished mapping classes. In this article we study subgroups generated by \textit{Dehn--Seidel twists} on $ADE$ configurations of Lagrangian spheres (Theorems \ref{theorem:ADEK3}-\ref{theorem:ADE}) and---more generally---homologically distinct configurations of Lagrangian spheres (Theorems \ref{theorem:summandK3}-\ref{theorem:summand}). Our results are strongest in the case when $(X, \omega)$ is a \textit{symplectic $K3$ surface}, meaning that $X$ is diffeomorphic to the $K3$ surface. (We recall that results by Taubes \cite{Taubes:more} imply that the canonical class $K := - c_1 (TX, \omega )$ of a symplectic $K3$ surface is necessarily trivial). 


Dehn--Seidel twists on configurations of Lagrangian spheres have been extensively studied, using techniques from
pseudo-holomorphic curves and Floer theory \cite{seidelknotted,seidel:exactseq,seidel,keating-free,keating-tori,rationalI,rationalII,positiverational}, mirror symmetry \cite{seidel-thomas,khovanov-seidel,sheridan-smith,hacking-keating}, and Seiberg--Witten gauge theory \cite{Smirnov}. Our tools will be based on the latter, building up on previous work by Kronheimer \cite{kronheimer}, Smirnov \cite{Smirnov} and J. Lin \cite{JLIN2022}.

\subsection{$ADE$ configurations and Braid group representations}

Let $\Gamma$ be one of the simply-laced Dynkin diagrams: either $A_n$ ($n\geq 1$), $D_n$ ($n \geq 4$), $E_6$, $E_7$ or $E_8$. An \textit{ADE configuration} of Lagrangian spheres of type $\Gamma$ in a closed symplectic $4$-manifold $(X, \omega )$ is a collection of Lagrangian spheres $L_i$ in $(X, \omega )$, one for each vertex $i$ of $\Gamma$, such that: if $i$ and $j$ are connected by an edge in $\Gamma$ then $L_i$ and $L_j$ intersect only once and transversely; otherwise $L_i$ and $L_j$ are disjoint (unless $i = j$). Typically, $ADE$ configurations of Lagrangian spheres arise as vanishing cycles of smoothings of algebraic surfaces with $ADE$ singularities
; they are particularly ubiquitous in the case of $K3$ surfaces, since $K3$ surfaces with $ADE$ singularities always admit a smoothing \cite{burns-wahl}.

Given an $ADE$ configuration of Lagrangian spheres, one obtains more Lagrangian spheres by taking their images under Dehn--Seidel twists on spheres in the same collection. 
The `obvious' relations among the Dehn--Seidel twists on all these spheres can be organised into a \textit{symplectic representation of the generalised Braid group} $B = B(\Gamma )$. Here $B $ is defined as the Artin group associated with the graph $\Gamma$: it has the presentation with one generator $s_i$ for each vertex $i$; and relations $s_i s_j s_i  = s_j s_i s_i $ if $i$ and $j$ are connected by an edge, and $s_i s_j = s_j s_i$ otherwise. (When $\Gamma$ is the $A_n$ graph, then $B\cong B_{n+1}$ is the classical Braid group on $n+1$ strands). If two Lagrangian spheres $L_i, L_j$ are disjoint then their Dehn--Seidel twists commute in $\pi_0 \mathrm{Symp}(X,\omega )$: $\tau_{L_1}\tau_{L_2} = \tau_{L_1}\tau_{L_2}$; in turn, if $L_i, L_j$ intersect only once and transversely, then their Dehn--Seidel twists satisfy the Braid relation $\tau_{L_i}\tau_{L_j}\tau_{L_i} = \tau_{L_j}\tau_{L_i}\tau_{L_j}$ \cite{seidelknotted,seidel-thomas}. Thus, 
this yields an associated representation of $B$,
\begin{align}
\rho : B \to \pi_0 \mathrm{Symp}(X,\omega )\quad , \quad s_i \mapsto \tau_{L_i}.\label{rho_intro} 
\end{align}


\begin{question}\label{question:faithful}
Given an $ADE$ configuration of Lagrangian spheres in $(X, \omega )$, is the associated representation (\ref{rho_intro}) faithful (i.e. injective)?
\end{question}

Question \ref{question:faithful} has been answered affirmatively when $(X, \omega )$ is an $A_n$ Milnor fiber \cite{khovanov-seidel}; for types $D$ and $E$ there has been significant progress \cite{brav-hugh,qiu-woolf}. However when $(X, \omega )$ is closed then Question \ref{question:faithful} has proved far more challenging. For certain symplectic $K3$ surfaces and certain $A_n$ configurations in them (with $n \leq 14$), Seidel \cite{seidel,khovanov-seidel} has answered Question  \ref{question:faithful} affimatively. 

Our first result (Theorem \ref{theorem:ADEK3}) answers affirmatively an `abelianised' version of Question \ref{question:faithful} for symplectic $K3$ surfaces. To describe this result, recall that the Weyl group $W = W(\Gamma )$ is the (finite) group defined like $B$ but with the additional relations $s_{i}^2 = 1$. (In type $A_n$, $W\cong S_{n+1}$ is the symmetric group on $n+1$ elements). The \textit{generalised pure Braid group} $P = P(\Gamma ) $ is defined as the kernel of the canonical surjective homomorphism $B \to W$. Since $\dim X = 4$ then squared Dehn--Seidel twists $\tau_{L}^2$ are smoothly trivial, and from this one easily sees that the restriction of (\ref{rho_intro}) to $P$ is valued in the smoothly-trivial symplectic mapping class group:
\begin{align}
\rho_0 := \rho|_P : P \to \pi_0 \mathrm{Symp}_0 (X, \omega )  .\label{rho0_intro}
\end{align}

We shall study the injectivity of (\ref{rho0_intro}) after abelianising. Here, we note that whereas the abelianisation of $B$ is uninteresting ($B^\mathrm{ab} \cong \mathbb{Z}$), the abelianisation of $P$ is large (classical results by Brieskorn \cite{brieskorn-braid} and Deligne \cite{deligne} imply that $P^{\mathrm{ab}}$ is the free abelian group on the positive roots in the root system with Dynkin diagram $\Gamma$). An additional interesting aspect here is the fact that, by definition of $P$, its abelianisation $P^{\mathrm{ab}}$ carries an action of the Weyl group $W$ by conjugation, which makes $P^{\mathrm{ab}}$ a $W$--\textit{representation}. 
(In type $A_n$, $P^{\mathrm{ab}}$ is the free abelian group of rank $n+1\choose2$; as a $W = S_{n+1}$--representation the action is induced by the permutation action of $S_{n+1}$ on the set of $2$-element subsets of $\{1, \ldots , n+1\}$, see \S \ref{subsubsection:exampleA}).
Similarly, in $\dim X = 4$ the abelianised group $(\pi_0 \mathrm{Symp}_0 (X, \omega ))^{\mathrm{ab}}$ is a $W$--representation as well: $s_i \in W$ acts as conjugation by the Dehn--Seidel twist $\tau_{L_i}$. The abelianisation of (\ref{rho0_intro}) is a $W$--equivariant homomorphism. (See \S \ref{subsection:ADEconfs} for details). Then:

\begin{thm}\label{theorem:ADEK3}
Let $(X, \omega ) $ be a symplectic $K3$ surface with an $ADE$ configuration of Lagrangian spheres. Then the associated homomorphism of $W$--representations
\[
\rho_{0}^{\mathrm{ab}} : P^{\mathrm{ab}}\to (\pi_0 \mathrm{Symp}_0 (X, \omega ))^{\mathrm{ab}}
\]
is $W$--equivariantly split-injective.
\end{thm}

That is, there is an isomorphism $(\pi_0 \mathrm{Symp}_0 (X, \omega ))^{\mathrm{ab}}\cong P^{\mathrm{ab}} \oplus \cdots$ of $W$--representations,
 not only of abelian groups. (For context, for any field $k$ the homomorphism $P^\mathrm{ab} \otimes k \to (\pi_0 \mathrm{Symp}_0 (X, \omega ))^{\mathrm{ab}}\otimes k$ splits provided it is injective; but an equivariant splitting is only guaranteed if in addition $char (k)$ does not divide $\# W$, by Maschke's Theorem). 
We note that this stronger conclusion would not follow from an affirmative answer to the original Question \ref{question:faithful}; it would follow---for example---if the homomorphism (\ref{rho_intro}) was \textit{split}-injective.

The results of Seidel \cite{seidelknotted,seidel} and Tonkonog \cite{tonkonog} show that Dehn–Seidel twists often yield non-trivial elements in the symplectic mapping class groups of projective hypersurfaces. Already for an $A_1$ configuration, our Theorem \ref{theorem:ADEK3} provides a new result: \textit{for any Lagrangian sphere $L$ in a symplectic $K3$ surface $(X, \omega)$, the Dehn--Seidel twist $\tau_L$ has infinite-order in $\pi_0 \mathrm{Symp}(X, \omega )$}.


\subsection{Homologically-distinct Lagrangian spheres}

Beyong the $ADE$ case, 
we study arbitrary configurations of Lagrangian spheres which are pairwise \textit{homologically-distinct}.
For a closed symplectic $4$-manifold $(X, \omega)$, we let $\mathcal{L}(X,\omega ) $ denote the subset of $H_2 (X, \mathbb{Z} )$ given by fundamental classes of oriented Lagrangian spheres in $(X, \omega )$, and let $\mathcal{L}(X,\omega)/\pm$ denote its quotient by multiplication by $-1$.

\begin{thm}\label{theorem:summandK3}
Let $(X, \omega ) $ be a symplectic $K3$ surface. For each $k \in \mathcal{L}(X,\omega)/\pm$ choose an (unoriented) Lagrangian sphere $L_{k}$ in $(X, \omega )$ whose fundamental class is $k$. Then the homomorphism 
\[
\bigoplus_{k \in \mathcal{L}(X,\omega)/\pm}\mathbb{Z} \to (\pi_0 \mathrm{Symp}_0 (X, \omega ) )^{\mathrm{ab}} 
\]
mapping each generator to the corresponding squared Dehn--Seidel twist $\tau_{L_k}^2$ is split-injective.
\end{thm}

Theorem \ref{theorem:summandK3} says---in particular---that squared Dehn--twists on homologically-distinct Lagrangian spheres are algebraically independent in the abelianisation of $\pi_0 \mathrm{Symp}_0 (X, \omega )$. A somewhat related result involving two Lagrangian spheres was proved in \cite{keating-free}.

Theorem \ref{theorem:summandK3} provides a `purely symplectic' counterpart of the main result of Smirnov \cite{Smirnov}, which involves instead the symplectic monodromy of the moduli space of marked polarized $K3$ surfaces. Whereas Smirnov's argument relies on Torelli Theorems for $K3$ surfaces---whence $\omega$ is assumed to be Kähler---our proof will rely instead on various gluing results in Seiberg--Witten theory \cite{mrowka-rollin,JLIN2022} which allow us to avoid the Kähler hypothesis and deal with arbitrary Lagrangian spheres.

An immediate consequence of Theorem \ref{theorem:summandK3} is the following infinite-generation result, which extends the ones previously known \cite{sheridan-smith,Smirnov,Smirnov-T4}:

\begin{corollary}\label{corollary:infgen}
Let $(X, \omega ) $ be a symplectic $K3$ surface. If the set $\mathcal{L}(X, \omega )$ is infinite, then $\pi_0 \mathrm{Symp}_0 (X, \omega )$ is infinitely-generated.
\end{corollary}

A typical example of a closed symplectic $4$-manifold $(X, \omega )$ with infinite $\mathcal{L}(X,\omega )$ is an $(X, \omega )$ which admits a symplectic embedding of a Milnor fiber of a \textit{non-$ADE$} singularity. Indeed, by \cite{gabrielov}, \cite[Theorem 10.1]{durfee-15} and \cite[Corollary 2.17]{keating-tori}, in such an $(X, \omega )$ one can find two homologically-distinct oriented Lagrangian spheres $L_1, L_2$ such that $L_1 \cdot L_2 = - 2$. From this 
one easily sees that the Lagrangian spheres in the following infinite sequence are all pairwise homologically-distinct: $L_1 \, , \, \tau_{L_2}(L_1) \, , \, \tau_{L_1}\tau_{L_2}(L_1) \, , \, \tau_{L_2}\tau_{L_1}\tau_{L_2}(L_1) \, , \, \text{etc}$.\\

We note at this point that Theorems \ref{theorem:ADEK3}-\ref{theorem:summandK3} and Corollary \ref{corollary:infgen} hold---in principle---for any closed symplectic $4$-manifold $(X, \omega )$ with canonical class $K = 0$ and $b^+ (X) \geq2$. However, among the known examples of these (\cite{geiges-torus,CYsurfaces}) the only ones that may contain Lagrangian spheres are the symplectic $K3$ surfaces.




\subsection{The fundamental group of the space of symplectic forms}

It is an interesting question whether the conclusion of Theorems \ref{theorem:ADEK3}-\ref{theorem:summandK3} holds in more generality, beyond the symplectic $K3$ surfaces. We now describe analogues of those results without the $K3$ assumption (Theorems \ref{theorem:summand}-\ref{theorem:ADE}). Rather than the smoothly-trivial symplectic mapping class group, these shall involve the fundamental group of the space of symplectic forms.

Let $\mathcal{S}_0 (X, \omega )$ denote space of all symplectic forms on $X$ which are \textit{isotopic} to $\omega$. Equivalently, this is the connected component of $\omega$ in the space of symplectic forms \textit{cohomologous} to $\omega$, by Moser's Theorem. Also by this result---in its parametric version---there is a fibration 
\[
\mathrm{Symp}_0 (X, \omega )\to \mathrm{Diff}_0 (X) \xrightarrow{p} \mathcal{S}_0 (X, \omega )
\]
where the total space is the identity component in the diffeomorphism group of $X$, and the projection map $p$ sends $f \mapsto f^\ast \omega$. Since $\dim X = 4$, then the squared Dehn--Seidel twist $\tau_L^2$ on a Lagrangian sphere $L$ is smoothly trivial: thus $\tau_{L}^2$ lies in the image of the connecting homomorphism $\delta : \pi_1 \mathcal{S}_0 (X,\omega ) \to \pi_0 \mathrm{Symp}_0 (X, \omega )$. In \S \ref{section:symplectic} we will show that $\tau_{L}^2$ has a \textit{canonical lift} along $\delta$, denoted $\mathcal{O}_L \in \pi_1 \mathcal{S}_0 (X, \omega )$, but this depends on a choice of \textit{orientation} of $L$ (in contrast, $\tau_{L}$ does not depend on a choice of orientation of $L$). Furthermore, we shall prove in Theorem \ref{theorem:generaliseddehntwist} that:
\begin{itemize} 
\item the canonical lifts associated to the two orientations of $L$, denoted $\mathcal{O}_{L}$ and $\mathcal{O}_{-L} \in \pi_1 \mathcal{S}_0 (X,\omega )$, are \textit{commuting} elements in $\pi_1 \mathcal{S}_0 (X, \omega )$
\item the difference $\mathcal{O}_L \cdot \mathcal{O}_{-L}^{-1} = \mathcal{O}_{-L}^{-1}\cdot \mathcal{O}_L \in \pi_1 \mathcal{S}_0 (X, \omega )$ only depends on the \textit{smooth} isotopy class of the oriented sphere $L$ ; in fact, it agrees with $p_\ast ( \gamma_L )$ where $\gamma_L \in \pi_1 \mathrm{Diff}_0 (X )$ is the loop of diffeomorphisms given by a `generalised Dehn twist' on the smoothly embedded oriented $(-2)$--sphere $L$, which was recently studied by J.Lin \cite{JLIN2022}.
\end{itemize}

The next result is an analogue of Theorem \ref{theorem:summandK3}, with the role of squared Dehn--Seidel twists now played by their canonical lifts:

\begin{thm}\label{theorem:summand}
Let $(X, \omega ) $ be a closed symplectic $4$-manifold. For each $e \in \mathcal{L}(X,\omega)$ choose an oriented Lagrangian sphere $L_{e}$ in $(X, \omega )$ whose fundamental class is $e$. Then the homomorphism 
\[
\bigoplus_{e\in \mathcal{L}(X,\omega)}\mathbb{Z} \to (\pi_1 \mathcal{S}_0 (X, \omega ) )^{\mathrm{ab}} 
\]
mapping each generator to the corresponding loop of symplectic forms $\mathcal{O}_{L_e}$ is split-injective.
\end{thm}

We will see that that the summand $\bigoplus_{e\in \mathcal{L}(X,\omega)}\mathbb{Z} \hookrightarrow  (\pi_1 \mathcal{S}_0 (X, \omega )$ detected in Theorem \ref{theorem:summand} splits up into two summands, each $\approx \bigoplus_{k \in \mathcal{L}(X,\omega )/\pm} \mathbb{Z}$. One of these two summands is contained in the subgroup of `non-symplectic loops of diffeomorphisms' 
\[
\pi_1 \mathrm{Symp}_0 (X, \omega ) \backslash \pi_1 \mathrm{Diff}_0 (X, \omega )\hookrightarrow \pi_1 \mathcal{S}_0 (X, \omega )
\]
and belongs in a larger free summand in this subgroup that was detected by J. Lin \cite[Theorem E]{JLIN2022} ; the other summand split-injects to $(\pi_0 \mathrm{Symp}_0 (X, \omega ))^\mathrm{ab}$ when $(X,\omega)$ is a symplectic $K3$ surface, inducing the summand detected in Theorem \ref{theorem:summandK3}. (See \S \ref{section:infinite}.)\\




Returning to an $ADE$ configuration in $(X, \omega )$ of Lagrangian spheres of type $\Gamma$, we now equip the $L_i$'s with orientations such that $L_i \cdot L_j \geq 0$ for $i \neq j$.
We shall see that the analogue of the abelianised homomorphism (\ref{rho_0}) is given by certain $W$--equivariant homomorphism $\rho_0 : \mathbb{Z}\Phi \to (\pi_1 \mathcal{S}_0 (X, \omega ))^\mathrm{ab}$ (see Proposition \ref{proposition:mapZPhi}). Here, $\mathbb{Z}\Phi$ is the free abelian group on the set $\Phi$ of roots in the root system of type $\Gamma$, with the $W$--representation structure induced by the permutation $W$--action on $\Phi$. Labeling the simple roots in $\Phi$ by the vertices $i \in \Gamma$, the homomorphism $\rho_0$ sends the generator corresponding to the $i$st simple root to $\mathcal{O}_{L_i}$. An analogue of Theorem \ref{theorem:ADEK3} is given by:

\begin{thm}\label{theorem:ADE}
Let $(X, \omega ) $ be a closed symplectic $4$-manifold with an $ADE$ configuration of Lagrangian spheres. Then the associated homomorphism of $W$--representations 
\[
\rho_0 : \mathbb{Z}\Phi \to (\pi_1 \mathcal{S}_0 (X, \omega ))^{\mathrm{ab}}
\]
is $W$--equivariantly split-injective.
\end{thm}


Finally, we expect the homomorphism $\rho_0 : \mathbb{Z}\Phi \to (\pi_1 \mathcal{S}_0 (X, \omega ))^{\mathrm{ab}}$ to have a natural `non-abelian' analogue, with a corresponding variant of Question \ref{question:faithful}.

\subsection{Outline and comments}

\S\ref{section:symplectic} discusses various topological aspects of Dehn--Seidel twists on Lagrangian spheres which are special to the four dimensional case. We introduce the canonical lift $\mathcal{O}_L$ of the squared Dehn--Seidel twist $\tau_{L}^2$ (associated to a choice of orientation of $L$), which for us will be the more convenient object to work with. 
We then establish various fundamental properties of the canonical lift, most notably Theorem \ref{theorem:generaliseddehntwist} which relates Dehn--Seidel twists to the `generalised Dehn twist' loop of diffeomorphisms studied in \cite{JLIN2022}.

\S\ref{section:kronheimer} discusses the main technical tool in this article: Kronheimer's invariant of loops of symplectic forms on a smooth $4$-manifold \cite{kronheimer} based on $2$-parametric Seiberg--Witten gauge theory, and its subsequent treatment by Smirnov \cite{Smirnov}. This section contains a mix of background and new material. We emphasise here a viewpoint in which Kronheimer's invariant can be regarded as a kind of `linking number' with certain (infinite-dimensional) codimension $2$ discriminant loci in the space of perturbations of the Seiberg--Witten equations. We discuss an \textit{excision property} of Kronheimer's invariant (Proposition \ref{proposition:excision}) which is a consequence of the gluing theory developed by Mrowka--Rollin \cite{mrowka-rollin} ; this provides a powerful tool when calculating Kronheimer's invariant on loops of symplectic forms given by local symmetries (such as the loop $\mathcal{O}_L$).

\S\ref{section:ADE} is devoted to the proof of Theorems \ref{theorem:ADEK3}-\ref{theorem:ADE}. When $(X,\omega )$ is a symplectic $K3$ surface we will assemble together Kronheimer's invariants into a homomorphism 
\[
q : (\pi_0 \mathrm{Symp}_0 (X, \omega ))^\mathrm{ab} \to P^{\mathrm{ab}}.
\]
The homomorphism $q$ is $W$--equivariant, and we shall exploit this, the excision property, and Kronheimer's calculation \cite{kronheimer} to establish that $q$ is a left-inverse of $\rho_{0}^\mathrm{ab} :P^\mathrm{ab}\to (\pi_0 \mathrm{Symp}_0 (X, \omega ))^\mathrm{ab}$, i.e. $q \circ\rho_{0}^\mathrm{ab} = \mathrm{Id}$ (Theorem \ref{theorem:splitting1}). 

Here we will crucially make use of the description of the generalised pure Braid group $P$ as the fundamental group of the complement of a complex hyperplane arrangement, due to Brieskorn \cite{brieskorn-braid} and Deligne \cite{deligne}. From this point of view, the abelianisation map $P \to P^\mathrm{ab}$ is given by taking linking numbers with the complex hyperplanes, which is reminiscent of Kronheimer's invariant. This analogy suggest a deeper connection between Seiberg--Witten gauge theory and the symplectic representation $\rho_0  : P \to \pi_0 \mathrm{Symp}_0 (X, \omega )$ before abelianising: in particular, we pose the following
\begin{question}
Let $(X, \omega )$ be a symplectic $K3$ surface with an $ADE$ configuration of Lagrangian spheres. Does the homomorphism $q : (\pi_0 \mathrm{Symp}_0 (X, \omega ))^\mathrm{ab} \to P^\mathrm{ab}$ lift to a homomorphism $ \mathbf{q} : \pi_0 \mathrm{Symp}_0 (X, \omega ) \to P$ which gives a left-inverse to the symplectic representation $\rho_0 : P \to \pi_0 \mathrm{Symp}_0 (X, \omega )$ (i.e. $\mathbf{q}\circ \rho_0 = \mathrm{Id}$) ?
\end{question}

\S\ref{section:infinite} is devoted to the proof of Theorems \ref{theorem:summandK3}-\ref{theorem:summand}. To this end, we will bring in an additional gluing result in Seiberg--Witten theory: the `Family Switching Formula' (\cite{liu_switching}, \cite[Theorem M]{JLIN2022}), which will be used in the proof of the splitting property in general (however, this is not needed to prove Corollary \ref{corollary:infgen}). 

Finally, we note that Theorem \ref{theorem:generaliseddehntwist} allows to translate various computations of Kronheimer's invariant in terms of the computations of Family Seiberg--Witten invariants from \cite{JLIN2022}, and vice-versa. In particular, a number of results in this article may be deduced using either technique.


\subsection{Acknowledgements} The author thanks Simon Donaldson, Peter Kronheimer, Radu Laza, Mark McLean and Paul Seidel for conversations from which this work has benefitted.

\section{Dehn--Seidel twists and their canonical lifts}\label{section:symplectic}


On a symplectic $4$-manifold $(X, \omega )$ with an embedded Lagrangian sphere $L \cong S^2$ there is a naturally associated \textit{Dehn--Seidel twist symplectomorphism} of $(X, \omega )$ \cite{seidel,seidelknotted}, denoted $\tau_L$, whose support can be arranged to be in an arbitrarily small neighborhood of $L$. This symplectomorphism is also defined in higher dimensions; this section discusses various aspects of $\tau_L$ which are special to the four-dimensional case. The main result in this section is Theorem \ref{theorem:generaliseddehntwist}.

\subsection{Notation and conventions}\label{notation}

Throughout the article, $(X, \omega)$ denotes a compact, connected, symplectic $4$-manifold with convex contact boundary $(Y, \xi )$ (possibly $Y = \emptyset$). The latter means that $Y = \partial X $ is an oriented $3$-manifold equipped with a co-oriented and positive contact structure $\xi$ for which there exists a vector field $V$ defined on a neighborhood of the boundary of $X$ such that: $V$ is Liouville for $\omega$ (i.e. $\mathcal{L}_V \omega = \omega$), $V$ is outward pointing along $Y$, and the contact form $\alpha $ on $Y$ obtained by restriction of $\iota_V \omega$ onto $Y$ has $\xi = \mathrm{ker} \alpha$.

Whenever a Lagrangian submanifold $L$ of $(X, \omega )$ is considered it will always be implicitly assumed that \textit{$L$ is disjoint from the boundary of $X$}.

We lay out below our notation for the various spaces of geometric objects associated to $(X,\omega )$ that we shall encounter (each of them is topologised with the $C^\infty$ topology, and is equipped with an `obvious' basepoint):
\begin{itemize}
\item $\mathrm{Diff}(X)$ is the group of orientation-preserving diffeomorphisms of $X$, which equal the identity near $\partial X$.
\item $\mathrm{Diff}_0 (X)$ is the connected component of the identity in $\mathrm{Diff}(X)$.
\item $\mathrm{Symp}(X, \omega )$ the group of symplectomorphisms of $(X, \omega )$ which equal the identity near $\partial X$.
\item $\mathrm{Symp}_0 (X, \omega ) $ is the subgroup of $\mathrm{Symp} (X, \omega )$ consisting of  \textit{smoothly trivial }symplectomorphisms. That is,
\[
\mathrm{Symp}_0 (X, \omega ) = \mathrm{ker} \Big(\mathrm{Symp} (X, \omega ) \to \pi_0 \mathrm{Diff}(X) \Big).
\]
\item $\mathcal{S}_{0 } (X, \omega )$ is the connected component of $\omega$ in the space of symplectic structures on $X$ in the cohomology class $[\omega]$ and which equal $\omega$ near $\partial X$. (Equivalently, by Moser's argument, $\omega^\prime \in \mathcal{S}_{0}(X, \omega )$ if and only if there exists a path $f_t \in \mathrm{Diff}(X)$ for $0 \leq t \leq 1$ such that $f_0 = \mathrm{Id}$ and $f_{1}^\ast \omega = \omega^\prime$). $\mathcal{S} (X, \omega )$ 
is the connected component of $\omega$ in the space of symplectic structures on $X$ which agree with $\omega$ near $\partial X$.

\item By the parametric version of Moser's argument, assigning to a diffeomorphism $f$ the symplectic form $f^\ast \omega$ gives a \textit{fibration} relating some of the spaces above:
\begin{align}
\mathrm{Symp}_0 (X, \omega ) \to \mathrm{Diff}_0 (X, \omega ) \xrightarrow{p} \mathcal{S}_{0}(X, \omega ) .\label{fibration}
\end{align}

\item Recall that a Riemannian metric $g$ on $X$ is said to be \textit{compatible} with the symplectic form $\omega$ if $\omega$ is self-dual and has length $\sqrt{2}$ with respect to $g$; such a pair $(\omega , g )$ is referred to as an \textit{almost-Kähler structure} on $X$. (We recall that a compatible metric $g$ gives an almost-complex structure $J$ of the tangent bundle of $X$ defined by $\omega = g(J \cdot , \cdot )$; the assignment $g \to J$ sets up a correspondence between Riemannian metrics on $X$ compatible with $\omega$ and almost-complex structures on $X$ compatible with $\omega$). We consider the following spaces of almost-Kähler structures on $X$. When $X$ has non-empty boundary, we assume a Riemannian metric $g$ on $X$ is fixed near the boundary, and $g$ is compatible with $\omega$ on its domain of definition. We define $\mathcal{AK}_0 (X, \omega )$ (resp. $\mathcal{AK}(X, \omega )$) as the space of almost-Kähler structures $(\omega^\prime , g^\prime )$ where $\omega^\prime \in \mathcal{S}_0 (X, \omega )$ (resp. $\omega^\prime \in \mathcal{S} (X, \omega )$) such that $g^\prime$ agrees with $g$ near the boundary of $X$. It is a fundamental fact in symplectic topology that the forgetful map $(\omega^\prime , J^\prime ) \mapsto \omega^\prime$ yields homotopy equivalences:
\[
\mathcal{AK}_0 (X, \omega )\simeq \mathcal{S}_0 (X, \omega ) \quad , \quad \mathcal{AK} (X, \omega )\simeq \mathcal{S} (X, \omega ).
\]

\end{itemize}

\subsection{The Dehn--Seidel twist}

Let $L $ be an embedded Lagrangian $2$-sphere in $(X, \omega )$. We now recall the construction of the Dehn--Seidel twist symplectomorphism $\tau_L$ along $L \subset X$, following \cite{seidel,seidelknotted}.

We first consider the model case when $X = T^\ast S^2$, $\omega$ is given by the canonical symplectic form, and $L = S^2$ is the zero section. Explicitly:
\[
T^\ast S^2 = \Big\{ (u,v) \in \mathbb{R}^3 \times \mathbb{R}^3 | \, \, | \, \, \langle u , v \rangle = 0 \, , \, \| v \| =1 \Big\} \quad , \quad \omega = du \wedge dv \quad , \quad S^2 = \{ u = 0\} \subset T^\ast S^2 .
\]

The natural action of $SO(3)$ on $S^2$ lifts to a Hamiltonian $SO(3)$--action on $(T^\ast S^2, \omega )$, denoted $\rho : SO(3) \times T^\ast S^2 \to T^\ast S^2$, with moment map\footnote{Our convention for a Hamiltonian $G$--action on a symplectic manifold $(M,\omega)$ is that $\iota_{X_\xi }\omega = - d \mu_\xi$ for all $\xi \in \mathfrak{g} = \mathrm{Lie}G$, where $\mu : M \to \mathfrak{g}^\ast$ is the moment map and $X_\xi$ is the vector field associated to the linearisation of the $G$--action at $\xi$.}
\[
\mu (u,v) = - u \times v \in \mathbb{R}^3 \cong \mathfrak{so}(3)^\ast .
\]
The Hamiltonian flow of the function $ \| u \|$ induces an $S^1$--action $\sigma$ on $T^\ast S^2 \setminus S^2$, and this agrees with the normalised geodesic flow on $T^\ast S^2 \setminus S^2$. More generally, we will use the following result, which we include here for future reference:

\begin{lemma}[Lemma 1.2 in \cite{seidel}] \label{lemma:circle} Let $(M,\omega )$ be a symplectic manifold, carrying a Hamiltonian $SO(3)$--action $\rho$ with moment map $\mu$. Then $ \| \mu \|$ is the Hamiltonian of an $S^1$--action $\sigma$ on $M \setminus \mu^{-1}(0)$.
\end{lemma}

The time $t = \pi$ Hamiltonian flow $\sigma_\pi$ of $\| u \|$ is the symplectomorphism $(u,v) \to (-u , -v )$ induced by the antipodal map, which therefore extends smoothly over to the zero section $S^2 \subset T^\ast S^2$. This map $\sigma_\pi$ is then cut off outside the compact neighborhood $D_\varepsilon (T^\ast S^2 )$, for a given $\varepsilon >0$, as follows. Fix a smooth function $r: \mathbb{R}_{\geq 0} \to \mathbb{R}_{\geq 0}$ such that $r(x) =x$ on an open neighborhood of $x = 0$ and $r(x) = 0$ on an open neighborhood of $[\varepsilon , \infty )$. Then set $h (u,v) = r ( \| u \| )$. The time $t$ Hamiltonian flow of $h$ is given by $(u,v) \mapsto \sigma_{t r^\prime (\| u \| )  }(u,v) $, which is a symplectomorphism of $(D_{\varepsilon} (T^\ast S^2  ) \setminus S^2 , \omega )$ that agrees with the identity on a neighborhood of $\partial D_{\varepsilon}(T^\ast S^2 )$, and at time $t = \pi$ still extends smoothly over to the zero section as the antipodal map $(0,v) \to (0,-v)$.

\begin{definition}[\cite{seidelknotted,seidel}] \label{definition:dehn}
The \textit{model Dehn--Seidel twist} on $(D_{\varepsilon }  (T^\ast S^2 ) , \omega )$ is the symplectomorphism $\tau (u,v) = \sigma_{\pi r^\prime (\| u \| )} (u,v) \in \pi_0 \mathrm{Symp}(D_\varepsilon (T^\ast S^2) , \omega ) $.\\
\end{definition}

More generally, let $L\subset X$ be an embedded Lagrangian sphere in a compact symplectic $4$-manifold. Choosing a diffeomorphism $L \cong S^2$, the Weinstein Neighborhood Theorem provides a canonical (up to homotopy) symplectomorphism of a neighborhood of $L \subset X$ with a neighborhood of the zero section in the model space $(T^\ast S^2 , \omega )$. The \textit{Dehn--Seidel twist} along $L \subset X$ is defined as the symplectomorphism $\tau_L \in \pi_0 \mathrm{Symp}(X, \omega )$ obtained by implanting the model Dehn--Seidel twist using these identifications.

As observed in \cite{seidel}, the element $\tau_L$ thus constructed is independent of the choice of diffeomorphism $L \cong S^2$ (because by Smale's Theorem \cite{smale} we have $\mathrm{Diff} (S^2) \simeq O(3)$ and the model Dehn--Seidel twist $\tau$ is $O(3)$--equivariant by construction). In particular, \textit{the Dehn--Seidel twist is independent of how $L$ is oriented}:
\begin{align}
\tau_{L}= \tau_{-L}.\label{symmetry1}
\end{align}

On the other hand, note that if $L$ is an embedded Lagrangian sphere in the symplectic $4$-manifold $(X, \omega )$ then $L$ can also be regarded as an embedded Lagrangian sphere in the \textit{conjugate} symplectic $4$-manifold $(X, -\omega )$ (note also that the orientations of $X$ induced by $\pm \omega$ are the same, because $X$ is $4$-dimensional). Thus, we can construct the Dehn--Seidel twist using either $\pm \omega$, and both $\tau_{L}^{\pm \omega }$ can be regarded as symplectomorphisms of the same manifold $(X, \omega )$. The following Lemma clarifies their relationship:

\begin{lemma}\label{lemma:symmetries}
Let $L$ be an embedded Lagrangian sphere in a symplectic $4$-manifold $(X, \omega )$. Then
$\tau_{L}^{\omega} = (\tau_{L}^{-\omega })^{-1}$ in $\pi_0 \mathrm{Symp} (X, \omega )$ (i.e. conjugating the symplectic form inverts the Dehn--Seidel twist).
\end{lemma}
\begin{proof}
This follows from the observation that the model space $(T^\ast S^2 , \omega = du \wedge dv )$ carries an \textit{anti}-symplectomorphism $c: (u,v) \mapsto (-u,v)$ (i.e. $c^\ast \omega = - \omega$) which \textit{reverses} the geodesic flow.  \end{proof}

\subsection{The canonical lift of the Dehn--Seidel twist}

The squared Dehn--Seidel twist $\tau_{L}^2 $ is smoothly isotopic to the identity in dimension $4$ \cite{seidelknotted,seidel} (and also in dimension $12$ \cite{RW-Keating}). Hence $\tau_{L}^2$ lies in the image of the connecting homomorphism of (\ref{fibration}):
\[
\delta: \pi_1 \mathcal{S}_{0}(X,\omega ) \to \pi_0 \mathrm{Symp}_0 (X, \omega ).
\]

We shall now equip $L$ with an \textit{orientation} (even if $\tau_L$ does not depend on the orientation of $L$, cf. Lemma \ref{lemma:symmetries}(1)) and using this we will construct a \textit{canonical} element $\mathcal{O}_L \in \pi_1 \mathcal{S}_{0} (X, \omega )$ with $\delta ( \mathcal{O}_L ) = \tau_{L}^2$. Further properties of this canonical element are described below.

\subsubsection{Construction of $\mathcal{O}_L$}


We first consider this element in the model situation $(X, \omega ) = (T^\ast S^2  ,  du \wedge dv )$, $L = S^2$, and discuss the general case later.\\

We fix a closed compactly-supported $2$-form $\widetilde{\beta}$ on $T^\ast S^2$ such that 
\[
\int_{S^2} \widetilde{\beta} = 4\pi .
\]
Thus, $[\beta]$ is a generator of $H_{c}^2 (T^\ast S^2 ,\mathbb{R}) \cong H^2 (S^2 , \mathbb{R}) \cong \mathbb{R}$. One can then turn $\widetilde{\beta}$ into an $SO(3)$--invariant $2$-form $\beta$, by averaging: fix a bi-invariant Riemannian metric on $SO(3)$, with volume form $d\mu$ and total volume $1$, and set
\[
\beta (X,Y) = \int_{ SO(3)}( g^\ast \widetilde{\beta})(X,Y) \cdot d\mu (g).
\]
The $SO(3)$--invariant $2$-form $\beta$ is still a closed and compactly-supported $2$-form on $T^\ast S^2$. In addition, since $S^2 \subset T^\ast S^2$ is $SO(3)$--invariant, we still have
\[
\int_{S^2} \beta = \int_{SO(3)} (\int_{S^2} g^\ast \widetilde{\beta}) \cdot d\mu (g) = 4\pi
\]
and furthermore, denoting by $\iota_0 : S^2 \hookrightarrow T^\ast S^2$ the inclusion of the zero section, the pullback form $\iota_{0}^\ast \beta$ agrees with the standard symplectic form $\beta_{S^2}$ on $S^2$ (i.e. the unique $SO(3)$--invariant symplectic form on $S^2$ with area $4 \pi$). Let $R>0$ be a fixed constant chosen so that the support of $\beta$ is contained in the interior of $D_{R} (T^\ast S^2 )$. For $c>0$, let $\ell_{c}$ denote the fiber-rescaling diffeomorphism $(u,v) \mapsto (cu , v)$.

We consider the following deformation of the symplectic structure $\omega$ on $T^\ast S^2$, given by a path of $SO(3)$--invariant symplectic forms $\omega^s$ on $D_\varepsilon (T^\ast S^2 ) $ which agree with $\omega$ near the boundary:
\begin{align}
\omega^s = \omega + s   \ell_{R/\varepsilon}^\ast \beta \, \in \mathcal{S}(D_\varepsilon (T^\ast S^2 ) , \omega )   \quad , \quad 0 \leq s \ll 1. \label{modeldeformation}
\end{align}
Observe that when $s>0$ the zero section $S^2 \subset (D_{\varepsilon} (T^\ast S^2 ) , \omega^s )$ ceases to be a Lagrangian submanifold; instead, it becomes a \textit{symplectic submanifold} of area $4 \pi s$. (In particular, the path $\omega^s$ cannot be realised by an ambient isotopy, since the cohomology class of $\omega^s$ changes by $[\omega^s ] = 4\pi s PD ([S^2 ])$).


We now show that the $\omega$--Hamiltonian $S^1$--action on $D_\varepsilon ( T^\ast S^2 )\setminus S^2 $ can be deformed to an $\omega^s$--Hamiltonian $S^1$--action which, in addition, extends over to $S^2$ for $s>0$. For this, we modify an argument due to Seidel \cite[Proof of Proposition 1.1]{seidel}, where he instead considered the path $\omega + s \pi^\ast \beta_{S^2}$, where $\pi: T^\ast S^2 \to S^2$ is the bundle projection. (Evidently, this path does not stay constant on the boundary of $D_\varepsilon (T^\ast S^2 )$, which is why we use instead the path (\ref{modeldeformation}).)


For given $\xi \in \mathbb{R}^3 \cong \mathfrak{so}(3)^\ast$, let $X_\xi$ denote the vector field on $T^\ast S^2$ obtained by linearising the $SO(3)$--action on $T^\ast S^2$. Since $\ell_{R/\varepsilon}^\ast \beta$ is closed and $SO(3)$--invariant, the contraction $\iota_{X_\xi} (\ell_{R/\varepsilon}^\ast \beta )$ is a closed compactly-supported $1$-form on $T^\ast S^2$; and because $H^{1}_c (T^\ast S^2 , \mathbb{R}) = 0$ then there exists a unique function $\nu_\xi$ with compact support in the interior of $D_{\varepsilon}(T^\ast S^2)$ such that $\iota_{X_\xi } (\ell_{R/\varepsilon}^\ast \beta ) = - d \nu_\zeta$. The functions $\nu_\zeta$ clearly assemble into a single smooth function $\nu : T^\ast S^2 \to \mathbb{R}^3 \cong \mathfrak{so}(3)^\ast$, i.e. $\langle \nu(u,v) ,  \xi \rangle = \nu_\xi (u,v) $. Thus, for each $s$, the $SO(3)$--action on $(D_{\varepsilon} (T^\ast S^2 ) , \omega^s )$ is Hamiltonian, with moment map $\mu^s = \mu + s\nu$ (where $\mu(u,v) = - u \times v$).

\begin{lemma}\label{lemma:nonzero}
Given $\varepsilon>0$, there exists $0 < s_0 \ll 1$ such that if $0 < s \leq s_0$ then the moment map $\mu^s$ is nowhere-vanishing on $D_{\varepsilon}(T^\ast S^2 ) $.
\end{lemma}

\begin{proof}
Since $\iota_{0}^\ast (\ell_{R/\varepsilon}^\ast \beta )$ agrees with the standard symplectic form $ \beta_{S^2}$ on $S^2$, then on the zero section $\nu$ agrees with the moment map for the Hamiltonian $SO(3)$ action on $(S^2 , \beta_{S^2} )$, that is:
\begin{align}
\nu (0,v) = v  .\label{nu0}
\end{align}

Let $U \subset D_\varepsilon (T^\ast S^2 )$ be the open subset given by
\[
U = \Big\{ (u,v) \, | \, \nu (u,v) \text{ is not orthogonal to } \mathrm{Span} (u,v)\Big\}.
\]
Observe that by (\ref{nu0}), $U$ contains the zero section $S^2$. Furthermore, for any $s>0$ we have $\mu^s \neq 0$ on $U$. Indeed, the equation $\mu^s (u,v) = 0$ says that $s\nu (u,v) = u \times v$, and hence $\nu (u,v)$ is orthogonal to both $u$ and $v$ when $s >0$, so $(u,v) \notin U$.

Suppose for a contradiction that there exists sequences $s_i \in \mathbb{R}$, $(u_i , v_i ) \in D_\varepsilon (T^\ast S^2 )$ with $s_i >0$, $s_i \to 0$ and $\mu^{s_i}(u_i , v_i ) = 0$ for all $i$. In particular, $(u_i , v_i ) \notin U$. Passing to a subsequence, we may assume that $(u_i , v_i ) $ converges to a point $ (u,v) \in T^\ast S^2 $; necessarily, this satisfies $(u,v) \in D_{\varepsilon}(T^\ast S^2 ) \setminus U$ and hence $u \neq 0$ because $U$ contains the zero section. On the other hand $\mu^{s_i} (u_i , v_i ) = 0$ implies $\mu (u,v) = 0$, i.e. $u \times v = 0$; this is only possible if $u = 0$, which gives a contradiction.
\end{proof}

By Lemmas \ref{lemma:circle} and \ref{lemma:nonzero}, when $0 < s \leq s_0$ the function $\| \mu^s \|$ is the Hamiltonian of an $S^1$--action $\sigma^s$ on $(D_{\varepsilon}(T^\ast S^2 ) , \omega^s )$; and the Hamiltonian flow of the modified function $r (\| \mu^s \| )$ on $(D_{\varepsilon}(T^\ast S^2 ) , \omega^s )$ at time $t$ is given by 
\[
\phi^{s}_t (u,v) = \sigma_{t r^\prime (\| \mu^s (u,v)\|)}^{s} (u,v)  .
\]


For a fixed $0 < s \leq s_0$, the path runnning from $t = 2\pi$ to $t = 0$ given by $\phi_{t}^s \in \mathrm{Symp}_{\partial} (D_\varepsilon (T^\ast S^2 ), \omega^s )$ is an $\omega^s$--symplectic isotopy from $\phi_{2\pi}^s$ to the identity map. On the other hand, as $s \to 0$ the actions $\sigma^s$ converge with all derivatives on compact subsets of $D_{\varepsilon} (T^\ast S^2 ) \setminus S^2$ to the $S^1$--action $ \sigma$ (the normalized geodesic flow), and on a neighborhood of the zero section $\phi_{2\pi}^s$ agrees with the identity; therefore, the path $\phi_{2\pi}^s$ defined for $0 < s \leq s_0$ extends smoothly across $s = 0$ as the squared model Dehn--Seidel twist
\[
\tau^2 (u,v) = \sigma_{2\pi r^\prime ( \| u \| )}(u,v).
\]

\begin{definition}\label{definition:modelO}
The \textit{model canonical lift} $\mathcal{O} \in \pi_1 \mathcal{S}_{0} (D_\varepsilon (T^\ast S^2 ),\omega )$ of $\tau^2 \in \pi_0 \mathrm{Symp}_0 (D_\varepsilon (T^\ast S^2 ) , \omega )$ is the homotopy-class of the loop represented by the concatenation of the following two paths: 
\begin{enumerate}
\item the path $(\varphi_{t}^{s_0})^\ast \omega $ running from $t = 0$ to $t = 2 \pi$ (for a fixed $0 < s_0 \ll 1 $)
\item the path $(\varphi_{2\pi}^{s} )^\ast \omega$ running from $s = s_0 $ to $s = 0$.
\end{enumerate}
By construction, the loop $\mathcal{O}$ provides a lift of $\tau^2$ under the connecting map $\delta$ of the fibration (\ref{fibration}) for $(X,\omega ) = (D_\varepsilon (T^\ast S^2 ) , du \wedge dv )$, i.e. $\delta ( \mathcal{O} ) = \tau^2$.\\

More generally, let $L \subset X$ be an embedded and \textit{oriented} Lagrangian sphere in a compact symplectic $4$-manifold $(X, \omega )$. Choosing an orientation-preserving identification $L \cong S^2$, we can again identify a neighborhood of $L \subset (X,\omega )$ with a neighborhood of the zero section in the model space $(T^\ast S^2 , \omega )$ and use this identification to implant $\mathcal{O}$ to obtain a loop $\mathcal{O}_L \in \pi_1 \mathcal{S}_{0} (X, \omega )$, which we call the \textit{canonical lift} of $\tau_{L}^2$. (The element $\mathcal{O}_L \in \pi_1 \mathcal{S}_0 (X, \omega )$ is independent of the choice of the orientation-preserving identification $L \cong S^2$: this is because the model element $\mathcal{O} \in \pi_1 \mathcal{S}_{0}(D_\varepsilon (T^\ast S^2 ) ,\omega )$ is $SO(3)$--invariant since $\varphi_{t}^s$ is $SO(3)$--equivariant, and by Smale's Theorem \cite{smale} which gives $\mathrm{Diff}_0 (S^2 ) \simeq SO(3)$).\\
\end{definition}

For the remainder of this section we discuss various properties of the canonical lift $\mathcal{O}_L$. The most basic one is the following:

\begin{lemma}\label{lemma:naturality}
Let $(X, \omega )$ be a symplectic $4$-manifold, and $L \subset (X, \omega )$ an embedded and oriented Lagrangian sphere. Then:
\begin{enumerate}
\item Let $f : (X, \omega_X ) \to (Y, \omega_Y )$ be a symplectomorphism. Under the canonical isomorphism $f^\ast : \pi_1 \mathcal{S}_{0}(Y, \omega_Y ) \to \pi_1 \mathcal{S}_{0}(X, \omega_X )$ we have $f^\ast \big( \mathcal{O}_{f(L)}^{\omega_Y}\big) = \mathcal{O}_{L}^{\omega_X} $, where $f(L)$ is given the orientation induced from $L$ and $f$.
\item If $\delta$ denotes the connecting map of the fibration (\ref{fibration}) associated to $(X, \omega )$, then we have $\delta ( \mathcal{O}_L ) = \tau_{L}^2$.
\end{enumerate}
 \end{lemma}
\begin{proof}
(1) is immediate from the construction of $\mathcal{O}_L$. (2) follows from the fact that this property holds in the model case $(X, \omega ) = (D_\varepsilon ( T^\ast S^2 ) , du \wedge dv )$ and $L = S^2$ by construction, and by the naturality of the fibration (\ref{fibration}) and the element $\mathcal{O}_L$ (i.e. assertion (1)).
\end{proof}

\subsubsection{Triviality in $\pi_1 \mathcal{S}(X,\omega )$}

\begin{lemma}\label{lemma:kerO}
Let $(X, \omega )$ be a symplectic $4$-manifold, and $L \subset (X, \omega )$ an embedded and oriented Lagrangian sphere. Let $\mathcal{S}_L (X, \omega )$ be the subspace of $ \mathcal{S}(X,\omega )$ consisting of symplectic forms $\omega^\prime$ with $\int_L \omega^\prime \geq 0$. Then $\mathcal{O}_L$ lies in the kernel of the natural map 
\[
\pi_1 \mathcal{S}_0 (X, \omega ) \to \pi_1 \mathcal{S}_L (X,\omega ).
\]
(In particular, $\mathcal{O}_L $ lies in the kernel of the natural map $\pi_1 \mathcal{S}_0 (X, \omega ) \to \pi_1 \mathcal{S} (X,\omega )$.)
\end{lemma}
\begin{proof}
By naturality (cf. Lemma \ref{lemma:naturality}(1)) it suffices to consider the model case $(X, \omega ) = (D_\varepsilon (T^\ast S^2) , du \wedge dv )$, $L = S^2$. Fix $0 < s_0 \ll 1$ as in Lemma \ref{lemma:nonzero}, and a smooth function $\chi : [0 , s_0 ] \to \mathbb{R}_{\geq 0}$ such that $\chi (s) = s$ on an open neighborhood of $s = 0$ and $\chi (s) = 0$ on an open neighborhood of $s = s_0$. Consider the deformation $\omega^s$ from (\ref{modeldeformation}). Since $\phi^{s}_t$ is an $\omega^s$--symplectomorphism, it follows that---even if $\phi^{0}_t$ is only defined on $D_\varepsilon (T^\ast S^2 ) \setminus S^2$---the following gives a family of symplectic structures defined for all $(s,t) \in D^2 := [0 , s_0 ] \times [0,2 \pi ]$:
\begin{align}
(\phi_{t}^s )^\ast \omega^{\chi (s)} .\label{disk}
\end{align}

Within the $D^2$--family (\ref{disk}) the two paths given by $s =0$, $0 \leq t \leq 1$ and $0 \leq s \leq s_0$, $t = 0$ are both constant (equal to $\omega$). The path given by $s = s_0$, $0 \leq t \leq 2 \pi$ coincides with the path (1) from Definition \ref{definition:modelO}. On the other hand, the path given by $0 \leq s \leq s_0$, $t = 2\pi$ does not agree with the path (2) from Definition \ref{definition:modelO}. However, there is an obvious homotopy between these two paths (fixing endpoints): 
\[
(\phi_{2\pi}^s)^\ast  \omega^{a \chi (s)} \quad , \quad 0 \leq a \leq 1.
\]
(Indeed, $\phi_{2\pi}^s$ is a well-defined path of diffeomorphisms of $X = D_\varepsilon (T^\ast S^2 )$; unlike $\phi_{t}^s$ for fixed $t \in (0,2\pi )$ which cannot be extended over to the zero section.)

Thus, appying this homotopy to deform the path $0\leq s \leq 1$, $t = 2\pi$ fixing endpoints, we can deform the $D^2$--family (\ref{disk}) to a new $D^2$-family for which the loop obtained by travelling along the boundary in counter-clockwise orientation constitutes a loop in $\mathcal{S}_0 (X, \omega )$ representing the model canonical lift $\mathcal{O}$. The $D^2$-family thus obtained does not lie in $\mathcal{S}_0 (X,\omega )$ but only in the larger space $\mathcal{S}_L (X, \omega )$. The existence of this family implies the required result. 
\end{proof}

\subsubsection{Orientation change and symplectic conjugation}
We shall see later---using Seiberg--Witten theory---that the canonical lift $\mathcal{O}_L$ \textit{does depend on the choice of orientation of $L$} (see Corollary \ref{proposition:calculation}). Thus, we have two different canonical lifts $\mathcal{O}_{\pm L}$ of $\tau_{L}^2$, corresponding to each of the two orientations of the Lagrangian sphere $L$.

A loop of $\omega_t$ in the space $\mathcal{S}_0 (X, \omega )$ yields a corresponding loop $- \omega_t$ in $\mathcal{S}_0 (X, -\omega )$ by taking the conjugate (i.e. negation) of each symplectic forms in the loop, and vice versa. We now describe the loop that one obtains by conjugating the canonical lift $\mathcal{O}_{L}^{\omega}$:

\begin{lemma}\label{lemma:symmetriesO}
Let $(X, \omega )$ be a symplectic $4$-manifold, and $L \subset (X, \omega )$ an embedded and oriented Lagrangian sphere. Then $- \mathcal{O}_{L}^{\omega} =  (\mathcal{O}_{-L}^{-\omega})^{-1}$ in $\pi_1 \mathcal{S}_0 (X, -\omega )$. 
\end{lemma}
(As a consistency check, we note that Lemma \ref{lemma:symmetriesO} and Lemma \ref{lemma:naturality}(2) recover the fact implied by Lemma \ref{lemma:symmetries} that $(\tau_{L}^{\omega})^2 = (\tau_{L}^{-\omega})^{-2}$).

\begin{proof}
Again by naturality (cf. Lemma \ref{lemma:naturality}(1)) it suffices to consider the model case $(X, \omega ) = (D_\varepsilon (T^\ast S^2) , du \wedge dv )$, $L = S^2$. Fix $0 < s_0 \ll 1$ as in Lemma \ref{lemma:nonzero}. Consider the diffeomorphisms $a$ and $c$ of $X = D_\varepsilon (T^\ast S^2 )$ given by $a(u,v) = (-u,-v)$ and $c(u,v) = (-u,v)$. The map $a$ is the symplectomorphism of $(X, \omega )$ induced by the orientation-reversing antipodal involution of $L = S^2$; in turn $c$ is the anti-symplectomorphism given by negation on the fibers, which fixes $L = S^2$. By Lemma \ref{lemma:naturality}(1) we thus have $a^\ast c^\ast \mathcal{O}_{L}^{\omega} = \mathcal{O}_{-L}^{-\omega}$. We shall next show that $a^\ast c^\ast \mathcal{O}_{L}^\omega = - (\mathcal{O}_{\omega})^{-1}$.

Inspecting the construction of the model canonical lift $\mathcal{O}$ (Definition \ref{definition:modelO}) one sees that $a^\ast c^\ast \mathcal{O}$ is given by the homotopy-class of the loop represented by the concatenation of the following two paths: 
\begin{enumerate}
\item the path $-(\varphi_{-t}^{s_0})^\ast \omega $ running from $t = 0$ to $t = 2 \pi$ (for a fixed $0 < s_0 \ll 1 $)
\item the path $-(\varphi_{-2\pi}^{s} )^\ast \omega$ running from $s = s_0 $ to $s = 0$.
\end{enumerate}
(Indeed, applying the involution $a \circ c = c \circ a$ has the effect of replacing $\omega , \beta , \mu , \nu , \omega^s , \| \mu^s \|$ by $-\omega , - \beta , -\mu , - \nu , -\omega^s , \| \mu^s \|$, and thus replacing $\phi_{t}^{s}$ by $(\phi_{t}^{s})^{-1} = \phi_{-t}^s$). Thus, the assertion $a^\ast c^\ast \mathcal{O}_{L}^\omega = - (\mathcal{O}_{\omega})^{-1}$ would then follow from the following:
\begin{claim}
Let $(X, \omega )$ a symplectic manifold (either closed or compact with convex contact boundary). Let $f_t$ be a path in $\mathrm{Diff}(X)$ with $f_0 = \mathrm{Id}$ and $f_1 \in \mathrm{Symp}(X,\omega )$. Then the loops $f_{t}^\ast \omega$ and $(f_{t}^{-1})^\ast \omega$ define inverse elements in $\pi_1 \mathcal{S}_0 (X, \omega )$.
\end{claim}

To prove the Claim, we consider the homotopy of based loops $F_a : S^1 \to \mathrm{Diff}(X)$ parametrised by $0 \leq a \leq 1$ which is given by the concatenation of the following two paths:
\begin{enumerate}
\item the path from $\mathrm{Id}$ to $f_a$ given by $f_t$ with $t$ going from $t = 0$ to $t = a$
\item the path from $f_a$ to $\mathrm{Id}$ given by $f_a f_{t}^{-1}$ with $t$ going from $t = 0$ to $t = a$.
\end{enumerate}

Since $f_1$ is a symplectomorphism, then $(f_1 f_{t}^{-1})^\ast \omega = (f_{t}^{-1} )^\ast \omega$. Thus, the loop $F_{1}(t)^\ast \omega$ agrees with the concatenation of the loops $f_{t}^\ast \omega$ and $(f_{t}^{-1})^\ast \omega$. In turn $F_0$ is the constant loop at $\mathrm{Id}$. Thus, the loop $f_{t}^\ast \omega$ followed by the loop $(f_{t}^{-1})^\ast \omega$ is null-homotopic, and a similar argument applies for the concatenation in the opposite order. This concludes the proof of the Claim, and thus of the Lemma.
\end{proof}

\subsection{The canonical lift and the generalised Dehn twist along a $(-2)$-sphere}\label{subsection:generaliseddehntwist}

Let $X$ be a smooth oriented $4$-manifold, and $S \subset X$ a smoothly embedded and oriented sphere with self-intersection $S \cdot S = -2$. Recently, J. Lin has studied a loop of diffeomorphisms $\gamma_S \in \pi_1 \mathrm{Diff}_0 (X)$ naturally associated to such a sphere $S$, which he called the \textit{generalised Dehn twist along} $S$ \cite{JLIN2022}. We now elucidate the relationship between $\gamma_S$ and the Dehn--Seidel twist along a Lagrangian sphere $\tau_{L}$.\\

We first recall the construction of the element $\gamma_S \in \pi_1 \mathrm{Diff}_0 (X)$. Consider the complex line bundle $L = \mathscr{O}_{\mathbb{C}P^1}(-2) \to \mathbb{C}P^1$. The $S^1$--action on  the fibers of $L$ preserves the unit sphere $S(L ) = \mathbb{R}P^3$ and thus induces a loop $\gamma$ of diffeomorphisms of $\mathbb{R}P^3$. Moreover, the loop $\gamma$ is a loop of isometries of $\mathbb{R}P^3$ equipped with a standard metric of constant curvature. The identity component in the isometry group, denoted $\mathrm{Isom}_0 (\mathbb{R}P^3 )$, is 
\[
\mathrm{Isom}_0 (\mathbb{R}P^3 ) = SO(4)/\pm I \quad (\cong SO(3) \times SO(3) ).
\]
The inclusion $U(2) \hookrightarrow SO(4)$ induces a map $U(2) \to SO(4)/\pm I$ and the loop $\gamma \in \pi_1 \mathrm{Isom}_0 (\mathbb{R}P^3 ) $ can be characterised as the image in $\pi_1 \big( SO(4)/\pm I\big)$ of the `standard' generator of $\pi_1 U(2)$ (namely, the generator associated to the isomorphism $det : \pi_1 U(2) \xrightarrow{\cong} \pi_1 U(1) = \mathbb{Z}$). The element $\gamma$ has order $2$ in $\pi_1 \mathrm{Isom}_0 (\mathbb{R}P^3 ) \, (\cong \mathbb{Z}_2 \times \mathbb{Z}_2 )$ and because $\pi_2 \mathrm{Isom}_0 (\mathbb{R}P^3 )  $ is trivial then $\gamma^2$ has a (homotopically) unique homotopy to the constant loop (based at the identity). Thus, one can `cut off' the squared $S^1$--action on the unit disk bundle $D(L)$ using the homotopy of $\gamma^2$ to produce a loop of diffeomorphism of $D(L)$ which agree with the identity near the boundary, denoted $\gamma_{\mathbb{C}P^1} \in \pi_1 \mathrm{Diff}_0 (D(L))$.

Returning to our $(-2)$--sphere $S \subset X$, we fix an orientation-preserving diffeomorphism $S \cong \mathbb{C}P^1$, the choice of which is homotopically unique (again, by Smale's Theorem \cite{smale}). A neighborhood $\nu (S) $ of $S$ in $X$ is then identified---in a homotopically canonical fashion---with a neighborhood of the zero section in the complex line bundle $L = \mathscr{O}_{\mathbb{C}P^1}(-2) \to \mathbb{C}P^1$, and the generalised Dehn twist element $\gamma_S \in \mathrm{Diff}_0 (X)$ is defined by implanting $\gamma_{\mathbb{C}P^1}$ using the neighborhood identification.


\begin{remark} Bamler--Kleiner \cite{bk3} have proved that there is a homotopy-equivalence $\mathrm{Isom} (\mathbb{R}P^3) \simeq \mathrm{Diff}(\mathbb{R}P^3)$ induced by the inclusion. Using this fact one can rephrase the construction of $\gamma_S $ in more flexible terms.
\end{remark}

It is easy to see that reversing the orientation of $S$ inverts the generalised Dehn twist:
\begin{lemma}\label{lemma:gammasymmetries}
Let $X$ be a smooth $4$-manifold, and $S \subset X$ and embedded and oriented sphere with $S\cdot S = -2$. Then $\gamma_{-S} = -\gamma_{S} \in \mathrm{\pi}_1 \mathrm{Diff_0 }(X)$.
\end{lemma}

Let $p_\ast : \pi_1 \mathrm{Diff}_0 (X) \to \pi_1 \mathcal{S}_0 (X, \omega )$ denote the map induced by the projection map in the fibration (\ref{fibration}). By Lemma \ref{lemma:naturality}(2) we have that the two elements of $\pi_1 \mathcal{S}_0 (X, \omega )$ given by\footnote{In our convention for concatenating loops or paths, $\gamma_2 \cdot \gamma_1$ is to be read from right to left, i.e. first travel along $\gamma_1$, then $\gamma_2$.} 
\[
\mathcal{O}_{-L}^{-1}\cdot \mathcal{O}_L \quad ,\quad \mathcal{O}_L\cdot \mathcal{O}_{-L}^{-1}
\]
have trivial image in $\pi_0 \mathrm{Symp}_0 (X, \omega )$, and hence are contained the image of $p_\ast$. We have the following


\begin{theorem}\label{theorem:generaliseddehntwist}
Let $(X, \omega )$ be a symplectic $4$-manifold, and $L$ an embedded and oriented Lagrangian sphere in $(X, \omega )$. 
Then the two canonical lifts commute:
\[
\mathcal{O}_{L}\cdot \mathcal{O}_{-L } = \mathcal{O}_{-L}\cdot \mathcal{O}_L \in \pi_1 \mathcal{S}_0 (X, \omega ).
\]
Furthermore, we have
\[
p_\ast ( \gamma_L ) = \mathcal{O}_{-L}^{-1}\cdot \mathcal{O}_L =  \mathcal{O}_L\cdot \mathcal{O}_{-L}^{-1}  .
\]
\end{theorem}
\begin{proof}
As usual, we shall only need to consider the model case where $(X, \omega ) = (D_\varepsilon ( T^\ast S^2 ) , du \wedge dv )$, and $L = S^2$ equipped with the standard orientation. The proof proceeds in several steps.\\

(i) We denote by $\mathcal{O}_{+} =\mathcal{O} \in \pi_1 \mathcal{S}_0 (D_\varepsilon (T^\ast S^2 ) , \omega)$ the model canonical lift associated the standard orientation of the zero section $S^2$. 
We will now establish an alternative---more convenient for our purposes---description of the element $\mathcal{O}_{+}$.\\


Rather than considering the deformation of $\omega$ given by (\ref{modeldeformation}) we shall now use instead the following one (for which we use the same notation, for convenience):
\begin{align}
\omega^s := \omega + s \pi^\ast \beta_{S^2} \quad , \quad s \geq 0 \label{Omega}
\end{align}
where $\pi : T^\ast S^2 \to S^2$ is the bundle projection, and $\beta_{S^2}$ the unique $SO(3)$--invariant symplectic form on $S^2$ of area $4 \pi$. (Explicitly---in terms of the $3$-dimensional cross product---we have $(\beta_{S^2})_v (\dot{v}_1, \dot{v}_2) = \langle v , \dot{v}_1 \times \dot{v}_2 \rangle$). 
The deformation (\ref{Omega}) is considered in \cite[Proposition 1.1]{seidel}. The $2$-forms in (\ref{Omega}) are $SO(3)$--invariant symplectic forms on $T^\ast S^2$ \textit{for all} values of $s$. 

We will now repeat the procedure leading up to Definition \ref{definition:modelO} but in the present context, using (\ref{Omega}). As before, $\sigma^s$ will denote the $S^1$--action associated to the Hamiltonian $\| \mu^s \|$, where $\mu^s$ is the moment map for the $SO(3)$--action on $(T^\ast S^2 , \omega^s )$ which now has the simpler expression $\mu^s= - u \times v  + s v $. 
Let $\phi^{s}_t$ denote the $\omega^s$--Hamiltonian flow associated to the cut-off Hamiltonian $r^{s} ( \| \mu^s \| ) = r^s ( \sqrt{s^2 + \| u \|^2 })$, where $r^{s}$ is a family of smooth non-negative functions on $\mathbb{R}_{\geq 0}$ with $r^s (x) = x$ for $x$ in a neighborhood of $[ 0 , \sqrt{s^2 + (\varepsilon/2)^2}]$ and $r^s (x) = 0$ for $x$ in a neighborhood of $[\sqrt{s^2 + \varepsilon^2} , \infty )$. 
(The maps $\sigma^{s}_t$ and $\phi_{t}^s$ have similar properties as before: they are everywhere defined on $T^\ast S^2$ when $s\neq 0$; and undefined along the zero section when $s = 0$, etc). 

With this in place, then we can define an element $\mathcal{O}_+ = \mathcal{O} \in \pi_1 \mathcal{S}_0 (D_\varepsilon ( T^\ast S^2 ) , \omega )$ exactly as in Definition \ref{definition:modelO} but with $\phi_{t}^s$ defined instead as in the previous paragraph and with $s_0 >0 $ allowed to take any given value (because now $\mu^s$ is nowhere vanishing for every $s \neq 0$). One can easy check that this construction coincides with the one from Definition \ref{definition:modelO}. (Indeed, the obvious interpolation between (\ref{Omega}) and (\ref{modeldeformation}) is through symplectic forms on $D_\varepsilon (T^\ast S^2 )$ with constant cohomology class for each fixed $s$, provided $s$ is restricted to $0 \leq s \leq s_0$ with $s_0 \ll 1$; and one can then use Moser's Theorem to establish the equivalence.)

For the remainder of the proof we shall only consider the constructions based on (\ref{Omega}), and no longer on (\ref{modeldeformation}). (We note that the construction of the canonical lift based on (\ref{modeldeformation}) was previously needed in order to prove Lemma \ref{lemma:kerO}, where it was crucial that the forms in (\ref{modeldeformation}) agree with $\omega$ near the boundary of $D_\varepsilon (T^\ast S^2 )$. The latter is a property that (\ref{Omega}) does not enjoy but that we shall not use in this proof.)\\

(ii) We now describe $\mathcal{O}_- \in \pi_1 \mathcal{S}_0 (D_\varepsilon (T^\ast S^2 ) , \omega )$, the canonical lift associated to the opposite orientation of the zero section.\\

We claim that the homotopy-class $\mathcal{O}_{\pm }$ is represented by the concatenation of the following two paths: 
\begin{enumerate}
\item the path $(\varphi_{t}^{\pm s_0})^\ast \omega $ running from $t = 0$ to $t = 2 \pi$ (for any given $s_0 >0$)
\item the path $(\varphi_{2\pi}^{s} )^\ast \omega$ running from $s = \pm s_0 $ to $s = 0$.
\end{enumerate}
For $\mathcal{O}_+$ this was just a repetition of the definition. For $\mathcal{O}_-$ this follows from the fact that reversing the orientation of $S^2$ amounts to replacing $\beta_{S^2}$ with $- \beta_{S^2}$; and in turn, the latter amounts to replacing $s$ by $-s$.\\


(iii) The forms on $T^\ast S^2$ given by (\ref{Omega}) are symplectic \textit{for all} values of $s$, and we shall now analyse the behaviour of the corresponding $S^1$--actions $\sigma^s$ as $s \to \infty$. (Analogous observations will apply to the case $s \to -\infty$).\\

Noting that $\| \mu^s \| = \sqrt{ s^2 + \| u \|^2 }$, a calculation shows that the vector field $X^s$ which integrates to $\sigma^s$ (defined by $\iota_{X^s} \omega^s = - d \| \mu^s \|$) is given by
\begin{align*}
    X^s (u,v) = \frac{\big( \, - \| u\|^2 v - s \cdot u \times v \,\, , \, \, u \, \big)  }{\sqrt{s^2 + \| u \|^2}}.
\end{align*}

Observe the following:
\begin{itemize}
\item As $s \to \infty$, $X^s$ converges with all derivatives on compact subsets of $T^\ast S^2$ to the vector field $X^\infty$ given by
\begin{align*}
X^\infty (u,v) = \big( \, - u \times v \,\, , \, \, 0 \, \big)
\end{align*}
\item The vector field $X^\infty$ integrates to an $S^1$--action on $T^\ast S^2$ given by
\begin{align*}
\sigma^\infty (u,v) = \big( \, \cos t \cdot u - \sin t \cdot u \times v \,\, , \, \, v \, \big) .
\end{align*}
\end{itemize}

Of course, the $S^1$--action $\sigma^\infty$ on $T^\ast S^2$ coincides under the identification $T^\ast S^2 \cong \mathscr{O}_{\mathbb{C}P^1}(-2)$ as a complex bundle with the fiberwise complex multiplication.\\

(iv) We now conclude the proof.\\

On the smaller region $D_{\varepsilon/2} (T^\ast S^2) $ we have $\phi_{t}^s = \sigma_{t}^s$. Thus, by (iii) on this region we have the following convergence (with all derivatives on compact sets): as $ s\to \pm \infty$
\[
\phi_{t}^s \to (\sigma_{t}^{\infty})^{\pm 1} .
\]
Therefore, we have that:
\begin{itemize}
\item The loop $\gamma^2$ 
is homotopic to the loop of diffeomorphisms of $S_{\varepsilon/2}(T^\ast S^2 ) = \mathbb{R}P^3$ obtained by the action of $\phi_{t(a)}^{s(a)}$ where $a \mapsto (s(a),t(a))$ is a loop parametrising the boundary of the square $[-s_0 , s_0 ]\times [0 , 2\pi ]$ counter-clockwise and based at $(s(0),t(0) ) = (s_0 , 0 )$. (Note that the maps $\phi_{t}^s$ preserve the hypersurfaces $\|u \| = \mathrm{const.}$). 
\item The correponding action of $\phi_{t}^s$ on $D_\varepsilon (T^\ast S^2 ) \setminus D_{\varepsilon /2}(T^\ast S^2) = \mathbb{R}P^3 \times [\varepsilon/2 , \varepsilon ]$ is describing a null-homotopy in $\mathrm{Diff}_0 ( \mathbb{R}P^3 )$ of the loop $\gamma^2$. (Since $\mathrm{Diff}_0 (\mathbb{R}P^3) \simeq \mathrm{Isom}_0 (\mathbb{R}P^3)$ \cite{bk3} then we need not worry about whether this null-homotopy lies inside $\mathrm{Isom}_0 (\mathbb{R}P^3 )$). 
\end{itemize}
By construction, $\mathcal{O}_{-}^{-1}\cdot \mathcal{O}_+$ is represented by the loop $(\phi_{t(a)}^{s(a)})^\ast \omega$ and thus $p_\ast ( \gamma_{S^2} ) = \mathcal{O}_{-}^{-1}\cdot \mathcal{O}_+$. In turn, $\mathcal{O}_+ \cdot \mathcal{O}_{-}^{-1}$ is also represented by $(\phi_{t(a)}^{s(a)})^\ast \omega$ except that now the loop $(s(a),t(a))$ starts at $(s(0),t(0)) = (0,2\pi )$ and hence $\phi_{s(a)}^{t(a)}$ is based at $\phi_{2\pi}^0$. This is no problem: $\phi_{2\pi}^0$ is a symplectomorphism for $\omega$, so we can replace $\phi_{t(a)}^{s(a)}$ by $(\phi_{2\pi}^{0})^{-1} \phi_{t(a)}^{s(a)}$. This shows $p_\ast ( \gamma_{S^2} ) = \mathcal{O}_+ \cdot \mathcal{O}_{-}^{-1}$. The proof is now complete.
\end{proof}

$ $

Given a symplectic $4$-manifold $(X, \omega )$ and an oriented embedded Lagrangian sphere $L$ we have the following diagram, which summarises some properties of the squared Dehn--Seidel twist $(\tau_{L}^\omega )^2$, its canonical lifts $\mathcal{O}_{\pm L}^{\omega}$, and the generalised Dehn twist $\gamma_L$ on the smoothly embedded $(-2)$--sphere $L$ (cf. Lemma \ref{lemma:naturality}, Theorem \ref{theorem:generaliseddehntwist}, Proposition \ref{proposition:calculation}):

\[
\begin{tikzcd}
\pi_1 \mathrm{Diff} (X) \arrow{r}{p_\ast}& \pi_1 \mathcal{S}_0 (X,\omega ) \arrow{r}{\delta}& \pi_0 \mathrm{Symp}_0 (X,\omega ) \arrow{r} & \pi_0 \mathrm{Diff}(X)\\
 & \mathcal{O}^{\omega}_{+L} \arrow[equal, "/" marking]{dd} \arrow[|->]{rd}&  & \\
 & & (\tau_{L}^{\omega})^2 \arrow[|->]{r} & \mathrm{Id} \\
 & \mathcal{O}^{\omega}_{-L}\arrow[|->]{ru}&  & \\
 \gamma_L \arrow[|->]{r} & (\mathcal{O}^{\omega}_{-L})^{-1}\cdot \mathcal{O}^{\omega}_{+L} = \mathcal{O}^{\omega}_{+L}\cdot (\mathcal{O}_{-L}^{\omega})^{-1} \arrow[|->]{r} & \mathrm{Id} 
\end{tikzcd}.
\]

\section{Kronheimer's invariant}\label{section:kronheimer}

In this section, we discuss Kronheimer's invariant for families of symplectic forms on $4$-manifolds \cite{kronheimer}, which constitutes the main technical tool of this article. 

Some aspects described here have appeared elsewhere in the literature \cite{kronheimer,monocont,li-liu,Smirnov}. We emphasize here a point of view on Kronheimer's invariant which regards it as a sort of `linking number' with certain codimension $2$ `hypersurfaces' of `infinite dimensions'; we will later encounter a finite-dimensional counterpart of this when we study the abelianisation map $P \to P^{\mathrm{ab}}$ on the pure Braid group. 

We will also discuss Kronheimer's invariant in the case of compact symplectic $4$-manifolds with contact boundary. Via an \textit{excision property} (Proposition \ref{proposition:excision}), the invariants for closed and compact manifolds with boundary can be related to each other, which will be an essential tool for performing calculations in later sections.

\subsection{Kronheimer's invariant for loops of symplectic forms}

Throughout the following discussion, we fix the following data:
\begin{itemize}
\item A compact, connected, symplectic $4$-manifold $(X,\omega)$ with convex boundary on an closed contact $3$-manifold $(Y, \xi )$ (we allow $Y = \emptyset$). 

\item a cohomology class $e \in H^2 (X, Y , \mathbb{Z} )$ such that
\begin{align}
 \quad [\omega]\cdot e \leq 0 \quad \text{and} \quad e\cdot e - K_\omega \cdot e = -2 . \label{condition1}
\end{align}
\end{itemize}
(Here $K_\omega = - c_1 (TX , \omega)  \in H^2 (X ,\mathbb{Z} )$ denotes the \textit{canonical class} of the symplectic $4$-manifold $(X, \omega )$.)\\

In his beautiful article \cite{kronheimer}, Kronheimer constructed certain group homomorphism using $2$-parametric Seiberg--Witten gauge theory (together with higher-parametric analogues, which we won't consider in this article), which we refer to as \textit{Kronheimer's invariant} in the class $e$:
\begin{align}
Q_e : (\pi_1 \mathcal{S}_{0}(X, \omega ))^{\mathrm{ab}} \to \mathbb{Z} .\label{kronheimer}
\end{align}
(To fix the overall sign of $Q_e$, an additional choice of a \textit{homological orientation} of $X$ is required; see below for details). We now recall the construction of (\ref{kronheimer}), in a somewhat reformulated and generalised version. 

\subsubsection{The moduli spaces when $X$ is closed}

We first describe the relevant moduli spaces that are used to define (\ref{kronheimer}). We first suppose that $(X,\omega )$ has empty boundary, and discuss the boundary case later. We shall assume the standard facts on spin geometry and the Seiberg--Witten equations on smooth and symplectic $4$-manifolds (see \cite{morgan,KM,Taubes94} for further details).\\

Recall that the cohomology group $H^{2}(X, \mathbb{Z} )$ acts freely and transitively on the set of isomorphism classes of spin-c structures on $X$; we denote the action of a class $c \in H^2 (X, \mathbb{Z} )$ on a spin-c structure $\mathfrak{s}$ by $\mathfrak{s} + c$. Attached to the symplectic form $\omega$ on $X$ is a canonical isomorphism class of spin-c structure on $X$ denoted $\mathfrak{s}_\omega$. Thus, one may identify $H^2 (X,\mathbb{Z} )$ canonically with the set of isomorphism classes of spin-c structures, using the correspondence $c \mapsto \mathfrak{s}_\omega + c$. 
The first Chern class of the determinant line bundle of $\mathfrak{s}_\omega +c$ is given by $c_1 ( \mathfrak{s}_\omega + c ) = c_1 ( \mathfrak{s}_\omega ) + 2c = -K_\omega + 2c$.

Let $\mathscr{P} $ denote the \textit{space of perturbations} on $X$, consisting of pairs $(g , \eta )$ where $g$ is a Riemannian metric on $X$ and $\eta \in \Omega^{2}(X, i \mathbb{R} )$ is an imaginary-valued $2$-form which is self-dual with respect to $g$. In what follows, it is necessary to make $\mathscr{P}$ into a Banach manifold. The details needed to achieve this are fairly standard and will be omitted; see \cite{morgan}. (For example, it suffices to consider metrics $g$ of class $C^l$ and perturbations $\eta$ of class $L^{2}_k$, with $k = 3$, $l = 5$.)

Let $\mathscr{M}_e \xrightarrow{\pi_e} \mathscr{P}$ denote the \textit{universal Seiberg--Witten moduli space} in the spin-c structure $\mathfrak{s}_\omega + e$. Namely, $\mathscr{M}_e$ consist of triples $([(A, \Phi )], g , \eta )$ where $[(A, \Phi )]$ is the gauge-equivalence class of a pair consisting of a spin-c connection $A$ on the spinor bundle $S_e = S_{e}^+ \oplus S_{e}^- \to X$ associated to the spin-c structure $\mathfrak{s}_\omega + e$ and a section $\Phi$ of $S_{e}^{+}$ which solves the \textit{Seiberg--Witten equation} for the perturbation $(g, \eta )$:
\begin{align}
F_{\widehat{A}}^{+} - \sigma(\Phi ) & =  \eta \label{SW1}\\
D_{A} \Phi & = 0 \label{SW2}.
\end{align}
(Above, $\widehat{A}$ denotes the induced connection on the determinant line bundle $\Lambda^2 S_{e}^+ $). Implicit in the notation used in (\ref{SW1}-\ref{SW2}) is that the quadratic map $\sigma$, the Dirac operator $D_A$, and the self-dual curvature $F_{\widehat{A}}^{+}$ are defined using the metric $g$. Indeed, as a unitary vector bundle, the spinor bundle $S_e = S_{e}^+ \oplus S_{e}^-$ can be regarded as `fixed', whereas the Clifford bundle structure on $S_e$ varies according to the particular metric $g$.

 Let $\mathscr{M}_{e}^\ast$ be the \textit{irreducible locus in the universal moduli space}, which is defined to be the subset  of $\mathscr{M}_{e}$ where $\Phi$ does not vanish identically, and let $\mathscr{M}^{\mathrm{red}}_e = \mathscr{M}_e \setminus \mathscr{M}_{e}^\ast$ be the \textit{reducible locus} where $\Phi$ vanishes identically. We consider the \textit{discriminant locus} $\Sigma_e := \pi_e ( \mathscr{M}_e )$ in $\mathscr{P}$.  The \textit{reducible locus of the discriminant} is denoted $\mathscr{P}_{e}^{\mathrm{red}} := \pi_e (\mathscr{M}_{e}^{\mathrm{red}}) \subset \Sigma_e$, and its complement in the perturbation space $\mathscr{P}$ is denoted $\mathscr{P}_{e}^\ast = \mathscr{P}\setminus \mathscr{P}_{e}^{\mathrm{red}}$. The locus $\mathscr{P}_{e}^{\mathrm{red}}$ can be explicitly described:
\begin{align}
\mathscr{P}_{e}^{\mathrm{red}} = \{ (g , \eta ) \, | \,  \Pi_{g} (\eta ) = - 2 \pi i \Pi_{g}^{+} ( c_1 ( \mathfrak{s}_\omega + e  ) )\} \label{Pred}
\end{align}
where $\Pi_g : \Omega^{+}_g (X) \to \mathscr{H}_{g}^{+} (X)$ and $\Pi_{g}^{+} : H^{2}(X, \mathbb{R} ) = \mathscr{H}_{g} (X ) \to \mathscr{H}_{g}^{+}(X )$ denote the $L^2$--orthogonal projections onto the space of $g$--self-dual harmonic $2$-forms on $X$. Thus $\mathscr{P}_{e}^{\mathrm{red}}$ is a codimension $b^{+}(X)$ Banach submanifold of $\mathscr{P}$. 

By standard results (\cite{morgan}) the space $\mathscr{M}_{e}^\ast $ is a Banach manifold and the map $\pi_e : \mathscr{M}_{e}^\ast \to \mathscr{P}^\ast$ is proper and Fredholm, with index given by 
\begin{align*}
\mathrm{ind} (\pi_e ) = e\cdot e - K \cdot e = -2 . \label{ind}
\end{align*}
by (\ref{condition1}). 

The Quillen determinant real line bundle of $\pi_e$, denoted $\mathrm{det}(\pi_e ) \to \mathscr{M}_{e}^\ast$ (cf. \cite[\S 20.2]{KM}), is orientable, and orientations of $\mathrm{det}(\pi_e )$ are in canonical correspondence with \textit{homological orientations} of $X$: orientations of the vector space
\[
 H^{1}(X, \mathbb{R} ) \oplus H^{+}(X, \mathbb{R} ).
\]
(Here $H^{+}(X, \mathbb{R} )$ denotes a maximal positive subspace of the intersection form on $H^{2} (X, \mathbb{R} )$).\\ 

\subsubsection{The moduli spaces when $X$ has boundary}\label{subsubsection:moduliboundary}

We now describe how to adapt the previous discussion to the case when $(X, \omega )$ has non-empty convex boundary $(Y, \xi )$. The possibility of extending Kronheimer's homomorphism (\ref{kronheimer}) to this setting was suggested in \cite[\S 5]{kronheimer}, using techniques developed by Kronheimer--Mrowka in \cite{monocont}. In what follows, we describe the construction of the relevant universal Seiberg--Witten moduli space $\mathscr{M}_e \to \mathscr{P}$. (This is similar to the moduli space appearing in \cite[Theorem 2.4]{monocont}, except that we also need to allow for variations of the Riemannian metric).  \\

The canonical spin-c structure $\mathfrak{s}_\omega$ on $X$ restricts along $Y = \partial X$ to the canonical spin-c structure $\mathfrak{s}_\xi$ on $Y$ induced by the contact structure $\xi$. By a spin-c structure on $X$ \textit{relative} to $(Y , \mathfrak{s}_\xi)$ (plainly, a `relative spin-c structure') we shall mean a spin-c structure $\mathfrak{s}$ on $X$ equipped with an isomorphism $\mathfrak{s}|_Y \xrightarrow{\cong} \mathfrak{s}_\xi$. The set of isomorphism classes of relative spin-c structures now has a free and transitive action of $H^2 (X, Y , \mathbb{Z} )$. 

Fix a Liouville field $V$ defined near the boundary of $(X,\omega )$, so as to identify a collar neighborhood of $Y = \partial X \subset X$ symplectically with $ ( (0,1]\times Y , d ( t \alpha ) )$ where $\alpha = (\iota_V \omega )|_Y$. Consider then the \textit{symplectisation completion} of $X$:
 \[
 (\widehat{X} , \widehat{\omega} ) := (X,\omega ) \cup ([1,\infty) \times Y , d(t\alpha ) ),\]
 where the gluing of the two pieces uses the collar neighborhood identification, and the Liouville vector field $V$ extends over the non-compact end $(1,\infty)\times Y \subset \widehat{X}$ as $V = t \cdot \partial/\partial t $. We fix a Riemannian metric $g_0$ on the non-compact end which is compatible with the symplectic form $\widehat{\omega}$ and is conical, i.e. $\mathscr{L}_V g_0 = g_0$. A given relative spin-c structure $\mathfrak{s}$ on $X$ has a canonical extension (by translation on the end) to a spin-c structure $\widehat{\mathfrak{s}}$ on $\widehat{X}$.

We make a quick digression now to recall a fundamental aspect of the Seiberg--Witten equation for almost-Kähler pair $(\omega , g )$ on a $4$-manifold $M$; namely, the fact that for the canonical spin-c structure $\mathfrak{s}_\omega$ there is a \textit{canonical configuration} $(A_0, \Phi_0 )$. Here, the spinor part $\Phi_0$ is chosen to be a unit length trivialisation of the $-2i$ eigenspace of the Clifford multiplication $\rho (\omega ) : S^+ \to S^+$ on the $S^+$ bundle for the canonical spin-c structure $\widehat{\mathfrak{s}}_\omega$, and $A_0$ can be characterised as the unique spin-c connection on that spinor bundle such that $D_{A_0}\Phi_0 = 0$ (see \cite{Taubes94} or \cite[Definition 2.2]{monocont}). 

With this in place, the perturbation space $\mathscr{P}$ for the manifold $X$ with convex boundary $(Y, \xi )$ is now defined as the space of quadruples $(g,\eta , A_0 , \Phi_0 )$ where: $g$ is a Riemannian metric on $\widehat{X}$ such that $g$ agrees with $g_0$ on the end $(1, \infty) \times Y$; $\eta$ is a $2$-form on $\widehat{X}$ which is self-dual with respect to $g$ and is suitably decaying on the end (cf. \cite[Formula (7)]{monocont}); $(A_0 , \Phi_0 )$ is a configuration on the spinor bundle $\widehat{S}_e \to \widehat{X}$ associated to the spin-c structure $\widehat{s}_\omega + e$ such that $(A_0 , \Phi_0 )$ agrees over the end $(1,\infty ) \times Y \subset X$ with the canonical configuration given by the almost-Kähler structure $(\omega , g_0 )$ on the end. (Possibly, the extension $\Phi_0$ can vanish over $\widehat{X} \setminus (1, \infty) \times Y$). The moduli space $\mathscr{M}_e$ is defined as the space of quintuples $([(A, \Phi )] , g , \eta , A_0 , \Phi_0 )$ where $[(A, \Phi )]$ is a gauge-equivalence class of a pair consisting of a spin-c connection $A$ on the spinor bundle $S_e = S_{e}^+ \oplus S_{e}^- \to \widehat{X}$ of the spin-c structure $\widehat{\mathfrak{s}}_\omega + e$ and $\Phi$ is a section of $S_{e}^+$, which solves a `deformed' version of the Seiberg--Witten equation with perturbation $(g,\eta )$:
\begin{align}
F_{\widehat{A}}^+ - \sigma (\Phi ) &=    F_{\widehat{A_0}}^+ - \sigma (\Phi_0 ) + \eta \label{SW1def}\\
D_A \Phi & = 0 . \label{SW2def}
 \end{align}
and such that over the end $(A, \Phi )$ approaches suitably fast the canonical configuration $(A_0 , \Phi_0 )$ associated to the almost-Kähler structure $(\omega , g_0 )$ on the end (cf. \cite[Formula (3)]{monocont}).

The treatment of reducibles becomes rather straightforward in this case: since $(A, \Phi )$ approaches the non-zero configuration $(A_0 , \Phi_0 )$ over the end, then the moduli space $\mathscr{M}_e$ contains only irreducible solutions. By \cite[Theorem 2.4]{monocont}, we have that the universal moduli space $\mathscr{M}_e $ and perturbation space $\mathscr{P}$ are both Banach manifolds, the projection $\pi_e : \mathscr{M}_e \to \mathscr{P}$ is proper and Fredholm, with index $\mathrm{ind}(\pi_e ) = e^2 - K \cdot e = -2$ (using (\ref{condition1})). The Quillen determinant real line bundle $\mathrm{det}(\pi_e ) \to \mathscr{M}_e $ is orientable \cite[Appendix]{monocont}; an orientation of this line bundle corresponds to a \textit{homological orientation} of $X$, which in the present setting shall mean an orientation of the vector space
\[
H^1 (X , \mathbb{R} ) \oplus H^+ (X, \mathbb{R} ) \oplus H^1 (Y, \mathbb{R} ) \oplus \Lambda (Y, \xi )
\]
where $\Lambda (Y, \xi )$ is the determinant line of the linearisation at $(A_0 , \Phi_0 )$ of the deformed Seiberg--Witten equations over $[1, \infty ) \times Y$ with Atiyah--Patodi--Singer boundary condition over $\{1\}\times Y$, see \cite[\S 3.3.4]{yo1}. (Unfortunately, it is not possible to assign to $\Lambda (Y, \xi )$ a canonical orientation, see \cite[Theorem H]{lin-ruberman-saveliev}).

\subsubsection{A `linking number' with the discriminant locus}


We fix from now on a homological orientation of $X$. We consider a smooth loop $\gamma$ in $\mathscr{P}$ avoiding the discriminant locus $\Sigma_e = \pi_e ( \mathscr{M}_e )$, $\gamma : S^1 \to \mathscr{P} \setminus \Sigma_e $ and define a suitable \textit{linking number} $\mathrm{lk}(\gamma , \Sigma_e ) \in \mathbb{Z}$ of the loop $\gamma$ with $\Sigma_e$, as follows.\\

First, we suppose that $b^{+}(X) \geq 4$ or $(X,\omega )$ has non-empty boundary. Let $D$ be the $2$-dimensional disk and $f : D \to \mathscr{P}$ a smooth mapping such that
\begin{itemize}
\item $f$ is an extension of $\gamma$: $f|_{\partial D} = \gamma$
\item $f(D) $ does not meet the codimension $b^{+}(X)$ Banach submanifold $\mathscr{P}_{e}^{\mathrm{red}} \subset \mathscr{P}$
\item $f$ is in `generic position': $f$ is transverse to the Fredholm map $\pi_e$.
\end{itemize}
Then the fiber product $\mathrm{Fib}(f,\pi_e )$ of $f$ and $\pi_e$ is a compact and zero-dimensional manifold. Our choice of homological orientation of $X$ determines an orientation of $\mathrm{Fib}(f, \pi_e )$, and the signed count of points of $\mathrm{Fib}(f, \pi_e )$ will be denoted by $f \cdot \pi_e := \# \mathrm{Fib}(F, \pi_e ) \in \mathbb{Z}$. Because $b^{+}(X) \geq 4$ if $(X,\omega )$ has empty boundary and $\mathrm{dim}D = 2$, then such an extension $f$ exists and any other such extension $f^\prime$ can be homotoped to $f$ by a smooth homotopy $F$ avoiding $\mathscr{P}_{e}^\mathrm{red}$. Such a homotopy $F$ can then be placed in `generic position' and thus by a standard cobordism argument we obtain that $f \cdot \pi_e = f^\prime \cdot \pi_e$. Hence $f \cdot \pi_e$ is independent of the chosen extension $f$, and we define the linking number of $\gamma$ with $\Sigma_e$ by
\[
\mathrm{lk}(\gamma , \Sigma_e ) := f \cdot \pi_e .\\
\]

Suppose now that $(X, \omega)$ has empty boundary and $b^{+}(X) <4$. The linking number $\mathrm{lk}(\gamma , \Sigma_e ) \in \mathbb{Z}$ can still be defined by making use of a `preferred chamber' in the space of perturbations determined by the symplectic class $[\omega ]$, as we now explain (this follows standard ideas which have appeared elsewhere; see \cite{li-liu,Smirnov,JLIN2022}). In this case, we shall further assume that the loop $\gamma (t) = (g_t , \eta_t )$ satisfies the following condition: if $\omega_t$ denotes the unique $g_t$--harmonic representative of the cohomology class $[\omega ]$, then
\begin{align}
\int_X i \eta_t \wedge \omega_t  > 2 \pi c_{1} (\mathfrak{s} + e ) \cdot [\omega ] \quad \forall t \in S^1 . \label{condition_pert}
\end{align}

Let $\mathbb{K} (X) \subset H^{2}(X, \mathbb{R} )$ be the positive cone in the intersection form: 
\[
\mathbb{K} (X ) = \{ a \in H^{2}(X,\mathbb{R} ) \, | \, a\cdot a >0 \} .
\]
It is easy to see that the positive cone has the homotopy type of a $(b^{+}(X) -1)$--dimensional sphere: $\mathbb{K}(X) \simeq S^{b^{+}(X)-1}$. 
\begin{lemma}\label{lemma:chambers}
The map $p: \mathscr{P}_{e}^\ast  \to \mathbb{K} (X) $ given by 
\[
(g, \eta ) \mapsto [ \Pi_g (i \eta ) - 2 \pi \Pi^{+}_g ( c_1 ( \mathfrak{s}_\omega + e )) ]
\]
is a fibration with contractible fibers. (In particular, $p$ is a homotopy-equivalence). 
\end{lemma}
\begin{proof}
By a result of Li--Liu \cite[\S 3.1]{li-liu}, we have a fibration $p_1 : \mathscr{P}_{e}^\ast \to \mathcal{P}$ with contractible fibers, where $\mathcal{P}$ is the `period space' consisting of pairs $(W, l )$ such that $W$ is a maximal positive-definite subspace of $H^2 (X,\mathbb{R} )$ and $l \in H^2 (X, \mathbb{R} )$ is not orthogonal to $W$, and where $p_1 ( g , \eta ) = ( \mathscr{H}^{+}_g (X) , \Pi_g (\eta + 2 \pi i  (c_1 ( \mathfrak{s}_\omega + e) ) ) $ (and $\Pi_g$ denotes the $L^2$--orthogonal projection to the harmonic space $\mathscr{H}_{g}^2 (X) = H^2 (X, \mathbb{R} )$). On the other hand, it is easy to see that the map $p_2 : \mathcal{P} \to \mathbb{K}(X)$ which sends $(W, l )$ to the orthogonal projection of $l$ onto $W$ is also a fibration with contractible fibers. The map $p$ is given by $p = p_2 \circ p_1$, and thus the result follows.
\end{proof}

Now, the positive cone $\mathbb{K}(X)$ has an open and \textit{contractible} subspace given by
\[
\mathbb{K}_{[\omega]} (X) = \{ a\in \mathbb{K}(X) \, | \, a \cdot [\omega] >0 \}.
\]
The condition (\ref{condition_pert}) precisely says that the loop $\gamma $ satisfies $p(\gamma (t)) \in \mathbb{K}_{[\omega]}(X)$ for all $t \in S^1$. Thus, by Lemma \ref{lemma:chambers} and $\mathbb{K}_{[\omega]}(X) \simeq \ast$, it follows that an extension $f : D \to \mathscr{P}$ of $\gamma$ can be chosen so that---in addition to the previous requirements---we have $p(f(t)) \in \mathbb{K}_{[\omega]}(X)$ for all $t \in D$. It also follows that any two such extensions can be connected by a homotopy avoiding $\mathscr{P}_{e}^{\mathrm{red}}$ and in generic position. The linking number $\mathrm{lk}(\gamma , \Sigma_e )$ is then defined again by $f \cdot \pi_e \in \mathbb{Z}$, and this is independent of the chosen extension $f$.

\subsubsection{Taubes' deformation of the Seiberg--Witten equation}\label{subsubsection:taubes}


As a consequence of fundamental results of Taubes \cite{Taubes94,Taubes:more,Taubes:SWimpliesGromov} (when $X$ is closed) and Kronheimer--Mrowka \cite{monocont} (when $X$ has non-empty boundary) we have:

\begin{lemma}\label{lemma:taubes} Suppose (\ref{condition1}) holds. Let $K $ be a compact subset in $\mathcal{AK}_{0}(X, \omega )$ (cf. \S \ref{notation})), which tautologically parametrises almost-Kähler pairs $(\omega_t , g_t )$ for $t \in K$. Then there exists $r_0 = r_0 (K)$ such that for all $r\geq r_0$ the Seiberg--Witten equations (\ref{SW1}-\ref{SW2}), resp. (\ref{SW1def}-\ref{SW2def}), in the spin-c structure $\mathfrak{s}_\omega + e$, resp. $\widehat{\mathfrak{s}}_\omega + e$, admit \textit{no solutions} for the $K$--family of perturbations
\begin{align}
 (\, g_t  \, , \, F_{\widehat{A}_{0,t}}^{+} - r \sigma ( \Phi_{0,t} ) \,  ) \quad , \quad \text{resp.} \quad   (\, g_t  \, , \, F_{\widehat{A}_{0,t}}^{+} - r \sigma ( \Phi_{0,t} ) \, , \,  A_{0,t} , \Phi_{0,t} )\label{perturbation}
\end{align}
where $(A_{0,t} , \Phi_{0,t} )$ denotes the canonical configuration on $X$, resp. $\widehat{X}$, associated to the almost-Kähler structure $(\omega_t , g_t )$.
\end{lemma}

With this in place, let $\ell  = \{\omega_t \}_{t \in S^1}$ be a loop in $\mathcal{S}_0 (X, \omega )$. Kronheimer's invariant $Q_e ( \ell ) \in \mathbb{Z}$ is defined as follows. By $\mathcal{S}_0 (X, \omega ) \simeq \mathcal{AK}_0(X, \omega )$ we may lift $\ell$ to a loop $\widetilde{\ell} = \{ (\omega_t , g_t )\}_{t \in S^1}$ in $\mathcal{AK}_0 (X, \omega )$. Choosing $r_0 = r_0 (\widetilde{\ell}) \gg 0$, $r\geq r_0$ and defining a loop $\gamma (t) =  (g_t , \eta_t )$ (resp. $(g_t, \eta_t , A_{0,t}, \Phi_{0,t})$ when the boundary is non-empty) in $\mathscr{P}$ by the formula (\ref{perturbation}), we have that $\gamma  $ avoids the discriminant $\Sigma_e$ by Lemma \ref{lemma:taubes}. We then set
\[
Q_e (\ell ) := \mathrm{lk}(\gamma , \Sigma_e ) \in \mathbb{Z} .
\]

If $b^{+} (X)\geq 4$ or $(X,\omega )$ has non-empty boundary then $Q_e (\ell )$ is well-defined and only depends on the homotopy class $\ell$. If $(X,\omega )$ has empty boundary and $b^{+}(X) <4$  then $Q_e (\ell )$ is also defined since $\gamma (t) = (g_t , \eta_t )$ satisfies the condition (\ref{condition_pert}) for $r\gg 0$: indeed, one can show $\sigma (\Phi_{0, t}) = i \omega_t$, and hence
\[
\int_X i \eta_t \wedge \omega_t = - 2 \pi  K \cdot [\omega] + r [\omega]\cdot [\omega ] \gg 0. 
\]

It is clear that $Q_e$ descends to a homomorphism $Q_e : (\pi_1 \mathcal{S}_0 (X, \omega ) )^{\mathrm{ab}} \to \mathbb{Z}$. This concludes the definition of Kronheimer's invariant.

\subsubsection{Naturality}

By construction, Kronheimer's invariant $Q_e$ is natural in the following sense:
\begin{proposition}\label{proposition:naturalityQ}
Let $(X, \omega )$ and $(X^\prime , \omega^\prime)$ be compact symplectic $4$-manifold, either closed or with convex contact boundary. Let $f : (X, \omega ) \to (X^\prime , \omega^\prime )$ be a symplectomorphism. Let $e \in H^2 (X^\prime , \partial X^\prime , \mathbb{Z} )$ be a class such that $e^2 - K_{\omega^\prime} \cdot e = -2$ and $[\omega^\prime ] \cdot e \leq 0$. Fix any homological orientation of $(X^\prime, \omega^\prime )$, and then fix the homological orientation of $(X, \omega )$ induced by that on $(X^\prime , \omega^\prime )$ and the symplectomorphism $f$. Then for any loop $\omega_t \in \pi_1 \mathcal{S}_0 (X, \omega )$ we have:
\begin{align*}
Q_{e}^{(X^\prime , \omega^\prime )} ( (f^{-1})^\ast \omega_t ) = Q_{f^\ast e }^{(X, \omega)} ( \omega_t ) .
\end{align*}
\end{proposition}

We also note the following particular case of Proposition \ref{proposition:naturalityQ}:
\begin{corollary}\label{corollary:naturalityQ}
Let $(X, \omega )$ be a compact symplectic $4$-manifold, either closed or with convex contact boundary. Let $f  \in \pi_0 \mathrm{Symp}(X, \omega )$ (i.e. $f$ is a symplectomorphism fixing a neighborhood of the boundary pointwise). Let $e \in H^2 (X , \partial X , \mathbb{Z} )$ be a class such that $e^2 - K_{\omega} \cdot e = -2$ and $[\omega ] \cdot e \leq 0$. Fix any homological orientation of $(X, \omega )$. Let $\sigma(f)  = 0$ or $1$ according as to whether $f^\ast$ preserves or reverses the orientation of $H^1 (X, \mathbb{R} ) \oplus H^+ (X ,\mathbb{R} )$. Then for any loop $\omega_t \in \pi_1 \mathcal{S}_0 (X, \omega )$ we have:
\begin{align*}
Q_{e}^{(X , \omega )} ( (f^{-1})^\ast \omega_t ) = (-1)^{\sigma(f)}  \cdot Q_{f^\ast e }^{(X, \omega)} ( \omega_t ) .
\end{align*}
In particular, if $f$ is a product of Dehn--Seidel twists on Lagrangian spheres embedded in $(X, \omega )$ (disjoint from $\partial X$) then $\sigma(f) = 0$ in the above formula. 
\end{corollary}

The last assertion from Corollary \ref{corollary:naturalityQ} can be deduced as follows. If $L$ is a Lagrangian sphere in $(X, \omega )$ then a neighborhood of $L$ diffeomorphic to the disk cotangent bundle $D(T^\ast L) \cong D(T^\ast S^2 )$. The manifold $D(T^\ast  L)$ is negative-definite and has boundary on a rational homology sphere ($\cong \mathbb{R}P^3$); hence the Dehn--Seidel twist $\tau_L$ acts trivially on $H^+ (X,\mathbb{R} )$ because it is supported on $D(T^\ast L)$. On the other hand, $\tau_L$ acts trivially on $H^1 (X, \mathbb{Z} )$. Thus, Dehn--Seidel twists---or products of such---preserve the orientation of $H^1 (X , \mathbb{R}) \oplus H^+ (X, \mathbb{R} )$.

\subsection{Invariants for symplectomorphisms}

Suppose here that $(X, \omega )$ is a closed symplectic $4$-manifold. (The discussion that follows can also be adapted to the case when $X$ has convex boundary, but we shall not need it). Kronheimer's invariant (\ref{kronheimer}) gives us a homomorphism $Q_e : \pi_1 \mathcal{S}_0 (X, \omega ) \to \mathbb{Z}$, and we now explain how to modify $Q_e$ to produce a homomorphism $\pi_0 \mathrm{Symp}_0 (X, \omega ) \to \mathbb{Z}$ for suitable $(X, \omega )$ (see below).  Many of the ideas contained in this subsection are due to Smirnov \cite[\S 5]{Smirnov}, which we reformulate to tailor to our purposes.

\subsubsection{Family Seiberg--Witten invariants}

Let $\mathscr{X}\to B$ be a smooth fiber bundle with fiber a closed oriented smooth $4$-manifold $X$ with $b^{+}(X) >0$. For simplicity, we take $B = S^2$, which the case of interest to us. Fix the following topological data: a family spin-c structure $\mathfrak{t}$ on $\mathscr{X}\to B$ (by definition, this is a spin-c structure on its vertical tangent bundle) whose restriction $\mathfrak{t}|_X$ to the fiber has $-1 + b^1 (X) - b^+ (X) + \frac{1}{4}(c_1 (\mathfrak{t}|_X )^2 - \sigma(X )) = -2$; an element $\phi$, commonly referred to as a `chamber', belonging in the set $[B , S(H^+ (X, \mathbb{R})) ]$ of homotopy classes of maps from $B = S^2$ to the unit sphere in $H^+ (X, \mathbb{R} )$ (in particular, if $b^+(X) \neq 1$ or $3$ then there is a unique choice of $\phi$); and a homological orientation of $X$. Then there is an associated integer-valued \textit{family Seiberg--Witten invariant} of the smooth bundle $\mathscr{X} \to B$, denoted $FSW^{\phi} (\mathscr{X}/S^2 , \mathfrak{t} ) \in \mathbb{Z}$. For the definition we refer to \cite{li-liu,JLIN2022}. 

In the present case $B = S^2$ one also has that $FSW_\phi (\mathscr{X}/S^2  , \mathfrak{t} )$ only depends on the restriction $\mathfrak{t}|_X$ of $\mathfrak{t}$ to the fiber (\cite[Lemma 2.7]{JLIN2022}).

In our case of interest, $X$ is equipped with a symplectic form $\omega$. Then there is a `preferred chamber' $\phi_\omega \in [B , S(H^+(X,\mathbb{R}))]$ obtained as follows: choosing an $\omega$--compatible metric $g$ yields a maximal positive-definite subspace $\mathscr{H}_{g}^+ (X )$ of $H^2 (X, \mathbb{R} ) = \mathscr{H}_{g}^2 (X)$ and $0 \neq \omega \in \mathscr{H}^{+}_g (X)$; we then take $\phi_\omega$ to be the constant map $ B \to S(\mathscr{H}_{g}^+ (X) )$ with value $\omega$. In the present setting, we shall denote by $FSW_e$ the family Seiberg--Witten invariant in the chamber $\phi_\omega$ for any family spin-c structure $\mathfrak{t}$ whose restriction to the fiber $X$ is $\mathfrak{s}_\omega + e$. An isomorphism class of bundle $X \to \mathscr{X}\to S^2$ can be equivalently regarded---via the clutching construction---as an element in $\pi_1 \mathrm{Diff}_0 (X)$. Thus $FSW_e$ defines a map
\[
FSW_e : \pi_1 \mathrm{Diff}_0 (X) \to \mathbb{Z}
\]
which is also a group homomorphism.




The proof of the following result follows easily from the definitions and can be found in \cite[Proof of Lemma 3]{Smirnov}:

\begin{lemma}\label{lemma:FSW}Let $(X,\omega )$ be a closed symplectic $4$-manifold, and $e \in H^2 (X, \mathbb{Z} )$ a class satisfying (\ref{condition1}). Let $\omega_t \in \pi_1 \mathcal{S}_0 (X, \omega )$ be the image of an element $f_t \in \pi_1 \mathrm{Diff}_0 (X)$ under the map $\pi_1 \mathrm{Diff}_0 (X) \to \pi_1 \mathcal{S}_0 (X, \omega )$ induced by (\ref{fibration}). 
Fix a homological orientation of $X$. Then
\[
Q_{e} (\omega_t ) = FSW_{e} (f_t ) .
\]
\end{lemma}

\subsubsection{Charge-conjugation symmetry}

We refer to $(X, -\omega )$ as the \textit{conjugate} symplectic $4$-manifold of $(X, \omega )$. (Note that the orientation on the $4$-manifold $X$ induced by both $\pm \omega$ is the same). Observe that condition (\ref{condition1}) is preserved under replacing $e$ by $-e$ and $\omega$ by its conjugate $-\omega$, since $K_{-\omega} = - K_\omega$. In what follows, we shall canonically identify the following groups:
\begin{align}
\pi_1 \mathcal{S}_0 (X, \omega ) \cong \pi_1 \mathcal{S}_0 (X, -\omega ) \quad \omega_t \mapsto - \omega_t . \label{pi1Sconj}
\end{align}

\begin{lemma}\label{lemma:symplecticconjugation}
Let $(X,\omega )$ be a closed symplectic $4$-manifold, and $e \in H^2 (X, \mathbb{Z} )$ a class satisfying (\ref{condition1}). Then, under the identification (\ref{pi1Sconj}), we have:
\[
Q_{e}^{(X,\omega )} = -(-1)^{\frac{K_{\omega}^2 - \sigma(X) }{8}} Q_{-e}^{(X,- \omega )} .
\]
Suppose further that $K_\omega = 0$. Then 
\[
Q_{e}^{(X,\omega )} =  - Q_{-e}^{(X,-\omega )} .
\]
\end{lemma}
(Note: by the identity $K^2 = 2\chi(X) + 3 \sigma(X)$, which holds on any closed almost-complex $4$-manifold, we can rewrite $\frac{K^2 - \sigma(X) }{8}$ as the topological invariant $\frac{\chi(X) + \sigma(X)}{4}$ or as the Todd genus $\frac{K^2 + \chi (X)}{12}$.)

\begin{proof}
Given an almost Kähler pair $(g, \omega )$ on $X$, denote by $ ( S^{\omega}_{ e} , \rho^{\omega}_{e} )$ denote the spinor bundle, regarded as a Clifford bundle over $(X, g)$, corresponding to the spin-c structure $\mathfrak{s}_\omega + e$. The `charge-conjugation' (\cite[\S 6.8]{morgan}) provides a complex anti-linear isomorphism $c : S^{\omega}_{ e} \to S^{-\omega}_{-e} $ which preserves the Clifford multiplication and conjugates the hermitian metric:
\begin{align}
\rho^{-\omega}_{ -e}\circ c =c \circ \rho^{\omega}_{ e} \quad , \quad  c^\ast h^{-\omega}_{- e} = \overline{h^{\omega}_{ e}}.\label{cliffordconj}
\end{align}

Denote the universal Seiberg--Witten moduli space in the spin-c structure $\mathfrak{s}_\omega +e$ by $\mathscr{M}^{\omega}_{ e} \xrightarrow{\pi^{\omega}_{ e}} \mathscr{P}$. Note that, because $\mathfrak{s}_{-\omega} = \mathfrak{s}_\omega + K_\omega$ then $\mathfrak{s}_{\omega} + e \cong \mathfrak{s}_{-\omega} + e - K_\omega $ and thus the corresponding moduli spaces (and their determinant line bundles) are the same:
\[
\mathscr{M}^{\omega}_{ e} = \mathscr{M}^{-\omega}_{ e - K_\omega} .
\]
It is a basic fact that the perturbed Seiberg--Witten equations are invariant under simultaneously applying charge-conjugation and negating the perturbation, furnishing a diffeomorphism between the universal moduli spaces $c : \mathscr{M}^{\omega}_{ e} \to \mathscr{M}^{-\omega}_{ -e}$ defined by $(A, \Phi, g , \eta ) \to (c(A), c(\Phi ) , g , - \eta )$. Defining $c : \mathscr{P} \to \mathscr{P}$ to be the map $(g , \eta ) \mapsto (g , -\eta )$, we thus obtain a commutative diagram whose horizontal maps are diffeomorphisms:
\begin{equation}\label{conjugation-diagram}
\begin{tikzcd}
\mathscr{M}^{\omega}_{ e} \arrow{d}{\pi^{\omega}_{ e}}\arrow{r}{c} & \mathscr{M}^{-\omega}_{ -e} \arrow{d}{\pi^{-\omega}_{- e}}\\
\mathscr{P} \arrow{r}{c} & \mathscr{P}
\end{tikzcd}.
\end{equation}

Consider now a loop $\omega_t \in \pi_1 \mathcal{S}_0 (X, \omega )$. The construction in \S \ref{subsubsection:taubes} then yields a loop $\gamma^{\omega_t} : S^1 \to \mathscr{P}$ (avoiding the discriminant loci $\Sigma_e$). This was obtained by choosing a family $g_t$ of $\omega_t$--compatible metrics and a constant $r \gg 0$, and setting
\[
\gamma^{\omega_t} (t ) = (g_t , F_{\widehat{A}^{\omega_t}_{0,t}}^+ - r \sigma (\Phi_{0,t}^{\omega_t}) ).
\]
where $(A^{\omega_t}_{0,t} ,\Phi_{0,t}^{\omega_t} )$ denotes the canonical configuration associated to the almost-Kähler pair $(g_t , \omega_t )$ (cf. \S \ref{subsubsection:moduliboundary}). The pair $(A^{\omega_t}_{0,t} ,\Phi_{0,t}^{\omega_t} )$ is characterised uniquely up to gauge-equivalence by the following properties:
\[
| \Phi_{0,t}^{\omega_t}| = 1 \quad , \quad \rho_{0}^{\omega_t} (\omega ) \Phi_{0,t}^{\omega_t} = -2i \Phi_{0,t}^{\omega_t} \quad , \quad D_{A_{0,t}^{\omega_t}}\Phi_{0,t}^{\omega_t } = 0.
\]
By this and (\ref{cliffordconj}) it follows that $c$ takes $(A^{\omega_t}_{0,t} ,\Phi_{0,t}^{\omega_t} )$ to $(A^{-\omega_t}_{0,t} ,\Phi_{0,t}^{-\omega_t} )$. From this and the identities $F_{\widehat{c(A)}} = - F_{\widehat{A}}$, $ \sigma(c(\Phi )) = - \sigma (\Phi )$ we thus have $F_{\widehat{A}_{0,t}^{-\omega_t}} = - F_{\widehat{A}_{0,t}^{\omega_t} }$, $\sigma( \Phi_{0,t}^{-\omega_t} ) = - \sigma ( \Phi_{0,t}^{\omega_t} )$. That is, we have a commuting diagram:
\begin{equation}\label{loop-diagram}
\begin{tikzcd}
\mathscr{P} \arrow{rr}{c} & &  \mathscr{P}\\
& S^1 \arrow{lu}{\gamma^{\omega_t}} \arrow{ru}{\gamma^{-\omega_t}} &  
\end{tikzcd}.
\end{equation}

Combining (\ref{conjugation-diagram}) and (\ref{loop-diagram}) we thus see that the homomorphisms $Q_{e}^{(X,\omega)}$ and $Q_{-e}^{(X,-\omega)}$ agree up to an overall sign $(-1)^{\sigma_e}$ which we determine below.

By (\ref{conjugation-diagram}) there is an induced isomorphism (also denoted $c$) of the associated Quillen determinant line bundles of the Fredholm maps $\pi^{\omega}_{ e}$ and $\pi^{-\omega}_{ -e}$:
\begin{equation}\label{conjugation-det}
\begin{tikzcd}
\mathrm{det}\pi^{\omega}_{ e} \arrow{r}{c}\arrow{d} & \mathrm{det}\pi^{-\omega}_{- e} \arrow{d}\\
\mathscr{M}^{\omega}_{ e} \arrow{r}{c} & \mathscr{M}^{-\omega}_{ -e} 
\end{tikzcd}.
\end{equation}
The homological orientation of $X$ orients the Quillen determinants. By the argument in \cite[Proof of Corollary 6.8.4]{morgan}, it follows that the isomorphism of Quillen determinants (\ref{conjugation-det}) is orientation-preserving or orientation-reversing according as to whether the following mod-2 integer $\sigma_e \in \mathbb{Z}/2$ is $0$ or $1$:
\begin{align*}
\sigma_e = 1 + b^1 (X) + b^{+}(X) + \mathrm{ind}_{\mathbb{C}} D_A \quad \in \mathbb{Z}/2 
\end{align*}
where $\mathrm{ind}_{\mathbb{C}} D_A$ denotes the complex index of the Dirac operator in the spin-c structure $\mathfrak{s}_\omega + e$ (coupled to any spin-c connection $A$). Recall that $1+b^1 (X) + b^+ (X) = 0 \in \mathbb{Z}/2$ for any closed almost-complex $4$-manifold $X$. (Indeed, by a computation using the Hirzebruch Signature Theorem we have
\[
-1 + b^1 (X) - b^+ (X) + \frac{1}{4} ( K^2 - \sigma (X) ) = 0 ,
\]
and $\frac{1}{4} ( K^2 - \sigma (X) ) $ is even since by the Atiyah--Singer Index Theorem $\frac{1}{4} ( K^2 - \sigma (X) ) $ is twice the complex index of the Dirac operator for the canonical spin-c structure associated to the almost complex structure). 
On the other hand, $\mathrm{ind}_\mathbb{C} D_A$ can be computed by the Atiyah--Singer Index Theorem as
\[
\mathrm{ind}_{\mathbb{C}} D_A = \frac{1}{8} \big( c_1 ( \mathfrak{s}_{\omega} +e )^2 - \sigma (X)  \big) 
= -1 + \frac{K_{\omega}^2 - \sigma(X) }{8} .
\]
Putting all together we thus have:
\begin{align}
\sigma_e = 1 + \frac{K_{\omega}^2 - \sigma(X) }{8} \in \mathbb{Z}/2  .\label{sigmae}
\end{align}

In the case $K_\omega = 0 $ then $X$ is spin and Rokhlin's Theorem says that $\sigma(X)$ is divisible by $16$. Thus from (\ref{sigmae}) we have $\sigma_e = 1 \in \mathbb{Z}/2$ in this case. (Thus, the isomorphism (\ref{conjugation-det}) of Quillen determinants is \textit{orientation-reversing} when $K_\omega = 0$).
\end{proof}

\subsubsection{Smirnov's modification of Kronheimer's invariant}

We consider the following modification of Kronheimer's invariant, previously considered by Smirnov \cite{Smirnov} valued in $\mathbb{Z}/2$:

\begin{definition}\label{definition:qtilde}
Let $(X, \omega )$ be a closed symplectic $4$-manifold, and $e\in H^2 (X, \mathbb{Z} )$ a class satisfying $e^2 = -2$, $K_\omega \cdot e = 0$ and $[\omega ]\cdot e = 0$. (In particular, the class $\pm e$ satisfies condition (\ref{condition1}) in $(X,\omega )$; hence both invariants $Q_{\pm e}^{(X, \omega )}$ are defined). Define the homomorphism
\[
q_e  : \pi_1 \mathcal{S}_0 (X, \omega ) \to \mathbb{Z} \quad , \quad q_e = Q_{e}^{(X, \omega )} + Q_{-e}^{(X,\omega )} .
\]
\end{definition}

\begin{proposition}\label{proposition:q}
Let $(X, \omega )$ be a symplectic Calabi--Yau surface (i.e. $K_\omega = 0$) with $b^{+}(X) > 1$. Let $e \in H^2 (X, \mathbb{Z} )$ be a class satisfying $e^2 = -2 $ and $[\omega ]\cdot e = 0$. Then $q_e$ descends to a (unique) homomorphism from the smoothly trivial symplectic mapping class group:
\begin{center}
\begin{tikzcd}
(\pi_1 \mathcal{S}_0 (X, \omega ) )^{\mathrm{ab}} \arrow{r}{q_e} \arrow[twoheadrightarrow]{d}{(\ref{fibration})} &  \mathbb{Z}\\
(\pi_0 \mathrm{Symp}_0 (X, \omega ))^{\mathrm{ab}} \arrow{ru}{q_e} & \, 
\end{tikzcd}.
\end{center}
\end{proposition}
\begin{proof}
By (\ref{fibration}) we need to show that $q_e$ vanishes on the image of $\pi_1 \mathrm{Diff}(X) \to \pi_1 \mathcal{S}_0 (X, \omega )$. Letting $\omega_t$ be an element of the latter group given as the image of an element $f_t$ in the former, by Lemma \ref{lemma:FSW} we have
\[
Q_{e}^{(X,\omega )} (\omega_t ) = FSW_{e}^\omega (f_t )\quad , \quad Q_{e}^{(X,-\omega )}(-\omega_t  ) = FSW_{e}^{-\omega}(f_t ).
\]
The condition $K = 0$ is equivalent to $\mathfrak{s}_\omega \cong \mathfrak{s}_{-\omega}$; hence $\mathfrak{s}_\omega +e  \cong \mathfrak{s}_{-\omega} +e$. The chambers $\phi_\omega$ and $\phi_{-\omega}$ are the same if and only if $b^+ (X) >1$. Thus, under the stated assumptions on $(X,\omega )$ (i.e. $K_\omega = 0$ and $b^+ (X) >1$) the above shows:
\[
Q_{e}^{(X,\omega )} (\omega_t ) = Q_{e}^{(X,-\omega )}(-\omega_t  ).
\]
On the other hand, by Lemma \ref{lemma:symplecticconjugation} (the `charge-conjugation' symmetry) we have
\[
Q_{e}^{(X,-\omega)}(- \omega_t ) = - Q_{-e}^{(X, \omega )}(\omega_t ).
\]
All combined we obtain $Q^{(X,\omega )}_e (\omega_t ) + Q_{-e}^{(X, \omega )} (\omega_t ) = 0$, as required.
\end{proof}

The invariant $q_e$ on $\pi_0 \mathrm{Symp}_0 (X, \omega )$ is natural in the following sense:
\begin{proposition}\label{proposition:naturalityq}
Let $(X, \omega )$ be symplectic Calabi--Yau surface with $b^+ (X) >1$. Let $f  \in \pi_0 \mathrm{Symp}(X, \omega )$. Let $e \in H^2 (X  , \mathbb{Z} )$ be a class such that $e^2 - K_{\omega} \cdot e = -2$ and $[\omega ] \cdot e \leq 0$. Fix any homological orientation of $(X, \omega )$. Let $\sigma(f)  = 0$ or $1$ according as to whether $f^\ast$ preserves or reverses the orientation of $H^1 (X, \mathbb{R} ) \oplus H^+ (X ,\mathbb{R} )$. Then for any $\phi \in \pi_0 \mathrm{Symp}_0 (X, \omega )$ we have:
\begin{align*}
q_{e}^{(X , \omega )} ( f \phi f^{-1} ) = (-1)^{\sigma(f)}  \cdot q_{f^\ast e }^{(X, \omega)} ( \phi ) .
\end{align*}
In particular, if $f$ is a product of Dehn--Seidel twists on Lagrangian spheres embedded in $(X, \omega )$ then $\sigma(f) = 0$ in the above formula. 
\end{proposition}
\begin{proof}
This follows from Corollary \ref{corollary:naturalityQ}, together with the following observation: if a loop $\omega_t \in \pi_1 \mathcal{S}_0 (X, \omega )$ maps to $\phi \in \pi_0 \mathrm{Symp}_0 (X, \omega )$ under the connecting map $\delta : \pi_1 \mathcal{S}_{0}(X, \omega ) \to \pi_0 \mathrm{Symp}_0 (X, \omega )$ of the fibration (\ref{fibration}), then the loop $(f^{-1})^\ast \omega_t$ maps to $f \phi f^{-1}$. For how to see the latter, note that $\delta$ can be described as follows. For a loop $\omega_t \in \pi_1 \mathcal{S}_{0} (X, \omega )$ Moser's argument provides an isotopy $\phi_t$, $t \in [0,1]$, through diffeomorphisms with $\phi_0 = \mathrm{Id}$ and $\omega_t = \phi_{t}^\ast \omega$ for all $t \in [0,1]$. Thus $\phi := \phi_1$ is a smoothly trivial symplectomorphism, and by definition $\delta ( [\omega_t ] ) = \phi $. Since
\begin{align*}
(f^{-1} )^\ast \omega_t =(f^{-1} )^\ast \phi_{t}^\ast \omega = (f^{-1} )^\ast \phi_{t}^\ast f^\ast \omega =( f \phi_t f^{-1} )^\ast \omega 
\end{align*}
it follows that $\delta ( (f^{-1})^\ast \omega_t ) = f \phi_1  f^{-1} = f \phi f^{-1}$, as required.
\end{proof}

The next result describes the symmetries of the invariant $q_e$, which follow readily from its definition and Lemma \ref{lemma:symplecticconjugation}:

\begin{corollary}\label{corollary:symmetries}
Let $(X, \omega )$ be a symplectic Calabi--Yau surface (i.e. $K_\omega = 0$) with $b^{+}(X) >1$. Let $e \in H^2 (X, \mathbb{Z} )$ be a class satisfying $e^2 = -2 $ and $[\omega ]\cdot e = 0$. Then
\begin{enumerate}
\item $q_{e}^{(X,\omega )} = q_{-e}^{(X, \omega )}$
\item $q_{e}^{(X, \omega )} = - q_{e}^{(X, -\omega )}$.
\end{enumerate}
\end{corollary}

\begin{remark}
We will use the invariant $q_e$ to probe the non-triviality of the squared Dehn--Seidel twist element $\tau_{L}^2 \in \pi_0 \mathrm{Symp}_0 (X, \omega )$. 
For that we shall employ $e = \mathrm{PD}([L])$ where $[L]$ is the fundamental class of $L$ associated to a choice of orientation. From this point of view, in order for $q_e (\tau_{L}^2 )$ to be non-vanishing for said class $e$ the symmetries of $q_e$ described in Corollary \ref{corollary:symmetries} are a necessary requirement, as they are analogous to the symmetries of the Dehn--Seidel twist (cf. (\ref{symmetry1}) and Lemma \ref{lemma:symmetries}). 

\end{remark}

\subsection{An Excision Property}


The next result describes an \textit{excision property} of Kronheimer's invariant, which was suggested in \cite[\S 5]{kronheimer}. We will make extensive use of this property when performing calculations.

\begin{proposition}\label{proposition:excision}
Let $(X, \omega )$ be a closed symplectic $4$-manifold, and let $M \subset X$ be a codimension zero compact symplectic submanifold with convex boundary on a contact $3$-manifold $(\partial M  , \xi )$. Let $\iota_\ast : H^2 (M , \partial M  , \mathbb{Z}) \to H^2 (X , \mathbb{Z} )$ denote the natural map. Let $e \in H^2 (X, \mathbb{Z} )$ be a cohomology class such that $e^2 - K_\omega \cdot e = -2$, $[\omega ]\cdot e \leq 0$ and $e \in \mathrm{Im}\iota_\ast $. Let $\omega_t $ be a loop in $\mathcal{S}_0 (X, \omega )$ which equals $\omega$ on a neighborhood of $X \setminus \mathrm{Int}M$. Fix homological orientations of $X$ and the compact $4$-manifold $(M , \xi )$ with convex contact boundary. Then
\[
Q_{e}^{(X,\omega )} (\omega_t ) = \pm \sum_{\widetilde{e} \in (\iota_\ast )^{-1} (e) } Q_{\widetilde{e}}^{(M, \omega|_M ) } ( \omega_{t}|_M ).
\]
(Furthermore, the sign $\pm$ only depends on the homological orientations of $X$ and the compact $4$-manifold $(M,\xi )$ with convex contact boundary, but not on the class $e$ or the loop $\omega_t$.)
\end{proposition}

\begin{remark}
By standard results in Seiberg--Witten theory (see \cite[Theorem 5.2.4]{morgan}) it follows that, for a given loop $\omega_t$, only finitely many classes $\widetilde{e} \in (\iota_\ast )^{-1} (e)$ have $Q_{\widetilde{e}}^{(M , \omega|_M )} (\omega_{t}|_M) \neq 0$. Hence the right hand side in the formula from Proposition \ref{proposition:excision} is well-defined.
\end{remark}

Proposition \ref{proposition:excision} is a consequence of the parametric version of a Gluing Theorem by Mrowka and Rollin \cite[Theorem E]{mrowka-rollin} (see also \cite[Remark 2.2.3]{mrowka-rollin}). Briefly, one regards $Z^\prime := ( X \setminus \mathrm{Int}M , \omega  )$ as an \textit{asymptotically flat almost-Kähler end} (AFAK, for short) in the sense of \cite[Definition 2.1.2]{mrowka-rollin}, and then considers the \textit{non-compact} $4$-manifold $\widehat{M}$ obtained by replacing the AFAK end $Z^\prime$ by the non-compact AFAK end $Z$ given by the symplectisation of $(\partial M, \xi )$. Thus, we have decompositions
\[
X = M \cup Z^\prime  \quad , \quad \widehat{M} = M \cup Z .
\]
The techniques of \cite{mrowka-rollin}, straightforwardly generalised to parametric families, establish a diffeomorphism of the relevant moduli spaces on $X$ and $\widehat{M}$ (after a suitably large dilation of the AFAK end $(X \setminus \mathrm{Int}M , \omega )$, see \cite[\S 2]{mrowka-rollin}). 

We also note that the gluing theory of Mrowka--Rollin has been developed further in more recent work by Echeverria \cite{echeverria}.

\subsection{Calculations}

This subsection discusses calculations of Kronheimer's invariant in two simple situations involving a single Lagrangian sphere (`$A_1$--configuration') or two Lagrangian spheres which intersect at a single point only and transversely (`$A_2$--configuration'). For this we combine the Excision Property (Proposition \ref{proposition:excision}) with Kronheimer's calculation (Proposition \ref{proposition:kronheimercalculation} below). In particular, we will establish that the Dehn--Seidel twist $\tau_L$ always has infinite order for a symplectic $K3$ surface (Proposition \ref{proposition:calculation}).

\subsubsection{$A_1$ degenerations and simultaneous resolution}

What follows is a description of Kronheimer's calculation (\cite[\S 3-4]{kronheimer}) with some additions.\\

Consider a complex-analytic family of closed complex surfaces $X_t$ over the disk in the complex plane:
 \[
 X_t \to \mathscr{X} \xrightarrow{\pi} B(\mathbb{C} ) \ni t .
 \]
\begin{definition}\label{definition:A1}
We shall the family $\mathscr{X} \xrightarrow{\pi} B(\mathbb{C} )$ a \textit{projective $A_1$ degeneration} if it satisfies the following assumptions:
\begin{itemize}
\item $X_t$ is non-singular when $t \neq 0$.
\item $X_0$ has an ordinary double-point singularity (an $A_1$ singularity) at a point $x_0 \in X_0$. Namely, the germ of $(X_0 , x_0 )$ is analytically equivalent to $(\{z_{1}^2 + z_{2}^2 + z_{3}^2 = 0 \} , 0 )$). In addition, $X_0 \setminus \{ x_0 \}$ is non-singular.
\item The total space $\mathscr{X}$ is embedded in $\mathbb{C}P^N \times B(\mathbb{C})$ (for some $N >0$), with $\pi$ given by the restriction to $\mathscr{X}$ of the projection map $\mathbb{C}P^N \times B(\mathbb{C}) \to B(\mathbb{C} )$. (In particular, the fibers $X_t \subset \mathbb{C}P^N$ are projective complex  algebraic surfaces).
\item The germ of the map $\pi : (\mathscr{X} , x_0 ) \to (\mathbb{C} , 0 )$ is analytically equivalent to the semi-universal deformation of the ordinary double-point:
\[
(\{ z_{1}^2 + z_{2}^2 + z_{3}^2 = t \} , 0 ) \xrightarrow{t} ( \mathbb{C} , 0 ) .
\]
\end{itemize}
\end{definition}


Let $\pi^\prime : \mathscr{X}^\prime \to B(\mathbb{C} )$ be the family of complex surfaces $X_{u}^\prime = (\pi^\prime )^{-1}(u)$ over the disk $B(\mathbb{C} )\ni u$ given by the base change of $\pi : \mathscr{X} \to B(\mathbb{C} )$ by $u \mapsto t = u^2$. Thus $X_{u}^\prime = X_{u^2}$.

Consider then the \textit{simultaneous resolution} of the universal unfolding of the ordinary double-point (\cite{atiyah}):
\[
\begin{tikzcd}
\widetilde{\mathscr{X}} \arrow{rd}{\widetilde{\pi}} \arrow{rr}{\rho} &  &  \mathscr{X}^\prime \arrow{ld}{\pi^\prime} \\
& B(\mathbb{C} ) & 
\end{tikzcd}.
\]
In the analytic local model $(\mathscr{X}^\prime, 0 )  = ( \{ z_{1}^2 + z_{2}^2 + z_{3}^2 = u^2 \} , 0 ) \subset \mathbb{C}^3 \times B(\mathbb{C} )$ this can be described as follows. Use $\xi_- : \xi_+$ to denote homogeneous coordinates on $\mathbb{C}P^1$. Then $\widetilde{\mathscr{X}} \subset \mathbb{C}^3 \times \mathbb{C}P^1 \times B(\mathbb{C} )$ is given by the two equations
\begin{align*}
\xi_- \cdot  (z_1 + i z_2 ) = \xi_+ \cdot (u + z_3 ) \quad , \quad \xi_+ \cdot  (z_1 - i z_2 ) = \xi_- \cdot (u - z_3 )
\end{align*}
with $\widetilde{\pi}$ given by projection to $u$, and $\rho$ given by the projection to $(z_1 , z_2 , z_3 , u )$. 

One can easily check that $\widetilde{X}_u = X_{u}^\prime$ for $u \neq 0$ and $\widetilde{X}_0$ is now a non-singular surface, given by the blowing up of the singular point $x_0 \in X_{0}^\prime  = X_{0}$ to a rational curve $E_0 \subset \widetilde{X}_0$ with self-intersection number $E_0 \cdot E_0 = -2$. The family $\widetilde{\pi}:\widetilde{\mathscr{X}} \to B(\mathbb{C})$ is now a smooth fiber bundle over the disk and hence is diffeomorphic to a product bundle (in a homotopically canonical fashion): 
\[
\widetilde{\mathscr{X}} \cong X \times B(\mathbb{C}),
\]
where we denote $X := \widetilde{X}_1 = X^{\prime}_1 = X_1$.

We let $\omega_{\mathbb{C}P^N}$ be the Fubini--Study form on $\mathbb{C}P^N$; and for $u \neq 0$ we let $\omega_u$ be its restriction to $X_{u}^\prime \subset \mathbb{C}P^N$. Using the product structure of $\widetilde{\mathscr{X}}$, we can regard $\{\omega_u \}_{u \in \partial B(\mathbb{C} }$ as a loop of Kähler forms on the fixed manifold $X$ based at $\omega = \omega_1$, such that the symplectic class $[\omega_u ] \in H^2 (X, \mathbb{Z} )$ \textit{remains constant}; hence $\{ \omega_u \}_{u \in \partial B(\mathbb{C} )}$ defines an element of $\pi_1 \mathcal{S}_0 (X, \omega )$.

\begin{proposition}\label{proposition:kronheimercalculation}
In the above setting of a projective $A_1$ degeneration (cf. Definition \ref{definition:A1}), suppose $e \in H^2 (X, \mathbb{Z} )$ satisfies (\ref{condition1}) (i.e. $e^2 - K_\omega \cdot e = -2$ and $[\omega ] \cdot e \leq 0$). Fix a homological orientation of $X$. Then:
\[
Q_{e}^{(X,\omega )} (\{  \omega_u  \}_{u \in \partial B (\mathbb{C}) }) = 
\begin{cases} \pm 1 & \text{ if } e = PD( [E_0] )\\
0 & \text{ otherwise}
\end{cases}.
\]
\end{proposition}
\begin{proof}
Following Kronheimer \cite[Proof of Theorem 3.1]{kronheimer}, fix a family $\nu_u$ of Kähler forms on the fibers of $\widetilde{\mathscr{X}}$, and let $\beta (u)$ be a smooth function on the disk $B(\mathbb{C} )$ which is non-negative, equals one on a neighborhood of $u = 0$, and vanishes on a neighborhood of $\partial B(\mathbb{C} )$. Then consider the family $\widetilde{\omega}_u$ of Kähler forms on the fibers of $\widetilde{\mathscr{X}}$ given by 
\[
\widetilde{\omega}_u = \rho^\ast ( \omega_{\mathrm{\mathbb{C}P^N}}|_{X^{\prime}_u} ) + \beta(u) \cdot \nu_u \quad , \quad u \in B(\mathbb{C} ) .
\]

This family provides an extension to the disk of the loop $\{ \omega_u \}_{u \in \partial B(\mathbb{C})}$. Equipped with the natural family of compatible Kähler metrics, it gives rise to a family of perturbations parametrised by the disk $B(\mathbb{C})$ for the Seiberg--Witten equations on $X$ (using the product structure on $\widetilde{\mathscr{X}}$) via Taubes' perturbation (\ref{perturbation}). By standard arguments \cite[\S 7]{morgan}, solutions to the Seiberg--Witten equations on $X$ for perturbations in this family are in correspondence with effective divisors $E$ on the fibers $\widetilde{X}_u$, $u \in B(\mathbb{C} )$, such that $[E ] = PD ( e )$.

Consider first a non-central fiber $\widetilde{X}_u$, $u \neq 0$. We claim there are no such (non-empty) effective divisors $E$ contained in $\widetilde{X}_u = X_u$. Indeed, considering the family $\widetilde{\omega}_u$ with $\beta$ sufficiently concentrated around the origin so that $\beta (u) = 0$, we would have $[\widetilde{\omega}_u ] \cdot E = [\omega ]\cdot E \leq 0 $ (using (\ref{condition1})), which is impossible for an effective divisor $E$.

It remains to describe the effective divisors $E$ in the central fiber $\widetilde{X}_0$ in class $PD (e)$. We may write $E = nE_0 + E^\prime$ where $n \geq 0$ is the multiplicity of the rational curve $E_0$ as an irreducible component of $E$, and $E^\prime$ is an effective divisor on $\widetilde{X}_0$ which does not contain $E_0$ as an irreducible component. 
We have $[\omega]\cdot e \leq 0$ by (\ref{condition1}) and hence $n  [\omega] \cdot E_0  + [\omega ] \cdot E^\prime \leq 0$. Because $[\omega ]\cdot E_0 = 0$ then $[\omega ] \cdot E^\prime \leq 0$. The class $[\omega]$ is represented by the closed $2$-form $\omega_0 := \rho^\ast ( \omega_{\mathbb{C}P^N}|_{X_0 })$ on $\widetilde{X}_0$, and the symmetric tensor $g_0 := \omega_0 (\cdot , I_0 \cdot )$ on $\widetilde{X}_0$ (where $I_0$ is the complex structure on $\widetilde{X}_0$) is positive on $\widetilde{X}_0 \setminus E_0$. This, the fact that $E^\prime$ does not contain $E_0$ as an irreducible component, and $[\omega_0 ] \cdot E^\prime \leq 0$ imply that $E^\prime$ is empty. Hence $E = n E_0$. Since $e^2 - K_\omega \cdot e = -2$ (by (\ref{condition1})), $E_{0}\cdot E_0 = -2$ and $K_\omega \cdot E_0 = 0$, it follows that $n = \pm 1$. But $-E_0$ is not an efffective divisor, so $n = 1$.


Thus $E_0$ is the only effective divisor in a class $e$ satisfying (\ref{condition1}) contained in a fiber of the family $\widetilde{\mathscr{X}}$. In particular, $Q_e ( \omega_u ) = 0$ for $e \neq PD ([E_0])$. Kronheimer's calculation \cite[\S 4]{kronheimer} shows that $E_0$ is transversely cut-out as a solution to the Seiberg--Witten equations; hence $Q_{PD ([ E_0  ])} (\omega_u ) = \pm 1$.
\end{proof}

In the above situation, we also obtain a natural symplectic form $\Omega$ on $\mathscr{X} \subset \mathbb{C}P^N \times B(\mathbb{C} )$, obtained by restriction of $ \omega_{\mathbb{C}P^N } \oplus ( dx \wedge dy  )$. Using $\Omega$ we obtain a natural connection on the fiber bundle $\mathscr{X}\setminus X_0 \to B(\mathbb{C} )\setminus \{ 0\} $, whose horizontal subbundle is defined as the $\Omega$--orthogonal complement to the vertical subbundle $\mathrm{Ker}d\pi \subset T\mathscr{X} $. Associated to a path $\gamma$ in $B(\mathbb{C} )$ avoiding $0$, we have the diffeomorphism $P_\gamma : X_{\gamma(0)} \to X_{\gamma(1)}$ given by the parallel transport map using the connection. Clearly, $P_\gamma$ is a \textit{symplectomorphism} of $(X_{\gamma (0)} , \omega_{\gamma(0)})$ with $(X_{\gamma (1)} , \omega_{\gamma(1)}) $. Letting $c : [0,1) \to B(\mathbb{C} )$ be the path $c (t) = 1-t$, we obtain a Lagrangian sphere $L \subset ( X  , \omega ) = (X_1 , \omega_1 ) $ as the \textit{vanishing cycle} of $c$ \cite[\S 1.3]{seidel:exactseq}:
\[
L := \Big\{ x \in X_1 \, \, | \, \, \lim_{t \to 1} P_{[0, t]} ( x) = x_0 \Big\} .
\]
The fundamental class of $L$ is $[L] = [E_0 ] $. We have a \textit{monodromy} symplectomorphism $\psi \in \pi_0 \mathrm{Symp}(X, \omega )$ associated to the loop $\ell (t) = (\cos (t) , \sin (t) )$, given by $\psi := P_\ell $. By \cite[Proposition 1.15]{seidel:exactseq}, the monodromy agrees with the Dehn--Seidel twist along the vanishing cycle $L$:
\[
\psi = \tau_L \text{  in  } \pi_0 \mathrm{Symp}(X, \omega ).
\]

The monodromy squared $\psi^2 $ is the image of the loop $\{ \omega_u \}_{u \in \partial B(\mathbb{C} )}$ under the connecting map $\pi_1 \mathcal{S}_0 (X, \omega ) \to \pi_0 \mathrm{Symp}_0 (X, \omega )$ of the fibration (\ref{fibration}). Thus, from $\psi^2  = \tau_{L}^2$, \ Proposition \ref{proposition:kronheimercalculation} and Proposition \ref{proposition:q} we obtain:

\begin{corollary}\label{corollary:A1}
In the situation of Proposition \ref{proposition:kronheimercalculation}, assume further that $(X, \omega )$ is a symplectic Calabi--Yau surface with $b^{+}(X) >1$. Let $e \in H^2 (X, \mathbb{Z} )$ be a class satisfying $e^2 = -2$ and $[\omega ]\cdot e = 0$. If $L$ is the Lagrangian sphere given by the vanishing cycle of $c$, then
\[
q_e ( \tau_{L}^2 ) = \begin{cases} \pm 1 & \text{  if  } e = \mathrm{PD} ([E_0 ]) \text{ or } -\mathrm{PD}([E_0 ])\\
0 & \text{otherwise}
\end{cases}.
\]
\end{corollary}


\subsubsection{$A_1$ and $A_2$ configurations}

\begin{proposition}\label{proposition:calculation}
Let $(X, \omega )$ be a closed symplectic $4$-manifold
, and $L \subset (X, \omega )$ an embedded and oriented Lagrangian sphere. Let $e = \mathrm{PD}([L])$. Fix a homological orientation of $X$. Then
\begin{align*}
 & Q_{ e} (\mathcal{O}_{ L} ) = Q_{-e}(\mathcal{O}_{-L}) = \pm 1 \\
 & Q_{ -e} (\mathcal{O}_{ L} ) = Q_{ e} (\mathcal{O}_{ -L} ) = 0 . 
 \end{align*}
From this, it follows that:
\begin{enumerate}
\item The elements $\mathcal{O}_{\pm L}$ induce a summand $\mathbb{Z}^2 \subset (\pi_1 \mathcal{S}_0 (X, \omega ))^{\mathrm{ab}}$.
\item Suppose that $(X, \omega )$ is a symplectic Calabi--Yau surface with $b^+ (X) >1$. Then $q_e ( \tau_{L}^2 ) = \pm 1$. Thus $\tau_{L}^2$ has infinite order in $(\pi_0 \mathrm{Symp}_0 (X, \omega ))^{\mathrm{ab}}$.
\end{enumerate}
\end{proposition}
\begin{proof}
Each loop $\mathcal{O}_{\pm L}$ bounds a $D^2$--family of symplectic forms whose cohomology class pairs non-positively with $\mp e$ (Lemma \ref{lemma:kerO}).
Thus, by Lemma \ref{lemma:taubes}, choosing any $D^2$--family of Riemannian metrics compatible with the symplectic forms and taking $r \gg 1$, we see that the moduli spaces contributing to the counts $Q_{\mp e }(\mathcal{O}_{\pm L} )$ are empty, and thus $Q_{\mp e }(\mathcal{O}_{\pm L} )=0$.

On the other hand, $\tau_L$ takes $L$ to $-L$ and $e$ to $-e$, and thus by the naturality of Kronheimer's invariant (Corollary \ref{corollary:naturalityQ}) we have $Q_e ( \mathcal{O}_L ) = Q_{-e} (\mathcal{O}_{-L} )$. Thus, it only remains to calculate $Q_e ( \mathcal{O}_L )$.

A compact neighborhood of $L$ in $(X, \omega )$ is symplectomorphic to the disk cotangent bundle of $L \cong S^2$, $(M, \omega) := (D^\ast S^2 , \omega )$. The latter is equipped with the canonical symplectic form and has convex contact boundary diffeomorphic to $\mathbb{R}P^3$. Since $b^1 (\mathbb{R}P^3 )=0$ then the map $\iota_\ast : H^2 (M , \partial M , \mathbb{Z} ) \to H^2 (X, \mathbb{Z} )$ is injective. By the Excision property (Proposition \ref{proposition:excision}) we then have $Q_{e}^{(X , \omega )}(\mathcal{O}_L ) = \pm Q_{e}^{(M,\omega )}(\mathcal{O}_L )$, from which we see that $Q_{e}^{(X, \omega )}(\mathcal{O}_L )$ is, up to an overall sign, independent of the ambient $(X, \omega )$. It therefore suffices to choose any particular $(X, \omega )$ that is most suitable, 
and we can certainly find $(X, \omega )$ which is a symplectic $K3$ surface arising as the smooth fiber of a projective $A_1$ degeneration (Definition \ref{definition:A1}) with vanishing cycle $L$. By $Q_{-e}(\mathcal{O}_L ) = 0 $ and Corollary \ref{corollary:A1} we thus have
\begin{align*}
Q_e (\mathcal{O}_L ) = q_e (\tau_{L}^2 ) = \pm 1 \, ,
\end{align*}
and the proof is complete.
\end{proof}

\begin{proposition}\label{proposition:A2}
Let $(X, \omega )$ be a closed symplectic $4$-manifold
, and $L_1 , L_2 \subset (X, \omega )$ two embedded Lagrangian spheres which intersect at a single point only with transverse intersection (an `$A_2$--configuration' of Lagrangian spheres). Fix orientations on each $L_i$, and let $e_i = \mathrm{PD}([L_i])$. Fix a homological orientation of $X$. Then:
\begin{align*}
& Q_{\pm e_2}(\mathcal{O}_{L_1}) = Q_{\pm e_2}(\mathcal{O}_{- L_1}) = 0\\
& Q_{\pm e_1}(\mathcal{O}_{L_2}) = Q_{\pm e_1}(\mathcal{O}_{- L_2}) = 0 .
\end{align*}
\end{proposition}
\begin{proof}
Without loss of generality, we suppose that the unique intersection point of $L_1$ and $L_2$ is positive, so $e_1 \cdot e_2 = 1 $. By the symmetries described in Lemma \ref{lemma:symmetriesO} and Lemma \ref{lemma:symplecticconjugation}, it suffices to show $Q_{\pm e_2}(\mathcal{O}_{L_1}) = 0$. 

We first show $Q_{e_2}( \mathcal{O}_{L_1}) = 0$. By Lemma \ref{lemma:kerO} the loop $\mathcal{O}_{L_1}$ bounds a disk $D^2$ in $\mathcal{S}_L (X, \omega )$. A closer examination of the proof of Lemma \ref{lemma:kerO} shows that the cohomology class of each symplectic form in this $D^2$--family is of the form $[\omega ] - c e_1$ for a non-negative number $c$ (which varies with the symplectic form in the disk). Thus, 
\[
([\omega] - ce_1)\cdot e_2 = 0 - c \leq 0 
\]
which by Lemma \ref{lemma:taubes} implies that (for any $D^2$--family of Riemannian metrics compatible with the symplectic forms, and taking $r \gg 1$) the moduli spaces relevant to the counts in $Q_{e_2}(\mathcal{O}_{L_1})$ are empty; hence $Q_{e_2}(\mathcal{O}_{L_1}) = 0$. 

We now show $Q_{-e_2}(\mathcal{O}_{L_1}) = 0$. The configuration of spheres $L_1 \cup L_2 \subset X$ admits a compact neighborhood with convex contact boundary $(M, \omega ) \subset (X, \omega )$ symplectomorphic to the \textit{plumbing} of two disk cotangent bundles of spheres (\cite[\S 7.6]{geiges-book}). The boundary $\partial M \cong L(3,2)$ is a rational homology sphere, and thus $H^2 (M , \partial M , \mathbb{Z} ) \to H^2 (X)$ is injective. Applying the Excision property (Proposition \ref{proposition:excision}) we have $Q_{-e_2}^{X} (\mathcal{O}_{L_1}) = \pm Q_{-e_2}^M (\mathcal{O}_{L_1} )$. Thus, again $Q_{-e_2}^{X}(\mathcal{O}_{L_1})$ is independent of the ambient $(X, \omega )$ up to an overall sign. Therefore, it suffices to assume that $(X, \omega )$ is any particular closed symplectic $4$-manifold of our choice, containing an $A_2$ configuration $L_1 , L_2$. Thus, we suppose that $(X,\omega )$ is a symplectic $K3$ surface which arises as a fiber in a projective $A_1$ degeneration (Definition \ref{definition:A1}) with $L_1$ Lagrangian isotopic to the vanishing cycle, and $L_2$ is another Lagrangian sphere meeting $L_1$ at a single point only and transversely (examples of such abound). By Proposition \ref{corollary:A1}, since $e_1 \neq \pm e_2$ we thus have
\[
0 = q_{e_2}(\tau_{L_1}^2) = Q_{e_2} (\mathcal{O}_{L_1 }) + Q_{-e_2} (\mathcal{O}_{L_1})
\]
and since we already proved $Q_{e_2} (\mathcal{O}_{L_1 }) =0$ then this implies $Q_{-e_2} (\mathcal{O}_{L_1 }) = 0$.
\end{proof}

\begin{remark}
Proposition \ref{proposition:A2} will later be `overrun' by the more general Proposition \ref{proposition:switching}, whose proof makes use of a different gluing result in Seiberg--Witten theory: the `Family Switching Formula' (\cite[Theorem 5.3]{JLIN2022} or \cite{liu_switching}). Proposition \ref{proposition:calculation} may also be proved by appealing instead to J. Lin's calculation from \cite[Proposition 8.2]{JLIN2022} based on the Family Switching Formula, using an argument similar to the proof of Proposition \ref{proposition:switching}.
\end{remark}

\section{$ADE$ configurations }\label{section:ADE}

\subsection{Recollections on Generalised Braid groups}

In this subsection we first discuss background material on generalised Braid groups of $ADE$ type. We then discuss in more details the abelianisation of pure braid groups, from various points of view that will be relevant to us (geometric, group-theoretic, representation-theoretic). The reader is assumed to have some familiarity with basic facts about root systems of $ADE$ type (\cite{humphreys-lie,humphreys-coxeter}).

\subsubsection{Weyl groups and root systems of ADE type}\label{subsection:rootsystems}

Let $\Gamma$ be a Dynkin diagram of type $ADE$. Its vertex set is denoted by $I$.

\begin{definition}
The \textit{Weyl group} $W = W(\Gamma )$ associated to $\Gamma$ is the group given by the presentation
\begin{align}
W = \Big\langle \, s_i , \, i \in I \, | \, s_{i}^2 = 1 \, \text{ and } \begin{cases} s_i s_j s_i = s_j s_i s_j & \text{if } i \text{ and } j \text{ are adjacent in } \Gamma \\
s_i s_j = s_i s_j & \text{otherwise}
\end{cases}
\Big\rangle .\label{weyl}
\end{align}
\end{definition}

Consider free $\mathbb{Z}$--module $V_\mathbb{Z} = \mathbb{Z}^I$ with a generator $e_i$ for each vertex $i \in I$. For a given field $\mathbb{F}$ (those relevant to us are $\mathbb{F} = \mathbb{R}$ or $\mathbb{C}$), we will also consider the associated $\mathbb{F}$--vector space $V_\mathbb{F} = V_{\mathbb{Z}}\otimes_{\mathbb{Z}} \mathbb{F}$. We equip $V_\mathbb{R}$ with a \textit{negative-definite} inner product denoted by $v\cdot w$ for a pair $v,w \in V$, and defined by 
\begin{align}
e_i \cdot e_j := \begin{cases}
-2 & \text{if } i = j\\
1 & \text{if } i \text{ and } j \text{ are adjacent in } \Gamma\\
0 & \text{otherwise}.
\end{cases}\label{ADElattice}
\end{align}

The Weyl group $W$ acts on $V_\mathbb{R}$ preserving the lattice $V_\mathbb{Z}$: for each $i \in I$ the generator $s_i \in W$ acts on $v \in V_\mathbb{R}$ as the \textit{reflection} through the hyperplane normal to $e_i$,
\begin{align}
s_i \cdot v := v + (v \cdot e_i ) e_i .\label{picard-lefschetz}
\end{align}

The following result is standard, but we didn't find a convenient reference:

\begin{lemma}\label{lemma:rootADE}
Let $\Phi \subset V_\mathbb{R}$ be the subset given by 
\[
\Phi = \Big\{ \alpha \in V_\mathbb{Z} \, | \, \alpha\cdot \alpha = -2 \Big\} . 
\]
Then:
\begin{enumerate}
    \item $\Phi$ is finite
    \item Every $\alpha \in \Phi$ is (uniquely) expressed as $\alpha = \sum_{i \in I} k_i e_i$ either with $k_i \in \mathbb{Z}_{\geq 0}$ for all $i\in I$, or with $k_i \in \mathbb{Z}_{\leq 0}$ for all $i\in I$
    \item $\Phi$ is preserved by the $W$--action on $V_\mathbb{R}$.
    \end{enumerate}
\end{lemma}
\begin{proof}
(1) follows from the fact that the inner product $v\cdot w$ on $V$ is negative-definite. (3) follows from the observation that $W$ acts by isometries on $V_\mathbb{R}$. For (2), let $\alpha = \sum_{i \in I }k_i e_i \in \Phi$ and set $I_{\pm} = \{ i\in I \, | \, \pm k_i >0 \}$. Observe that
\begin{align}
-2  = \alpha^2 = (\sum_{i \in I_+} k_i e_i )^2 + (\sum_{j \in I_-}k_j e_j )^2 + 2 \sum_{i \in I_+ , j \in I_-} k_i k_j e_i \cdot e_j  .\label{alpha2}
\end{align}
If both $I_+ \neq \emptyset$, $I_- \neq \emptyset$ then the first and second terms on the right-hand side of (\ref{alpha2}) are negative even integers because the quadratic form on $V_\mathbb{Z}$ defined by (\ref{ADElattice}) is even and negative-definite. The third term is non-positive (since $k_i k_j <0$ and $e_i \cdot e_j= 0$ or $1$). Thus, (\ref{alpha2}) yields the contradiction $-2 \leq -2  -2 + 0$; hence (2) follows.
\end{proof}

Lemma \ref{lemma:rootADE}(1) implies that $\Phi$ is a (reduced, crystallographic) \textit{root system} in the inner product space $V_\mathbb{R}$ \cite[\S 9]{humphreys-lie}. Lemma \ref{lemma:rootADE}(2) shows that the collection $\Delta = \{ e_i \, | \, i\in I\}$ forms a \textit{basis} of the root system $\Phi$ \cite[\S 10]{humphreys-lie}, whose associated Dynkin diagram \cite[\S 11]{humphreys-lie} is $\Gamma$.  
From now on we refer to $\Phi$ as the \textit{root system of type $\Gamma$}. The elements of $\Phi$ will be called \textit{roots}, and the elements of $\Delta = \{ e_i \, | \, i \in I\} $ will be called \textit{simple roots} of $\Phi$. By Lemma \ref{lemma:rootADE}(2), we have a decomposition of the set of roots $\Phi = \Phi_+ \sqcup \Phi_-$ according to the sign of the integer coefficients $k_i$; elements of $\Phi_{+}$ (resp. $\Phi_-$) are called \textit{positive roots} (resp. negative roots) of $\Phi$.

It is easy to see that $W$ acts by isometries on $V_\mathbb{R}$. By \cite[Theorem 1.5, Theorem 1.9]{humphreys-coxeter} $W$ acts \textit{faithfully} on $V_\mathbb{R}$; hence $W$ is naturally a subgroup of the orthogonal group $O(V_\mathbb{R})$, generated by the reflections through the hyperplanes normal to $e_i$. Thus, by Lemma \ref{lemma:rootADE}(3) $W$ is also a subgroup of the permutation group on the finite set $\Phi$; hence $W$ is \textit{finite}.

For future reference, we include here the following standard result,

\begin{lemma}\label{lemma:positive+simple}
Let $\alpha \in \Phi_+$ be a positive root. If $\alpha$ is not simple, then there exists a positive root $\beta \in \Phi_+$ and a simple root $e_i$ such that $\alpha = \beta + e_i$. For any such $\beta$ and $e_i$ then $s_i \cdot \alpha = \beta$.
\end{lemma}
\begin{proof}
The first assertion is proved in \cite[\S 10.2, Lemma A]{humphreys-lie}. By (\ref{picard-lefschetz}), the second assertion is equivalent to $\beta  \cdot e_i =1$. To show this, suppose first that $\beta \cdot e_i <0$. Then $\alpha \cdot e_i = \beta \cdot e_i + e_{i}^2 < -2 $. The latter is impossible in the $ADE$ root system, where the inner product of any two roots equals $0, \pm1$ or $\pm2$ \cite[\S 9.4, Table 1]{humphreys-lie}. If $\beta \cdot e_i = 0$ then $\alpha^2 = \beta^2 + e_{i}^2 = -4$, which contradicts the fact that $\alpha$ is a root. Thus we have $\beta \cdot e_i >0$. By \cite[\S 10.2, Lemma A]{humphreys-lie}, then $\beta \cdot e_i = 1$ is the only possibility in an $ADE$ root system.
\end{proof}

\subsubsection{Braid groups}

Let $\Gamma$ be an ADE Dynkin diagram. 

\begin{definition}
The \textit{generalised Braid group of type $\Gamma$} (or `Artin--Tits group of type $\Gamma$') \cite{brieskorn-braid,deligne,brieskorn-tresses} is the group $B= B(\Gamma)$ given by the presentation
\begin{align}
B = \Big\langle \, s_i , \, i \in I \, | \, \begin{cases} s_i s_j s_i = s_j s_i s_j & \text{if } i \text{ and } j \text{ are adjacent in } \Gamma \\
s_i s_j = s_i s_j & \text{otherwise}
\end{cases}
\Big\rangle . \label{braidgroup}
\end{align}
(When $\Gamma$ is the $A_n$ Dynkin diagram then $B$ is Artin's classical Braid group on $n+1$ strands; a standard reference on this is \cite{farb-margalit}). There is a canonical surjective homomorphism $B \to W$ whose kernel is---by definition---the \textit{generalised pure Braid group of type $\Gamma$}, which we denote $P = P(\Gamma)$. Thus, $P$ is the subgroup of $B$ normally generated by the elements $s_{i}^2$ with $i \in I$. (For a presentation of $P$, see \cite{digne-gomi}). 
\end{definition}

Thus, $B$ gives an extension of $W$ by $P$: we have a short exact sequence of groups
\begin{align}
1 \to P \to B \to W \to 1 . \label{SES1}
\end{align}
By a fundamental result of Brieskorn \cite{brieskorn-braid} and Deligne \cite{deligne}, the extension (\ref{SES1}) can be interpreted geometrically as follows. Associated to each positive root $\alpha \in \Phi_+$ there is the hyperplane $H(\alpha) \subset V_\mathbb{R}$ normal to $\alpha$, and a corresponding complex hyperplane $H_{\mathbb{C}}(\alpha) := H(\alpha) \otimes \mathbb{C}$ in $V_\mathbb{C}$. The $W$--action on $V_\mathbb{C}$ permutes the hyperplanes, and the map from $V_\mathbb{C} \setminus \bigcup_{\alpha \in \Phi_+} H_\mathbb{C} (\alpha )$ to its quotient by $W$ is a regular covering with deck group $W$. This gives an associated short exact sequence:
\begin{align}
1 \to \pi_ 1 \Big(  V_\mathbb{C} \setminus \bigcup_{\alpha \in \Phi_+} H_\mathbb{C} (\alpha ) \, , \,  x_0  \Big) \to \pi_ 1 \Big(\, \big( V_\mathbb{C} \setminus \bigcup_{\alpha \in \Phi_+} H_\mathbb{C} (\alpha ) \big) / W \,  , \,   [x_0 ] \Big)\to W \to 1 , \label{SES2}
\end{align}
for any choice of basepoint $x_0$ avoiding the hyperplanes. From now on, the basepoint $x_0$ will be chosen to lie in the real locus $V_\mathbb{R} \subset V_\mathbb{C}$ .

\begin{theorem}[Brieskorn--Deligne] \label{theorem:hyperplane}
There is an isomorphism of the short exact sequences (\ref{SES1}) and (\ref{SES2}), inducing the identity map on $W$.
\end{theorem}

The isomorphism is canonical. The generator $s_i \in B$, $i \in I$, corresponds to the following homotopy class of loop $\gamma_i$ based at $[x_0 ]$ in $\big( V_\mathbb{C} \setminus \bigcup_{\alpha \in \Phi_+} H_\mathbb{C} (\alpha ) \big) / W$. The line $\ell$ containing $x_0 \in V_\mathbb{R} $ and $s_i (x_0 )\in V_\mathbb{R}$ meets the hyperplane $H (e_i )$ transversely at a single point $x_i$; and the segment of $\ell$ joining $x_0$ and $s_i (x_0 )$ does not meet any other hyperplanes $H(\alpha )$ for $\alpha \neq e_i$. The loop $\gamma_i$ is represented by the path in $V_\mathbb{C} \setminus \bigcup_{\alpha \in \Phi_+} H_\mathbb{C} (\alpha )$ from $x_0$ to $s_i (x_0 )$ which follows $\ell$ from $x_0$ to $s_i (x_0)$ except that, just before arriving to $x_i$, it avoids $x_i$ by instead doing a positive half-rotation about $x_i$ in the complex line $\ell \otimes \mathbb{C}$.


\subsubsection{Abelianisation of Braid groups}

We now describe the abelianisation of the pure Braid group $P = P(\Gamma )$ associated to an $ADE$ Dynkin diagram $\Gamma$. It should be noted that the abelianisation of the Braid group $B = B(\Gamma )$ is essentially uninteresting: we have an isomorphism $B^{\mathrm{ab}} \cong \mathbb{Z}$, with a generator given by any $s_i$. In turn, as we will explain momentarily, the pure Braid group $P$ has a large abelianisation.

First, observe that the Weyl group $W$ acts by automorphisms on the abelian group $P^{\mathrm{ab}}$. Namely, given $w \in W$ and $\beta \in P$, we choose some lift $\tilde{w} \in B$ of $w \in W$ (cf. (\ref{SES1}) and consider $\tilde{w} \beta \tilde{w}^{-1} \in P $. It is clear that we have a well-defined action
\begin{align}
W \times P^{\mathrm{ab}} \to P^{\mathrm{ab}} \quad (w , [\beta]) \to [ \tilde{w} \beta \tilde{w}^{-1} ].\label{actionPab}
\end{align}
(More generally, there is a well-defined action of $W$ on the group $B / [P , P]$ where $[P , P]$ is the commutator subgroup of $P$, defined in a similar fashion).

Returning to the hyperplane arrangement complement $V_\mathbb{C} \setminus \bigcup_{\alpha \in \Phi_+ } H_{\mathbb{C}}(\alpha )$, it is an elementary observation that the abelianisation of its fundamental group---the singular homology with $\mathbb{Z}$ coefficients---is the free abelian group with one generator for each hyperplane:
\begin{align}
H_1 \Big(  V_\mathbb{C} \setminus \bigcup_{\alpha \in \Phi_+ } H_{\mathbb{C}}(\alpha ) , \mathbb{Z} \Big) =   \bigoplus_{\alpha \in \Phi_+} \mathbb{Z} . \label{homologyhyperplane}
\end{align}
The generator $T_\alpha$ of this group associated to the hyperplane $H_\mathbb{C} (\alpha)$ is a small loop in $V_\mathbb{C}$ linking once and positively with the hyperplane $H_{\mathbb{C}}(\alpha )$. (This much is true for an arbitrary complex hyperplane arrangement.) 


The Weyl group $W$ acts on $V_\mathbb{C}$ preserving the hyperplane arrangement, hence acts on the homology group (\ref{homologyhyperplane}). This action
\begin{align}
W \times  \bigoplus_{\alpha \in \Phi_+} \mathbb{Z} \to \bigoplus_{\alpha \in \Phi_+} \mathbb{Z} \label{actionpositive} \end{align}
can be described explicitly: for a standard generator $s_i \in W$ and a positive root $\alpha \in \Phi_+$, we have
\[
s_i \cdot t_\alpha = \begin{cases} t_{s_i \cdot \alpha } & \text{ if } \alpha \neq e_i \\
t_{e_i} & \text{ if } \alpha = e_i 
\end{cases}.
\]
(In the above formula, note that if $\alpha \in \Phi_+$ and $\alpha \neq e_i$ then indeed $s_i \cdot \alpha \in \Phi_+ $. This is easily verified, e.g. see \cite[\S 10.2, Lemma B]{humphreys-lie}.)


With this in place, Theorem \ref{theorem:hyperplane} therefore yields an isomorphism $P^{\mathrm{ab}} \cong \bigoplus_{\alpha \in \Phi_+} \mathbb{Z}$. Most succinctly, this can be characterised as follows:

\begin{proposition}\label{proposition:characterisation}
With respect to the $W$--actions (\ref{actionPab}) and (\ref{actionpositive}), there exists a unique $W$--equivariant isomorphism of abelian groups $P^{\mathrm{ab}} \cong \bigoplus_{\alpha \in \Phi_+} \mathbb{Z}$ sending $[ s_{i}^2 ] \in P^{\mathrm{ab}}$ to $T_{e_i} \in \bigoplus_{\alpha \in \Phi_+} \mathbb{Z}$ for each $i \in I$.
\end{proposition}
\begin{proof}
The existence readily follows from Theorem \ref{theorem:hyperplane} and the paragraph that follows it. Uniqueness follows from: (1) the fact that $P^{ab}$ is generated as an abelian group by the $W$--orbits of the elements $[s_{i}^2 ]$ (because $P$ is the subgroup of $B$ normally generated by the $s_{i}^2$); and (2) the fact that for every $\alpha \in \Phi_+$ there exists $w \in W$ such that $w \cdot \alpha = e_i$ for some $i \in I$ \cite[\S 10.3, Theorem(c)]{humphreys-lie}. (For a root system of type $ADE$ more is true: $W$ acts transitively on $\Phi$ \cite[\S 10.4, Lemma C]{humphreys-lie}).
\end{proof}

The $W$--representation structure on $P^{\mathrm{ab}} = \bigoplus_{\alpha \in \Phi_+} \mathbb{Z}$ given by (\ref{actionpositive}) (equivalently, by (\ref{actionPab})) arises in a natural way as a sub-representation of a larger one. Consider the free abelian group generated by all the roots: 
\[
\mathbb{Z}\Phi := \bigoplus_{\alpha \in \Phi} \mathbb{Z},
\]
with one generator $T_\alpha$ for each root $\alpha$. The Weyl group $W$ naturally acts on it by automorphisms by permuting the roots: $s_i \cdot T_\alpha = T_{s_i \cdot \alpha }$. As a $W$--representation, $P^{\mathrm{ab}}$ is isomorphic to the sub-representation of $\mathbb{Z}\Phi$ generated by the elements $t_\alpha : = T_\alpha + T_{-\alpha}.$
There is another $W$--representation structure that the abelian group $\bigoplus_{\alpha \in \Phi_+}\mathbb{Z}$ can be endowed with, which we denote $\overline{P}^{\mathrm{ab}}$ and refer to as the \textit{conjugate} $W$--action, which is given by the sub-representation of $\mathbb{Z}\Phi$ generated by the elements $\overline{t}_\alpha : = T_\alpha - T_{-\alpha},$
which satisfy
\[
s_i \cdot \overline{t}_\alpha  = \begin{cases} \overline{t}_{s_i \cdot \alpha } & \text{ if } \alpha \neq e_i \\
- \overline{t}_{e_i } & \text{ if } \alpha = e_i 
\end{cases}.
\]

The three $W$--representations described above fit into a short exact sequence of $W$--representations:
\begin{align}
0 \to \overline{P}^{\mathrm{ab}}\to \mathbb{Z}\Phi \to P^{\mathrm{ab}}\to 0 .\label{SESreps}
\end{align}
(The first map is $\overline{t}_\alpha \mapsto T_\alpha - T_{-\alpha}$, and the last map is $T_\alpha \mapsto t_{\pm \alpha}$ for $\alpha \in \Phi_{\pm }$.) As a short exact sequence of abelian groups, (\ref{SESreps}) is split. 
\begin{lemma}\label{lemma:nonsplit}
As a sequence of $W$--representations, we have that (\ref{SESreps}) is:
\begin{enumerate}
\item non-split
\item split after localising away from $2$ (i.e. inverting $2$). 
\end{enumerate}
\end{lemma}
\begin{proof}
(1) Suppose that a splitting existed, and denote it $s : P^{\mathrm{ab}}\to \mathbb{Z}\Phi$. That is: $s$ is $W$--equivariant, and the composition $P^{\mathrm{ab}}\xrightarrow{s} \mathbb{Z}\Phi \to P^{\mathrm{ab}}$ is the identity. Writing $s(t_\alpha ) = \sum_{\beta \in \Phi } s_{\alpha , \beta } T_\beta$ with $s_{\alpha, \beta }\in \mathbb{Z}$, the second property of $s$ says that 
\[
s_{\alpha , \alpha} + s_{\alpha , - \alpha } = 1 \quad , \quad s_{\alpha , \beta} + s_{\alpha , - \beta} = 0 \text{ if } \beta \neq \alpha .
\]
But the equivariance condition implies that $s_{\alpha , \alpha } = s_{\alpha , - \alpha }$, which is a contradiction with the above. (2) After localising away from $2$ we have a splitting $ s: P^{\mathrm{ab}}[\frac{1}{2}] \to \mathbb{Z}\Phi [\frac{1}{2}]$ sending $t_\alpha$ to $\frac{1}{2}(T_\alpha + T_{-\alpha })$.
 \end{proof}

\subsubsection{Example: the case $\Gamma = A_n$}\label{subsubsection:exampleA}

For the sake of concreteness, we briefly spell out here how the previous discussion looks like when $\Gamma$ is the $A_n$ Dynkin diagram ($n\geq 1$). Then $W \cong S_{n+1}$ is the symmetric group on $n+1$ elements $\{ 1 , 2 , \ldots , n+1\}$, and the generators $s_i$ ($i = 1 , \ldots , n$) from the presentation (\ref{weyl}) can be identified with the transposition $s_i = ( i  ,  i +1 )$. Now $B \cong B_{n+1}$ is the `classical' Artin Braid group on $n+1$ strands, and the generator $s_i$ ($i = 1 , \ldots , n$) from the presentation (\ref{braidgroup}) can be identified with the braid $s_i$ where the $(i+1)$st strand passes in front of the $i$st strand; and $P = P_{n+1}$ is the `classical' pure Braid group, consisting of braids which induce the trivial permutation of the strands. (For details, see e.g. \cite{farb-margalit}). 

The set of positive roots consists of $n+1\choose2$ elements:
\[
\Phi_+ = \Big\{ e_{i}+e_{i+1} + \cdots + e_{j-1} + e_{j} \, \, | \, \, 1\leq i \leq j\leq n \Big\} \subset \mathbb{R}^{n}.
\]
As a representation of $W = S_{n+1}$, $P^{\mathrm{ab}} \cong  \mathbb{Z}^{n+1\choose2}$ is the permutation representation associated to the action of $S_{n+1}$ on the set of unordered pairs of distinct elements (or $2$-element subsets) of $\{ 1 , 2, \ldots , n+1 \}$: the generator of $P^{\mathrm{ab}}$ corresponding to the positive root $e_{i} + e_{i+1}+\cdots +e_{j-1} +e_j$ is identified with the unordered pair $\{ i , j +1 \} $. (The $S_{n+1}$--representation $\mathbb{Z}\Phi$ has a similar description involving ordered pairs). Working over $\mathbb{Q}$, the $S_{n+1}$--representation $P^{\mathrm{ab}}\otimes\mathbb{Q}$ splits as follows as a sum of irreducible representations (\cite[Example 7.18.9]{stanley}): denoting a representation of $S_{n+1}$ by its corresponding Young diagram (so that e.g. $(n+1)$ denotes the trivial representation, $(n,1)$ denotes the `standard' representation $\{ x \in \mathbb{Q}^{n+1} \, |\, \sum_{i=1}^{n+1} x_i = 0\}$) we have:
\begin{align*}
P^{\mathrm{ab}}\otimes\mathbb{Q} \cong (n+1) \oplus (n,1) \oplus (n-1,2) \quad \text{   if } n \geq 3 \, , \\
P^{\mathrm{ab}}\otimes\mathbb{Q} \cong (3) \oplus (2,1) \quad \text{   if } n = 2 \quad , \quad P^{\mathrm{ab}}\otimes \mathbb{Q}\cong (2) \quad \text{    if  } n =1 .
\end{align*}

\subsubsection{Mapping out of $P^{\mathrm{ab}}$}

We conclude this subsection by discussing the following result, which gives a `universal property' of the $W$--representation $P^{\mathrm{ab}}$:

\begin{lemma}\label{lemma:mapfromPab}
Let $\Gamma$ be an $ADE$ Dynkin digram (with vertex set denoted $I$). Let $A$ be an abelian group on which the Weyl group $W = W(\Gamma )$ acts by automorphisms. Let $a_i \in A$ be elements indexed by $i \in I$, such that 
\[
s_i \cdot a_j =  \begin{cases} a_j & \text{ if there is no edge connecting } i \text{ and } j\\
s_{j}\cdot a_i &  \text{ if there is an edge connecting } i \text{ and } j 
\end{cases}.
\]
Then there exists a unique $W$--equivariant homomorphism $P^{\mathrm{ab}} \to A$ sending $[ s_{i}^2 ] \mapsto a_i$.
\end{lemma}

\begin{remark}\label{remark:mapfromPab}
There is an analogous statement---with very similar proof---that holds instead for the representations $\mathbb{Z}\Phi$ and $\overline{P}^{\mathrm{ab}}$: 
given elements $a_{i}\in A $ for each $i \in I$ such that
\begin{align*}
& s_i \cdot a_j = s_j \cdot a_i \,\, \,  \,\,\, , \,\,\,  i , j \in I  \\
\text{resp. }\,\, &   s_i \cdot a_j =  \begin{cases} -a_j & \text{ if } i =j\\
a_j & \text{ if } i \neq j \text{ and there is no edge connecting } i \text{ and } j\\
s_{j}\cdot a_i &  \text{ if there is an edge connecting } i \text{ and } j 
\end{cases} \, \, ,  
\end{align*}
there exists a unique $W$--equivariant homomorphism $\mathbb{Z}\Phi \to A$, resp. $\overline{P}^\mathrm{ab} \to A$, sending $T_{e_i} \mapsto a_i$, resp. $\overline{t}_{e_i} := T_{e_i} - T_{e_i} \mapsto a_i$.
\end{remark}

\begin{proof}
The uniqueness assertion is proved as in Proposition \ref{proposition:characterisation}. To establish existence we use induction on the height of positive roots. The \textit{height} of a positive root $\alpha \in \Phi_+$ is defined as the positive integer $\sum_{i\in I} k_i $ where $\alpha = \sum_{i \in I} k_i e_i$ is the unique expression of $\alpha$ in terms of the simple roots, where $k_i \in \mathbb{Z}_{\geq 0}$ by Lemma \ref{lemma:rootADE}(2). We use Proposition \ref{proposition:characterisation} to identify $P^{\mathrm{ab}} = \bigoplus_{\alpha \in \Phi_+} \mathbb{Z}$. For a given integer $h \geq 1$, let $P_{h}^{\mathrm{ab}}$ be the subgroup of $P^{\mathrm{ab}}$ generated by the elements $t_{\alpha}$ where $\alpha \in \Phi_+$ has height $\leq h$.

We will establish the following assertion by induction on $h \geq 1$: there exists a homomorphism $f_h : P_{h}^{\mathrm{ab}} \to A$ such that
\begin{enumerate}
\item $f_h (t_{e_i} ) = a_i$
\item if $\alpha $ is a positive root of height $\leq h$ and $s_i \cdot \alpha$ is a positive root of height $\leq h$ then 
\[
    f_h ( s_i \cdot t_\alpha ) = s_i \cdot f_h ( t_\alpha ).
\]
\item if $\beta, \gamma$ are positive roots of height $\leq h$ such that $\beta +e_j$ and $\gamma + e_k$ are positive roots and $\beta + e_j = \gamma + e_k$, then
\[
s_j \cdot f_h (t_\beta ) = s_k \cdot f_h (t_\gamma ).
\]
\end{enumerate}
(Because there is an upper bound on the height of positive roots, the existence of a $W$--equivariant homomorphism $f: P^{\mathrm{ab}}\to A$ with the required properties will follow from the above assertion with $h \gg 0$.)

The base case $h = 1$ is obtained by setting $f_1 (t_{e_j} ) = a_j$ (the positive roots of height one are the simple roots $e_j$, $j \in I$), and hence $f_1$ satisfies (1). It satisfies (2) because if $s_i \cdot e_j$ is a positive root of height one then there is no edge connecting $i$ and $j$; thus $f_1 (s_i \cdot t_{e_j } ) = f_1 ( t_{e_j} ) = a_j  = s_i \cdot a_j = s_i \cdot f_1 ( t_{e_j} )$. It also satisfies (3), because $e_i + e_j$ is a positive root precisely when there is an edge connecting $i$ and $j$, hence $s_j \cdot f_1 ( t_{e_i} ) = s_j \cdot a_i = s_i \cdot a_j = s_i \cdot f_1 (t_{e_j} )$.

Suppose that $f_h$ has been constructed for a given $h \geq 1$. Then we define $f_{h+1}$ as follows. For a positive root $\alpha$ of height $\leq h$ we define $f_{h+1} (t_\alpha ) = f_h (t_\alpha )$. Suppose that $\alpha \in \Phi_+$ is a positive root of height $h+1$. Then $\alpha = \beta + e_j$ for a positive root $\beta \in \Phi_+$ and a simple root $e_j$ by Lemma \ref{lemma:positive+simple} (hence $\beta$ has height $h$). We define
\[
f_{h+1}(t_\alpha ) = s_j \cdot f_h (t_\beta ).
\]
We must verify that $f_{h+1} (t_\alpha )$ is indeed well-defined. Write $\alpha = \beta + e_j = \gamma + e_k$ for another positive root $\gamma$ and simple root $e_k$. By Lemma \ref{lemma:positive+simple} then $s_j \cdot \beta = \alpha = s_k \cdot \gamma $. Then $s_j \cdot f_h ( t_\beta ) = s_k \cdot f_h (t_\gamma )$ holds by property (3) of $f_h$; thus $f_{h+1}$ is well-defined. Clearly, (1) holds for $f_{h+1}$ because it holds for $f_h$. 

We now verify property (2) for $f_{h+1}$. Let $\alpha$ be a positive root of height $\leq h+1$ such that $s_i \cdot \alpha$ is also a positive root of height $\leq h+1$. We write $\alpha = \beta + e_j$ for a positive root $\beta$ and simple root $e_j$, which by Lemma \ref{lemma:positive+simple} satisfy $s_j \cdot \alpha = \beta$ (or equivalently, $\beta \cdot e_j = 1$). We want to prove $f_{h+1}(s_i \cdot t_\alpha ) = s_i \cdot f_{h+1}( t_\alpha )$. We have the following cases:

\textbf{Case 1.} Both $\alpha$ and $s_i \cdot \alpha$ have height $\leq h$. Then the required identity follows by property (2) of $f_h$. \\

\textbf{Case 2.} $\alpha$ has height $\leq h$ and $s_i \cdot \alpha$ has height $h+1$. By \cite[\S 9.4, Table 1]{humphreys-lie}, since $\Phi$ is an $ADE$ root system it follows that $\alpha$ has height $h$ and $\alpha \cdot e_i = 1$. Hence
\[
f_{h+1}(s_i \cdot t_\alpha ) = f_{h+1}(T_{\alpha + e_i }) = f_{h} (t_\alpha )
\]
by definition of $f_{h+1}$.\\

\textbf{Case 3.} $\alpha$ and $s_i \cdot \alpha$ both have height $h+1$. Then $\alpha \cdot e_i = 0$. We have the following sub-cases:\\

\textbf{Sub-case 3.1.} $i =j$. Then $f_{h+1} ( s_i \cdot t_\alpha ) = f_{h+1} (t_\beta ) = s_i \cdot s_i \cdot f_{h}(t_\beta ) = s_i \cdot f_{h+1}(t_\alpha )$.\\

\textbf{Sub-case 3.2.} $i \neq j$ and there is no edge from $i$ to $j$. Then $\beta \cdot e_i = 0$. Hence 
\begin{align*}
f_{h+1}(s_i \cdot t_\alpha  ) & = f_{h+1} (t_\alpha ) = s_j \cdot f_h (t_\beta ) \\
& = s_j \cdot f_h (s_i \cdot t_\beta ) = s_j  s_i \cdot f_h ( t_\beta ) \\
&  = s_i  s_j \cdot f_h (t_\beta ) = s_i \cdot f_{h+1}(t_\alpha )
\end{align*}
using (3) for $f_h$.\\

\textbf{Sub-case 3.3.} $i \neq j$ and there is an edge from $i$ to $j$. Since $s_i \cdot \alpha = \alpha$ and $e_i \cdot e_j =1$, then $\beta \cdot e_i = -1$. Thus $\beta = \gamma + e_i$ for a positive root $\gamma$; namely, $\gamma = s_i \cdot \beta$, which is a positive root because $\beta\neq e_i$. Thus, using (2) for $f_h$ we obtain
\begin{align*}
f_{h+1}(s_i \cdot t_\alpha ) & = f_{h+1}(t_\alpha ) = s_j \cdot f_h (t_\beta )  = s_j  s_i \cdot f_h (t_\gamma ) 
\end{align*}
and
\begin{align*}
s_i \cdot f_{h+1}(t_\alpha ) & = s_i s_j f_h (t_\beta ) = s_i s_j s_i \cdot f_{h} (t_\gamma )\\
& = s_j s_i s_j \cdot f_h (t_\gamma ) = s_j s_i \cdot f_h (s_j \cdot t_\gamma ).
\end{align*}
We claim that $s_j \cdot \gamma = \gamma$. From this we have $s_j \cdot t_\gamma = t_\gamma$ and hence the two calculations above give the required result. To verify the claim observe that
\begin{align*}
\gamma + e_i & = \beta  = s_j \cdot \alpha\\
& = s_j \cdot ( \gamma + e_i + e_j )  = s_j \cdot \gamma + (e_i + e_j ) - e_j\\
& = s_j \cdot \gamma + e_i
\end{align*}
and hence $s_j \cdot \gamma = \gamma$.\\

Finally, we verify property (3) for $f_{h+1}$. Let $\beta, \gamma$ be positive roots of height $\leq h+1$ such that $\beta +e_j$ and $\gamma + e_k$ are positive roots and $\beta + e_j = \gamma + e_k$. We need to prove $s_j \cdot f_{h+1} (t_\beta ) = s_k \cdot f_{h+1} (t_\gamma )$. By the corresponding property for $f_h$, it remains only to consider the case when both $\beta$ and $\gamma$ have height $h+1$. By Lemma \ref{lemma:positive+simple}, we can write $\beta = \beta_0 + e_l$ and $\gamma = \gamma_0 + e_m$ for positive roots $\beta_0$ and $\gamma_0$ and simple roots $e_l$ and $e_m$. We need to prove that
\begin{align}
s_j s_l \cdot  f_h (t_{\beta_0 }) = s_k s_m \cdot f_h (t_{\gamma_0} ).\label{3ind}
\end{align}
We consider again three cases:\\

\textbf{Case 1.} $j = k$. Then $\beta = \gamma$ and by (3) for $f_h$ we obtain $s_l \cdot  f_h (t_{\beta_0 }) = s_m f_h (t_{\gamma_0} )$ and hence (\ref{3ind}) after applying $s_j = s_k$ on both sides. \\

\textbf{Case 2.} $j \neq k$ and there is no edge connecting $j$ and $k$. Then by Lemma \ref{lemma:positive+simple} we have
\[
\beta = s_j \cdot (\beta + e_j ) = s_j \cdot (\gamma + e_k ) = s_j \cdot \gamma + e_k .
\]
Since $\beta \cdot e_j =1$ and $e_k \cdot e_j = 0$ then $(s_j \cdot \gamma ) \cdot e_j  = 1$, and hence $\gamma \neq  e_j$. It follows by \cite[\S 10.2, Lemma B]{humphreys-lie} that $\gamma \cdot e_j$ is a positive root. Then by (3) applied to $f_h$ we obtain
\begin{align*}
s_l \cdot f_h (t_{\beta_0} ) & = s_k \cdot f_h ( t_{s_j \cdot \gamma } ) = s_k s_j  \cdot f_{h+1}(t_\gamma )\\
& = s_j  s_k \cdot f_{h}(t_\gamma ) = s_j s_k s_m \cdot f_{h} (t_{\gamma_0}) 
\end{align*}
from which (\ref{3ind}) follows by applying $s_j$ to both sides.\\

\textbf{Case 3.} $ j \neq k$ and there is an edge connecting $j$ and $k$. Then
\[
-1 = (\beta + e_j ) \cdot e_j = (\gamma + e_k ) \cdot e_j = \gamma \cdot e_j +1
\]
and hence $\gamma \cdot e_j = -2$. By \cite[\S 9.4, Table 1]{humphreys-lie} it follows that $\gamma = e_j$. But $\gamma$ has height $h+1 \geq 2$, which gives a contradiction. (Hence Case 3 doesn't appear.) \end{proof}

\subsection{$ADE$ configurations of Lagrangian spheres and Weyl group representations}\label{subsection:ADEconfs}

We now describe various Braid group representations in symplectic mapping class groups, together with associated Weyl group representation structures, associated to $ADE$ configurations of Lagrangian spheres. 
In what follows, let $(X, \omega )$ be a compact symplectic $4$-manifold; and let $\Gamma$ be an $ADE$ Dynkin diagram, with vertex set denoted $I$.
\subsubsection{Braid group representations associated to $ADE$ configurations of Lagrangian spheres}

\begin{definition}For an $ADE$ Dynkin diagram $\Gamma$, by a \textit{$\Gamma$--configuration of Lagrangian spheres} in $(X, \omega )$ we shall mean a collection $L_i \quad  ( i \in I )$ of embedded Lagrangian spheres in $(X, \omega )$ such that for $i \neq j$ the Lagrangian spheres $L_i $ and $ L_j$ have non-empty intersection if and only if there is an edge connecting $i$ and $j$, and in this case $L_i$ meets $L_j$ only at a single point and transversely. A \textit{consistent orientation} of a $\Gamma$--configuration of Lagrangian spheres shall mean a choice of orientation of every Lagrangian $L_i$, such that whenever there is an edge connecting $i$ and $j$ the intersection number of the corresponding Lagrangians is $L_i \cdot L_j = +1$. (Since $\Gamma$ is a tree, there exist exactly two consists orientations.)
\\
\end{definition}

When two Lagrangian spheres $L, L^\prime$ in $(X, \omega )$ are \textit{disjoint} then obviously their Dehn--Seidel twists give commuting elements of $\pi_0 \mathrm{Symp}(X, \omega )$: 
\begin{align}
\tau_L \tau_{L^\prime} = \tau_{L^\prime}\tau_L.\label{rel1}
\end{align}
When instead they \textit{intersect but only at a single point and transversely}, then by \cite[Proposition A.1]{seidelknotted} their Dehn--Seidel twists satisfy the braid relation in $\pi_0 \mathrm{Symp}(X, \omega )$:
\begin{align}
\tau_{L}\tau_{L^\prime}\tau_L = \tau_{L^\prime }\tau_L \tau_{L^\prime } . \label{rel2}
\end{align}

It follows from the presentation (\ref{braidgroup}) and (\ref{rel1}-\ref{rel2}) that a $\Gamma$--configuration of Lagrangian spheres $\{ L_i \, | \, i \in I \}$ in $(X, \omega )$ induces a group homomorphism from the generalised Braid group of type $\Gamma$ into the symplectic mapping class group:
\begin{align}
\rho : B \to \pi_0 \mathrm{Symp}(X, \omega ) \quad , \quad s_i \mapsto \tau_{L_i}.\label{rho}
\end{align}
Since the squared Dehn--Seidel twists $\tau_{L_i}^2$ are trivial in $\pi_0 \mathrm{Diff}(X)$, and the generalised pure braid group $P$ of type $\Gamma$ is the subgroup of $B$ normally generated by $s_{i}^2$, it follows that the restriction of (\ref{rho}) to the subgroup $P \leq B$ has image contained in the smoothly trivial symplectic mapping class group:
\begin{align}
\rho_0 := \rho|_P : P \to \pi_0 \mathrm{Symp}_0 (X, \omega ). \label{rho_0}
\end{align}

\subsubsection{Weyl group actions on abelianisations}

\begin{proposition}\label{proposition:weylaction}
Let $\{ L_i \, | \, i \in I\}$ be an $ADE$ configuration of Lagrangian spheres in $(X, \omega )$ of type $\Gamma$. Then:
\begin{enumerate}
\item The Weyl group $W = W(\Gamma )$ acts by automorphisms on $(\pi_0 \mathrm{Symp}_0 (X, \omega ))^{\mathrm{ab}}$ as follows: for $i \in I$ and $[f] \in (\pi_0 \mathrm{Symp}_0 (X, \omega ))^{\mathrm{ab}}$
\[
s_i \cdot [f] := [\tau_{L_i}f \tau_{L_i}^{-1}].
\]

\item The Weyl group $W = W(\Gamma )$ acts by automorphisms on $(\pi_1 \mathcal{S}_{0} (X, \omega ))^{\mathrm{ab}}$ as follows: for $i \in I$ and a homotopy class of loop $[\omega_t] \in (\pi_1 \mathcal{S}_{0} (X, \omega ))^{\mathrm{ab}}$
\[
s_i \cdot [\omega_t] := [(\tau_{L_i}^{-1})^\ast \omega_t ].
\]
(Equivalently: $s_i \cdot [\omega_t] = [\tau_{L_i}^\ast \omega_t ]$, since $\tau_{L_i}^2$ is trivial in $\pi_0 \mathrm{Diff}(X)$.)
\item The map $(\pi_1 \mathcal{S}_{0} (X, \omega ))^{\mathrm{ab}} \to (\pi_0 \mathrm{Symp}_0 (X, \omega ))^{\mathrm{ab}}$ induced by the connecting homomorphism $\delta$ of the fibration (\ref{fibration}) is $W$--equivariant for the $W$--actions defined by (1-2).
\end{enumerate}
\end{proposition}

\begin{proof} (1) The group $\pi_0 \mathrm{Symp}(X, \omega )$ acts on itself (by group automorphisms) via left-conjugation: $\phi \cdot f := \phi f \phi^{-1} $. This action preserves the normal subgroup given by the smoothly trivial symplectic mapping class group $\pi_0 \mathrm{Symp}_0 (X, \omega ) \leq \pi_0 \mathrm{Symp}(X, \omega )$. Using the homomorphism $\rho : B \to \pi_0 \mathrm{Symp}(X, \omega )$ from (\ref{rho}) we have an induced action of $B$ on $\pi_0 \mathrm{Symp}_0 (X, \omega )$ given by $s_i \cdot f := \tau_{L_i} f \tau_{L_i}^{-1}$. The subgroup $P \leq B$ acts trivially on the abelianisation $(\pi_0 \mathrm{Symp}_0 (X, \omega ) )^{\mathrm{ab}}$ because $\tau_{L_i}^2$ lies in $\pi_0 \mathrm{Symp}_0 (X, \omega )$ and $P\leq B$ is the subgroup normally generated by the $s_{i}^2$ for $i \in I$; hence by the short exact sequence (\ref{SES1}) the $P$--action on $(\pi_0 \mathrm{Symp}_0 (X, \omega ) )^{\mathrm{ab}}$ descends to a $W$--action. 

(2) The symplectic mapping class group $\pi_0 \mathrm{Symp}(X, \omega)$ acts on $\pi_1 \mathcal{S}_{0}(X, \omega )$ by $\phi \cdot [\omega_t ] = [(\phi^{-1})^\ast \omega_t ]$. Via the homomorphism $\rho : B \to \pi_0 \mathrm{Symp}(X, \omega )$ this induces an action of $B$ on $\pi_1 \mathcal{S}_{0} (X, \omega )$ by $s_i \cdot [\omega_t ] = [(\tau_{L_i}^{-1})^\ast \omega_t ]$. Again, the induced action of $P \leq B$ on the abelianisation $(\pi_0 \mathcal{S}_{0} (X, \omega ) )^{\mathrm{ab}}$ is trivial because $\tau_{L_i}^2$ is trivial in $\pi_0 \mathrm{Diff} (X )$; hence by the short exact sequence (\ref{SES1}) the $P$--action on $(\pi_0 \mathcal{S}_{0} (X, \omega ) )^{\mathrm{ab}}$ descends to a $W$--action. 

(3) The connecting map $\delta : \pi_1 \mathcal{S}_{0}(X, \omega ) \to \pi_0 \mathrm{Symp}_0 (X, \omega )$ can be described as follows. For a loop $[\omega_t ]\in \pi_1 \mathcal{S}_{0} (X, \omega )$ Moser's argument provides an isotopy $f_t$, $t \in [0,1]$, through diffeomorphisms with $f_0 = \mathrm{Id}$ and $\omega_t = f_{t}^\ast \omega$ for all $t \in [0,1]$. Thus $f_1$ is a smoothly trivial symplectomorphism, and by definition $\delta ( [\omega_t ] ) = f_1 $. Since
\begin{align*}
(\tau_{L_i}^{-1} )^\ast \omega_t =(\tau_{L_i}^{-1} )^\ast f_{t}^\ast \omega = (\tau_{L_i}^{-1} )^\ast f_{t}^\ast \tau_{L_i}^\ast \omega =( \tau_{L_i} f_t \tau_{L_i}^{-1} )^\ast \omega 
\end{align*}
it follows that $\delta ( s_i \cdot [\omega_t ] ) = s_i \cdot \delta ([\omega_t ] )$, as required.
\end{proof}

\subsubsection{Equivariance properties}

\begin{proposition}\label{proposition:equivariancerho0}
Let $\{ L_i \, | \, i \in I\}$ be an $ADE$ configuration of Lagrangian spheres in $(X, \omega )$ of type $\Gamma$. The homomorphism $\rho_0 : P^{\mathrm{ab}} \to ( \pi_0 \mathrm{Symp}_0 (X, \omega ))^{\mathrm{ab}} $ induced by (\ref{rho_0}) is $W$--equivariant for the Weyl group representations defined in (\ref{actionPab}) and Proposition \ref{proposition:weylaction}(1).
\end{proposition}
\begin{proof}Evidently, the homomorphism $\rho_0 : P \to \pi_0 \mathrm{Symp}_0 (X, \omega )$ is $B$--equivariant for the left-conjugation $B$--action on $P$, and the left-conjugation $B$--action on $\pi_0 \mathrm{Symp}_0 (X, \omega )$ defined via $\rho : B \to \pi_0 \mathrm{Symp}(X, \omega )$ (see the proof of Proposition \ref{proposition:weylaction}). From this the assertion follows immediately.
\end{proof}

Below, for an oriented Lagrangian sphere $L \subset (X, \omega )$ we consider the canonical lifts $\mathcal{O}_{\pm L} \in \pi_1 \mathcal{S}_{0}(X, \omega )$ of the squared Dehn--Seidel twist associated to $L$ and its orientation-reversal $-L$.

\begin{lemma}\label{lemma:relationsO}
Let $\{ L_i \, | \, i \in I\}$ be an $ADE$ configuration of Lagrangian spheres in $(X, \omega )$ of type $\Gamma$, equipped with a consistent orientation. Then, with respect to the $W$--action defined in Proposition \ref{proposition:weylaction}(2), the elements $\mathcal{O}_{\pm L_i} \in (\pi_1 \mathcal{S}_{0}(X, \omega ))^{\mathrm{ab}}$ satisfy: for $i , j \in I$
\[
s_i \cdot  \mathcal{O}_{\pm L_j} = \begin{cases}
\mathcal{O}_{\mp L_j} & \text{ if } i = j\\
\mathcal{O}_{\pm L_j} & \text{ if } i \neq j \text{ and there is no edge connecting } i \text{ and } j \\
s_j \cdot \mathcal{O}_{\pm L_i} & \text{ if there is an edge connecting } i \text{ and } j 
\end{cases}.
\]
\end{lemma}
\begin{proof}
By Proposition \ref{lemma:naturality} we have the following identity in $\pi_1 \mathcal{S}_{0}(X, \omega )$:
\begin{align}
(\tau_{L_i}^{-1})^\ast \mathcal{O}_{\pm L_j} = \mathcal{O}_{\pm \tau_{L_i} (L_j )}. \label{formula:naturality} 
\end{align}

When $i = j$ then $\tau_{L_i}$ preserves $L_j = L_i$ but reverses its orientation; thus from (\ref{formula:naturality}) we have:
\begin{align}
(\tau_{L_i}^{-1})^\ast \mathcal{O}_{\pm L_i} = \mathcal{O}_{\mp L_i } .\label{O1}
\end{align}

When $i \neq j$ and there is no edge connecting $i$ and $j$, then $\tau_{L_i} (L_j ) = L_j$ and thus from (\ref{formula:naturality}) we have:
\begin{align}
(\tau_{L_i}^{-1})^\ast \mathcal{O}_{\pm L_j} = \mathcal{O}_{\mp L_j }. \label{O2}
\end{align}

When there is an edge connecting $i$ and $j$ then $L_i$ and $L_j$ intersect at a single point only, both transversely and positively. By \cite[Proposition A.2 and Proposition A.3]{seidelknotted} the spheres $\tau_{L_i} (L_j )$ and $\tau_{L_j}^{-1} (L_i )$ are oriented Lagrangian isotopic (they are both oriented Lagrangian isotopic to the sphere $L_i \# L_j$ obtained by \textit{Lagrangian surgery}), and hence
\[
\mathcal{O}_{\pm \tau_{L_i} (L_j )} = \mathcal{O}_{\pm \tau_{L_j}^{-1} (L_i )}.
\]
Hence from (\ref{formula:naturality}) and the similar identity $\tau_{L_j}^\ast \mathcal{O}_{\pm L_i }=  \mathcal{O}_{\pm \tau_{L_j}^{-1} (L_i )}$ we obtain:
\begin{align}
(\tau_{L_i}^{-1})^\ast \mathcal{O}_{\pm L_j} = \tau_{L_j}^\ast \mathcal{O}_{\pm L_i}. \label{O3}
\end{align}

Passing to $(\pi_1 \mathcal{S}_{0} (X, \omega ) )^{\mathrm{ab}}$, the identities (\ref{O1}), (\ref{O2}) and (\ref{O3}) say: $s_i \cdot \mathcal{O}_{L_i} = \mathcal{O}_{-L_i}$, $s_i \cdot \mathcal{O}_{L_j} = \mathcal{O}_{L_j}$ and $s_i \cdot \mathcal{O}_{L_j} = s_{j}^{-1} \cdot \mathcal{O}_{L_i}$, respectively. Since $s_{j}^2 = 1 \in W$, then the last identity reads $s_i \cdot \mathcal{O}_{L_j } = s_j \cdot \mathcal{O}_{L_i }$. This completes the proof.
\end{proof}

Using this, we obtain an analogue of the homomorphism $P^\mathrm{ab}\to (\pi_0 \mathrm{Symp}_0 (X, \omega ))^\mathrm{ab}$ now at the level of $\pi_1 \mathcal{S}_0 (X, \omega )$ :
\begin{proposition}\label{proposition:mapZPhi}
Let $\{ L_i \, | \, i \in I\}$ be a consistently oriented $ADE$ configuration of Lagrangian spheres in $(X, \omega )$ of type $\Gamma$. Then, with respect to the $W = W(\Gamma)$--actions defined in Proposition \ref{proposition:weylaction}, there exists a unique $W$--equivariant homomorphism $\rho_0 : \mathbb{Z}\Phi \to (\pi_1 \mathcal{S}_0 (X, \omega ))^{\mathrm{ab}}$ which sends $T_{e_i}\mapsto \mathcal{O}_{L_i}$.
\end{proposition}
\begin{proof}
The existence and uniqueness of such map follows from the analogue of Lemma \ref{lemma:mapfromPab} for $\mathbb{Z}\Phi$ (see Remark \ref{remark:mapfromPab}) combined with the relations coming from Lemma \ref{lemma:relationsO}.
\end{proof}

\subsection{$ADE$ configurations and Kronheimer's invariant}

Let $(X, \omega )$ be a closed symplectic $4$-manifold. Let $\Gamma$ be an $ADE$ Dynkin diagram with vertex set denoted $I$, and let $\{ L_i \, | \, i \in I\}$ be an $ADE$ configuration of Lagrangian spheres of type $\Gamma$ equipped with a consistent orientation.

Let $[L_i ] \in H^2 (X, \mathbb{Z} ) $ be the fundamental class of $L_i$. 
Consider the $\mathbb{R}$-vector subspace of $H^2 ( X, \mathbb{R} )$ generated by classes $\mathrm{PD}([L_i ])$ over all $i \in I$, equipped with the quadratic form inherited from the intersection product. This is isometrically identified with $V_\mathbb{R}$ (cf. \S \ref{subsection:rootsystems}) by identifying the classes $\mathrm{PD}([L_i ])$ with the simple roots $e_i$ in the $ADE$ root system $\Phi \subset V_\mathbb{R}$. We thus have an embedding of the $ADE$ root system $ \Phi $ in the cohomology $H^{2} (X, \mathbb{Z} )$. 

\begin{definition}
Let $\{ L_i \, | \, i \in I\}$ be a consistently oriented $\Gamma$--configuration of Lagrangian spheres in a closed symplectic $4$-manifold $(X, \omega )$. Let 
\begin{align}
\widetilde{q} : ( \pi_1 \mathcal{S}_0 (X, \omega ) )^{\mathrm{ab}} \to \bigoplus_{\alpha \in \Phi_+ } \mathbb{Z}\quad , \quad \widetilde{q} = \bigoplus_{\alpha \in \Phi_+} q_\alpha \label{qtildeADE}
\end{align}
be the homomorphism defined as the sum of the homomorphisms $q_\alpha$ (cf. Definition \ref{definition:qtilde}) over all positive roots $\alpha \in \Phi_+ \subset H^2 (X, \mathbb{Z})$.
By Proposition \ref{proposition:q}, if $(X, \omega )$ is a symplectic Calabi--Yau surface with $b^{+} (X) \neq 1$ then the homomorphism $q$ descends to a (unique) homomorphism
\begin{align}
q : ( \pi_0 \mathrm{Symp}_0 (X, \omega ) )^{\mathrm{ab}} \to \bigoplus_{\alpha \in \Phi_+ } \mathbb{Z}. \label{qADE} 
\end{align}
\end{definition}

\begin{proposition}\label{proposition:equivarianceq}
Let $\{ L_i \, | \, i \in I\}$ be a consistently oriented $ADE$ configuration of Lagrangian spheres of type $\Gamma$ in a closed symplectic $4$-manifold $(X, \omega )$. Let $W = W(\Gamma)$ be the Weyl group of type $\Gamma$. Then:
\begin{enumerate}
\item The homomorphism $\widetilde{q}$ from (\ref{qtildeADE}) is $W$--equivariant, with respect to the $W$--representation on $( \pi_1 \mathcal{S}_0 (X, \omega ) )^{\mathrm{ab}}$ defined by Proposition \ref{proposition:weylaction}(2) and the $W$--representation on $\bigoplus_{\alpha \in \Phi_+} \mathbb{Z}$ given by (\ref{actionpositive}).
\item Suppose $(X, \omega )$ is a symplectic Calabi--Yau surface with $b^{+}(X) > 1$. Then the homomorphism $q$ from (\ref{qADE}) is $W$--equivariant, with respect to the $W$--representation on $( \pi_0 \mathrm{Symp}_0 (X, \omega ) )^{\mathrm{ab}}$ defined by Proposition \ref{proposition:weylaction}(1) and the $W$--representation on $\bigoplus_{\alpha \in \Phi_+} \mathbb{Z}$ given by (\ref{actionpositive}).
\end{enumerate}
\end{proposition}
\begin{proof}
It suffices to prove (1), because the map $(\pi_1 \mathcal{S}_{0} (X, \omega ))^{\mathrm{ab}} \to (\pi_0 \mathrm{Symp}_0 (X, \omega ))^{\mathrm{ab}}$ induced by the connecting homomorphism $\delta$ of the fibration (\ref{fibration}) is $W$--equivariant (by Proposition \ref{proposition:weylaction}(3)). The statement that $\widetilde{q}$ is $W$--equivariant amounts to saying: for any loop $\omega_t$ in $\mathcal{S}_0 (X, \omega )$, $i \in I$ and $\alpha \in \Phi_+$ there is the identity
\begin{align}
Q_{\alpha} ((\tau_{L_i}^{-1})^\ast \omega_t )  + Q_{-\alpha} ((\tau_{L_i}^{-1})^\ast \omega_t )
= \begin{cases} Q_{s_i \cdot \alpha}  (\omega_t ) + Q_{ - s_i \cdot \alpha}  (\omega_t ) & \text{ if } \alpha \neq e_i\\
Q_{e_i} ( \omega_t ) +Q_{-e_i} ( \omega_t )& \text{ if } \alpha = e_i .
\end{cases} \label{eq:equivarianceq}
\end{align}

Using the naturality of Kronheimer's invariant (Corollary \ref{corollary:naturalityQ})
applied to $f = \tau_{L_i}$ and $e = \alpha \in \Phi_+$, and the Picard--Lefschetz formula $\tau_{L_i}^\ast \alpha = s_i \cdot \alpha$, we have:
\[
Q_\alpha ((\tau_{L_i}^{-1})^\ast \omega_t ) = Q_{s_i \cdot \alpha} ( \omega_t ).
\]
(Note that when $\alpha = \pm e_i$ this says $Q_{\pm e_i} ((\tau_{L_i}^{-1})^\ast \omega_t ) = Q_{\mp e_i} ( \omega_t )$). From this, (\ref{eq:equivarianceq}) follows.
\end{proof}

\subsection{Equivariant splittings}

In this subsection we prove Theorems \ref{theorem:ADEK3}-\ref{theorem:ADE}. More generally, we shall establish a fundamental connection between the short exact sequence of $W$--representations (\ref{SESreps}) and the following short exact sequence derived from the fibration (\ref{fibration}):\\
\begin{align}
0 \to \pi_1 \mathrm{Symp}_0 (X,\omega ) \backslash \pi_1 \mathrm{Diff}_0 (X,\omega )\to \pi_1 \mathcal{S}_0 (X,\omega ) \to \pi_0 \mathrm{Symp}_0 (X,\omega )\to 0 .\label{SESdiffsymp}
\end{align}
$ $

Our main result is:

\begin{theorem}\label{theorem:splitting1}
Let $(X, \omega )$ be a symplectic Calabi--Yau surface with $b^{+}(X) \neq 1$. Let $\{ L_i \, | \, i \in I \}$ be a consistently oriented $ADE$ configuration of Lagrangian spheres of type $\Gamma$. Fix a homological orientation on $X$ such that $Q_{\mathrm{PD}([L_i ])} (\mathcal{O}_{L_i } ) = 1$ for some $i \in I$ (cf. Proposition \ref{proposition:calculation}), and define the homomorphism $q$ from (\ref{qADE}) using this homological orientation. Then the composition
\[
P^{\mathrm{ab}} \xrightarrow{\rho_0}  (\pi_0 \mathrm{Symp}_0 (X, \omega ))^{\mathrm{ab}} \xrightarrow{q} \bigoplus_{\alpha \in \Phi_+} \mathbb{Z} \overset{\text{Prop. } \ref{proposition:characterisation}}{\cong}  P^{\mathrm{ab}}
\]
is the identity map. Hence $q$ is a $W$--equivariant splitting of $\rho_0$ (cf. Proposition \ref{proposition:equivarianceq}).
\end{theorem}

\begin{lemma}\label{lemma:splitting}
Let $\Phi$ be the $ADE$ root system of type $\Gamma$. Let $f : \bigoplus_{\alpha \in \Phi_+} \mathbb{Z} \to \bigoplus_{\alpha \in \Phi_+} \mathbb{Z}$ be a homomorphism $\mathbb{Z}$--modules. Write $f_{\beta} (t_\alpha ) \in \mathbb{Z}$ for the coefficient of $t_\beta$ in $f(t_\alpha )$. Suppose $f$ has the following properties:
\begin{enumerate}
\item $f$ is equivariant with respect to the action of the Weyl group $W = W(\Gamma )$ on $\bigoplus_{\alpha \in \Phi_+} \mathbb{Z}$ defined in (\ref{actionpositive})
\item For every simple root $\alpha = e_i$ we have $f_{\alpha} ( t_\alpha ) \in \{ + 1 , -1 \}$
\item For every pair of distinct simple roots $\alpha , \beta$ we have $f_\beta (t_\alpha ) = 0$.
\end{enumerate}
Then $f = \mathrm{Id}$ or $- \mathrm{Id}$. 
\end{lemma}

\begin{proof}

If $i , j \in \Gamma$ are connected by an edge, then $s_j s_i \cdot e_j = s_j \cdot ( e_i + e_j ) = e_i$; thus by (1) $f_{e_i } (t_{e_i } ) = f_{e_j } (t_{e_j } )$. This, the fact that $\Gamma$ is a connected graph, and (2) imply that either $f_{\alpha} (t_\alpha ) =1$ for every simple root $\alpha  $ or else $-1$ for all simple roots. Hence, we may assume that $f_{\alpha} (t_\alpha ) = 1$ and then prove that $f = \mathrm{Id}$. 

We need to prove that for positive roots $\alpha , \beta \in \Phi_{+}$, we have 
\[
f_\beta ( t_\alpha ) = \begin{cases} 1 & \text{ if } \alpha = \beta \\ 0 & \text{ if } \alpha \neq \beta \end{cases}.
\]
To prove this it suffices to suppose that $\beta$ is a simple root, because of (1) and the fact that for any positive root $\beta$ there is $w\in W$ such that $w \cdot \alpha$ is a simple root \cite[\S 10.3, Theorem(c)]{humphreys-lie}. 
By (3), it only remains to show that if $\beta$ is a simple root and $\alpha$ is a positive root which is not simple then $f_\beta (t_\alpha ) = 0$. 

Because $\alpha$ is non-simple then by Lemma \ref{lemma:positive+simple} we may write $\alpha = \alpha_1 + \ldots + \alpha_n$ where each of $\alpha_k$ is a simple root, and $s_{\alpha_k} \cdot (\alpha_1 + \ldots + \alpha_k ) = \alpha_1 + \ldots + \alpha_{k-1}$ for all $k > 1$. Recall that $n$ is called the height of the positive root $\alpha$. We prove $f_\beta (t_\alpha ) = 0$ for all simple roots $\beta$ and all non-simple positive roots $\alpha$, by induction on the height $n$. (The base case $n = 1$ is proved already).

Let $i$ denote the maximum value in $  \{ 1 , \ldots , n\}$ such that $\beta = \alpha_i$ (with $i = 0$ if $\beta \neq \alpha_k$ for all $k$). Note that $\beta \cdot \alpha $ must be either $\pm 1$ or $0$, and for each $k$ we have $\beta \cdot \alpha_k $ equals $0$ or $1$ or $-2$. Thus, we distinguish the following cases:\\

\textbf{Case 1.} $i = 0$. Then there are two possibilities:\\

\textbf{Sub-case 1.1.} $\beta \cdot \alpha_k = 0$ for all $k \in \{ 1 , \ldots , n \}$. Thus, by (1)
\begin{align*}
f_\beta (t_\alpha ) = f_\beta ( t_{s_{\alpha_n} \cdots s_{\alpha_2} \cdot \alpha_1 } ) = f_{s_{\alpha_n} \cdots s_{\alpha_2}  \cdot \beta }  (t_{\alpha_1 } ) = f_{\beta} (T_{\alpha_1} )  = 0
\end{align*}
using $\beta \cdot \alpha_k = 0$ and $\beta \neq \alpha_1$.\\

\textbf{Sub-case 1.2.} There exists a unique $j \in \{ 1, \ldots , n \}$ such that $\beta \cdot \alpha_j = 1$ (and $\beta \cdot \alpha_k = 0$ for $k \neq j$). Then using (1) we have
\begin{align*}
f_\beta (t_\alpha ) & = f_\beta ( t_{s_{\alpha_n}\cdots s_{\alpha_{j+1}} \cdot (\alpha_1 + \cdots \alpha_j ) } ) = f_\beta ( t_{\alpha_1  + \cdots \alpha_j } )\\
& = f_{\beta + \alpha_j } (t_{\alpha_1 + \cdots + \alpha_{j-1}} ) = f_{\alpha_j} ( t_{\alpha_1 + \cdots \alpha_{j-1}} ) = 0
\end{align*}
where $=0$ follows by induction, since the height of $\alpha_1 +\cdots +  \alpha_{j-1}$ is $j-1 < n$.\\

\textbf{Case 2.} $i = 1$. Then since $\beta \cdot \alpha = -2 + \beta \cdot (\alpha_2 + \cdots \alpha_n )$ we deduce that $\beta \cdot (\alpha_2 + \cdots \alpha_n )$ must be either $1$, $2$ or $3$. Hence, there exists at least one $j \in \{ 1 , \cdots , n \}$ and at most three such that $\beta \cdot \alpha_j  = 1$. Let $j$ denote the maximum such. We have $j >1$. Then proceeding as before, we have
\begin{align*}
f_\beta (t_\alpha ) = f_{\beta + \alpha_j } (t_{\alpha_1 + \cdots + \alpha_{j-1}} ) = f_{\alpha_j} ( s_\beta \cdot t_{\alpha_1 + \cdots + \alpha_{j-1} } )  .
\end{align*}
If $j =2$ then $s_\beta \cdot t_{\alpha_1 + \cdots + \alpha_{j-1} } = t_{\alpha_1 }$ and since $\alpha_1$ is a simple root distinct from $\alpha_j$ then the above gives $f_\beta (t_\alpha ) = 0$. If $j >2$ then $s_\beta \cdot t_{\alpha_1 + \cdots + \alpha_{j-1} } = t_{s_\beta \cdot ( \alpha_1 + \cdots \alpha_{j-1})}$ and 
\[
s_\beta \cdot ( \alpha_1 + \cdots \alpha_{j-1}) = \alpha_1 + \cdots + \alpha_{j-1} + (\beta \cdot (\alpha_1 + \cdots + \alpha_{j-1})) \beta
\]
with $(\beta \cdot (\alpha_1 + \cdots + \alpha_{j-1})) \leq -2 + 2 = 0$. Hence $s_\beta \cdot ( \alpha_1 + \cdots \alpha_{j-1})$ has height $\leq j-1 < j \leq n$. Hence $f_{\alpha_j} ( s_\beta \cdot t_{\alpha_1 + \cdots + \alpha_{j-1} } ) = 0$ by induction, and thus $f_\beta (t_\alpha ) = 0$.\\

\textbf{Case 3.} $i >1$. Then $\beta \cdot (\alpha_1 + \cdots + \alpha_i ) = -1$ so we deduce that $ \beta\cdot (\alpha_{i+1} + \cdots + \alpha_n ) $ is either $0$ or $1$ or $2$.\\

\textbf{Sub-case 3.1.} $\beta \cdot \alpha_k = 0$ for all $k >i$. Then, proceeding as before 
\begin{align*}
f_\beta (t_\alpha ) = f_{\beta  } (t_{\alpha_1 + \cdots + \alpha_{i}} ) = f_{\alpha_i} ( t_{\alpha_1 + \cdots + \alpha_{i-1}} )
\end{align*}
which vanishes by induction because the height of $\alpha_1 + \cdots + \alpha_{i-1}$ is $< i \leq n$.\\

\textbf{Sub-case 3.2.} There exists at least one $j >i $, and at most two, such that $\beta \cdot \alpha_j = 1$. Let $j$ denote the maximum such. Then, as before
\begin{align*}
f_\beta ( t_\alpha ) = f_{\alpha_j}  ( s_\beta \cdot t_{\alpha_1 + \cdots + \alpha_{j-1} } ).
\end{align*}
Since $1<i<j$ then $2 \leq j-1$, so $\beta \neq \alpha_1 + \cdots + \alpha_{j-1}$ and hence $s_\beta \cdot t_{\alpha_1 + \cdots + \alpha_{j-1} } = t_{s_\beta \cdot (\alpha_1 + \cdots + \alpha_{j-1})} $. Then 
\begin{align*}
s_\beta \cdot (\alpha_1 + \cdots + \alpha_{j-1}) = (\alpha_1 + \cdots + \alpha_{j-1}) + ( \beta \cdot (\alpha_1 + \cdots + \alpha_{j-1} ))\beta 
\end{align*}
with $( \beta \cdot (\alpha_1 + \cdots + \alpha_{j-1} )) \leq -1 +1 = 0$. Hence $s_\beta \cdot (\alpha_1 + \cdots + \alpha_{j-1})$ has height $\leq j-1 < n$, so by induction $f_{\alpha_j}  (  t_{s_\beta \cdot ( \alpha_1 + \cdots + \alpha_{j-1} )} ) = 0$ and hence also $f_\beta ( t_\alpha ) = 0$.\\

The above comprise all possible cases; hence the proof is complete.
\end{proof}

\begin{proof}[Proof of Theorem \ref{theorem:splitting1}]
Let $f : \bigoplus_{\alpha \in \Phi_+} \mathbb{Z} \to \bigoplus_{\alpha \in \Phi_+} \mathbb{Z}$ be the homomorphism $f := q\circ \rho_0$ (where we identify $P^{\mathrm{ab}} \cong \bigoplus_{\alpha \in \Phi_+} \mathbb{Z}$ via Proposition \ref{proposition:characterisation}). Assumption (1) of Lemma \ref{lemma:splitting} holds for $f$ by Proposition \ref{proposition:characterisation}, Proposition \ref{proposition:equivariancerho0} and Proposition \ref{proposition:equivarianceq}. Assumption (2) holds by Proposition \ref{proposition:calculation}. Assumption (3) holds by Proposition \ref{proposition:A2}. Thus by Lemma \ref{lemma:splitting} we have $f = \pm \mathrm{Id}$, and because we have $Q_{[L_i]} (\mathcal{O}_{L_i}) =1$ for at least one $i$, then $f = \mathrm{Id}$.
\end{proof}

\begin{proof}[Proof of Theorems \ref{theorem:ADEK3}-\ref{theorem:ADE}]
Theorem \ref{theorem:ADEK3} follows immediately from Theorem \ref{theorem:splitting1}. For Theorem \ref{theorem:ADE}, 
we consider instead the homomorphism
\begin{align*}
Q : (\pi_1 \mathcal{S}_0 (X, \omega ) )^{\mathrm{ab}}\to \mathbb{Z}\Phi \quad , \quad Q  = \bigoplus_{\alpha \in \Phi} Q_{\alpha}
\end{align*}
which is also $W$--equivariant, by an argument similar to Proposition \ref{proposition:equivarianceq}. We can now show that $Q \circ \rho_0 = \mathrm{Id}$ by an argument almost identical to the proof of Theorem \ref{theorem:splitting1}, replacing Lemma \ref{lemma:splitting} with its obvious analogue involving instead the $W$--representation $\mathbb{Z}\Phi$.
\end{proof}



Similarly, given a consistently-oriented $ADE$ configuration of Lagrangian spheres in a closed symplectic $4$-manifold $(X, \omega )$, the abelian group $\pi_1 \mathrm{Diff}_0 (X )$ becomes a $W$--representation: $s_i \in W$ acts on $f$ by $f \mapsto \tau_{L_i}f\tau_{L_i}^{-1}$ (again, this defines $W$--action by the smooth triviality of $\tau_{L_i}^2$ and the Braid relations). By Lemma \ref{lemma:gammasymmetries} and the analogue of Lemma \ref{lemma:mapfromPab} for $\overline{P}^\mathrm{ab}$ there exists a unique $W$--equivariant homomorphism
\begin{align*}
\rho_0 : \overline{P}^{\mathrm{ab}} \to \pi_1 \mathrm{Diff}_0 (X) 
\end{align*}
uniquely characterised by $\overline{t}_{e_i}  \mapsto \gamma_{L_i}$, where $\gamma_{L_i}$ is the generalised Dehn twist defined in \S\ref{subsection:generaliseddehntwist}.


By Lemma \ref{lemma:FSW} we have that $Q$ precomposed with $p_\ast : \pi_1 \mathrm{Diff}_0 (X) \to \pi_1 \mathcal{S}_0 (X, \omega )$ (cf. (\ref{fibration})) coincides with 
\[
Q \circ p_\ast = FSW := \bigoplus_{\alpha \in \Phi} FSW_{\alpha}
\]
which furthermore vanishes on the image of $\pi_1 \mathrm{Symp}_0 (X, \omega ) \to \pi_1 \mathrm{Diff}_0 (X )$ by Lemma \ref{lemma:taubes}. By Lemma \ref{lemma:symplecticconjugation} when $(X, \omega )$ is a symplectic Calabi--Yau surface with $b^+ (X) >1$ the homomorphism $FSW$ takes values in the subrepresentation $\overline{P}^{\mathrm{ab}}\subset \mathbb{Z}\Phi$. 

The following result---whose proof follows by similar arguments as above---sums up the relationship between the short exact sequences (\ref{SESdiffsymp}) and (\ref{SESreps}):

\begin{theorem}\label{theorem:diagram}
Let $(X, \omega )$ be a symplectic Calabi--Yau surface with $b^+ (X) >1$ with a consistently oriented $ADE$ configuration of Lagrangian spheres. Then the above maps yield a commutative diagram 
\begin{equation*}
\begin{tikzcd}
 0 \arrow{r} & \overline{P}^{\mathrm{ab}} \arrow{r} & \mathbb{Z}\Phi \arrow{r} & P^{\mathrm{ab}} \arrow{r} & 0 \\
  & \pi_1 \mathrm{Symp}_0 (X, \omega ) \backslash \pi_1 \mathrm{Diff}_0 (X)  \arrow{u}{FSW} \arrow{r} & (\pi_1 \mathcal{S}_0 (X, \omega ))^{\mathrm{ab}} \arrow{u}{Q} \arrow{r} & (\pi_0 \mathrm{Symp}_0 (X, \omega ))^{\mathrm{ab}} \arrow{u}{q} \arrow{r} & 0 \\
  0 \arrow{r} & \overline{P}^{\mathrm{ab}} \arrow{u}{\rho_0} \arrow{r} & \mathbb{Z}\Phi \arrow{u}{\rho_0} \arrow{r} & P^{\mathrm{ab}} \arrow{u}{\rho_0} \arrow{r} & 0
\end{tikzcd}
\end{equation*}
where all maps are $W$--equivariant, rows are exact (except possibly the middle one on the left-most term), and compositions along each column equal the identity map.
\end{theorem}

There is a simple application of Theorem \ref{theorem:diagram} that is worth mentioning. Recall that the short exact sequence of $W$--representations (\ref{SESreps}) does not split by Lemma \ref{lemma:nonsplit}(1) (both on the left or the right, since we are working in an abelian category). 
Thus, the homomorphism $\pi_1 \mathrm{Symp}_0 (X, \omega ) \backslash \pi_1 \mathrm{Diff}_0 (X)  \to (\pi_1 \mathcal{S}_0 (X, \omega ))^{\mathrm{ab}}$ is not $W$--equivariantly left-split, and the surjective homomorphism $(\pi_1 \mathcal{S}_0 (X, \omega ))^{\mathrm{ab}} \to (\pi_0 \mathrm{Symp}_0 (X, \omega ))^{\mathrm{ab}}$ is not $W$--equivariantly right-split. In particular:

\begin{corollary}\label{corollary:nonsplit}
In the category whose objects are symplectic $K3$ surfaces and morphisms are symplectomorphisms, the short exact sequence of groups (associated to the fibration (\ref{fibration}))
\[
0 \to \pi_1 \mathrm{Symp}_0 (X, \omega ) \backslash \pi_1 \mathrm{Diff}_0 (X) \to \pi_1 \mathcal{S}_0 (X, \omega )\to \pi_0 \mathrm{Symp}_0 (X, \omega ) \to 0
\]
does not split in a natural way: both on the left or the right.
\end{corollary}

We saw that to the squared Dehn--Seidel twist $\tau_{L}^2$ we could associate a lift given by the loop of symplectic forms $\mathcal{O}_L$, but this required the additional choice of an orientation of $L$. Corollary \ref{corollary:nonsplit} says---for example---that such a natural construction would be impossible without this additional choice.

On the other hand, the short exact sequence (\ref{SESreps}) does split after localising away from $2$, by Lemma \ref{lemma:nonsplit}(2). Precomposing $\rho_0 : (\mathbb{Z}\Phi)[\frac{1}{2}]\to (\pi_1 \mathcal{S}_{0} (X, \omega ))^{\mathrm{ab}} [\frac{1}{2}]$ with the splitting described in the proof of that result yields the (unique) $W$--equivariant homomorphism
\[
\widetilde{\rho}_0 : P^{\mathrm{ab}}[\frac{1}{2}] \to (\pi_1 \mathcal{S}_{0} (X, \omega ))^{\mathrm{ab}} [\frac{1}{2}]
\]
sending $t_{e_i} \mapsto \frac{1}{2} ( \mathcal{O}_{L_i} + \mathcal{O}_{- L_{i}})$. Note that the definition of $\widetilde{\rho}_0$ does not involve orientations of the $L_i$'s anymore. The homomorphism $\widetilde{\rho}_0$ may also be regarded as an analogue of the symplectic representation $\rho_0 :   P^{\mathrm{ab}} \to ( \pi_0 \mathrm{Symp}_0 (X, \omega ))^{\mathrm{ab}} $ at the level of $(\pi_1 \mathcal{S}_0 (X, \omega ))^\mathrm{ab}$, since it lifts the symplectic representation along the connecting homomorphism associated to the fibration (\ref{fibration}):
\begin{center}
\begin{tikzcd}
\, & ( \pi_1 \mathcal{S}_{0}(X, \omega ))^{\mathrm{ab}}[\frac{1}{2}] \arrow{d}{(\ref{fibration})}\\
P^{\mathrm{ab}}[\frac{1}{2}] \arrow[dashed]{ru}{\widetilde{\rho_0}} \arrow{r}{\rho_0 [\frac{1}{2}]} & ( \pi_0 \mathrm{Symp}_{0}(X, \omega ))^{\mathrm{ab}}[\frac{1}{2}] .
\end{tikzcd}
\end{center}
It follows from Theorem \ref{theorem:diagram} that $\widetilde{\rho}_0$ splits $W$--equivariantly as well, for every closed symplectic $4$-manifold.

\section{Infinite-generation results}\label{section:infinite}

Let $(X, \omega )$ be a closed symplectic $4$-manifold. Consider once again the short exact sequence (\ref{SESdiffsymp}):
\begin{align*}
0 \to \pi_1 \mathrm{Symp}_0 (X,\omega ) \backslash \pi_1 \mathrm{Diff}_0 (X,\omega )\to \pi_1 \mathcal{S}_0 (X,\omega ) \to \pi_0 \mathrm{Symp}_0 (X,\omega )\to 0 .
\end{align*}
Through this, the group $\pi_1 \mathcal{S}_0 (X, \omega )$ can be regarded as a `hybrid' combination of the smoothly-trivial symplectic mapping class group $\pi_0 \mathrm{Symp}_0 (X, \omega )$ and the abelian group $\pi_1 \mathrm{Symp}_0 (X,\omega ) \backslash \pi_1 \mathrm{Diff}_0 (X,\omega )$ of `non-symplectic loops of diffeomorphisms'. The latter abelian group has been recently studied by J. Lin \cite{JLIN2022}, and the results in this section generalise his infinite-generation result \cite[Theorem E]{JLIN2022} to the groups $\pi_1 \mathcal{S}_0 (X, \omega )$ and $\pi_0 \mathrm{Symp}_0 (X, \omega )$.

\subsection{An application of the Family Switching Formula}

The next result `overrides' Propositions \ref{proposition:calculation}-\ref{proposition:A2}. Its proof relies on a different gluing result: the `Family Switching Formula' (\cite[Theorem 5.3]{JLIN2022} or \cite{liu_switching}).

\begin{proposition}\label{proposition:switching}
Let $(X, \omega )$ be a closed symplectic $4$-manifold, and $L$ an oriented Lagrangian sphere. Suppose $e \in H^2 (X, \mathbb{Z} )$ is a class of the form $e = \mathrm{PD}([S])$ for a smoothly embedded and oriented sphere $S \subset X$ with $S^2 - K \cdot S = -2$ and $[\omega]\cdot S \leq 0$
(for example, this holds if $S$ is a Lagrangian sphere). Fix a homological orientation of $X$. Then
\[
Q_e (\mathcal{O}_L ) = \begin{cases} \pm 1 & \text{  if  } e = \mathrm{PD}([L])\\
0 & \text{  otherwise}
\end{cases}.
\]
\end{proposition}

\begin{proof}
When $e = \mathrm{PD}([L])$ the result was proved in Proposition \ref{proposition:calculation}. Thus, suppose now that $e \neq \mathrm{PD}([L])$. 

We have
\begin{align*}
Q_e ( \mathcal{O}_L ) & = Q_e ( \mathcal{O}_L \cdot \mathcal{O}_{-L}^{-1} ) + Q_e (\mathcal{O}_{-L}) &  \text{(by additivity of } Q_e\text{)}\\
& = FSW_e ( \gamma_L ) + Q_e (\mathcal{O}_{-L})&  \text{(by Theorem \ref{theorem:generaliseddehntwist} and Lemma \ref{lemma:FSW})}.
\end{align*}

By the proof of Lemma \ref{lemma:kerO}, the loop $\mathcal{O}_L$ bounds a disk of symplectic forms on $X$ where the cohomology class of each symplectic form is $[\omega] - c \mathrm{PD}([L])$ for some $c\geq 0$. By Lemma \ref{lemma:taubes} it then follows that if $ e\cdot [L]\geq 0$ then $Q_{ e}(\mathcal{O}_L ) = 0$. Thus, it suffices to assume $e\cdot [L] < 0$, and then by the same argument $Q_e (\mathcal{O}_{-L}) = 0$. Thus, by the above:
\[
Q_e ( \mathcal{O}_L ) = FSW_e ( \gamma_L ) .
\]

Finally, by \cite[Proposition 8.4]{JLIN2022} applied to the case $\mathfrak{s} = \mathfrak{s}_\omega$, $S_1 = -S$, $S_0 = L$ we have $FSW_e ( \gamma_L ) = 0$ (our assumptions on $S$ ensure the hypothesis of that result hold). 
The aforementioned result is a consequence of the `Family Switching Formula' (\cite[Theorem 5.3]{JLIN2022} and \cite{liu_switching}). This concludes the proof.
\end{proof}

\subsection{Algebraic independence of squared Dehn--Seidel twists }

Let $\mathcal{L}(X, \omega )$ be the subset of $H^2 (X, \mathbb{Z} )$ consisting of cohomology classes Poincaré dual to the fundamental class of an oriented Lagrangian sphere. Consider the free abelian group $ \bigoplus_{e \in \mathcal{L}(X,\omega)}\mathbb{Z}$, with the generator corresponding to $e \in \mathcal{L}(X, \omega )$ denoted $T_e$. We have an involution $\iota : T_e \to T_{-e}$ of this abelian group, and corresponding invariant and anti-invariant subgroups: $\Big( \bigoplus_{e \in \mathcal{L}(X,\omega)}\mathbb{Z}\Big)^\pm := \mathrm{Ker}(\iota \mp \mathrm{Id} )$. These are freely generated by the elements $T_e \pm T_{-e}$, respectively. Thus, there is a natural identification of the invariant subgroup with
\[
\Big( \bigoplus_{e \in \mathcal{L}(X,\omega)}\mathbb{Z}\Big)^+  \cong  \bigoplus_{e \in \mathcal{L}(X,\omega)/\pm}\mathbb{Z} .
\]
but to obtain a similar identification of the anti-invariant subgroup we need to make a choice of section of the map $\mathcal{L}(X, \omega ) \to \mathcal{L}(X, \omega )/\pm$. There is a natural short exact sequence
\begin{align}
 0 \to\Big( \bigoplus_{e \in \mathcal{L}(X,\omega)}\mathbb{Z}\Big)^- \xrightarrow{\text{incl.}} \bigoplus_{e \in \mathcal{L}(X,\omega)}\mathbb{Z}  \to \Big( \bigoplus_{e \in \mathcal{L}(X,\omega)}\mathbb{Z}\Big)^+ \to 0  \label{SESsummand}
\end{align}
which splits non-canonically (again, after making a choice of section of $\mathcal{L}(X, \omega ) \to \mathcal{L}(X, \omega )/\pm$).


Now, we make a choice of an (isotopy class of) oriented embedded Lagrangian sphere $L_e$ in $(X, \omega )$ for each class $e \in \mathcal{L}(X, \omega )$, such that $e = \mathrm{PD}([L_e]) $. Using the canonical lifts $\mathcal{O}_{L_e} \in (\pi_1 \mathcal{S}_0 (X, \omega ))^{\mathrm{ab}}$ we define a homomorphism
\[
\mathcal{O} : \bigoplus_{e \in \mathcal{L}(X, \omega )}\mathbb{Z} \to  (\pi_1 \mathcal{S}_0 (X, 
\omega ))^{\mathrm{ab}} .
\]

For a fixed homological orientation of $X$ 
we consider the following homomorphism: 
\begin{align*}
 Q : (\pi_1 \mathcal{S}_0 (X, \omega ))^\mathrm{ab} \to \bigoplus_{e \in \mathcal{L}(X,\omega)} \mathbb{Z} \quad , \quad Q = \bigoplus_{e \in \mathcal{L}(X,\omega)} Q_e 
\end{align*}
\begin{remark}
By standard results in Seiberg--Witten theory (see \cite[Theorem 5.2.4]{morgan}) it follows that, for a given loop $\omega_t \in \pi_1 \mathcal{S}_0 (X, \omega )$ there are only finitely many classes $e \in H^2 (X, \mathbb{Z} )$ satisfying (\ref{condition1}) such that $Q_{e} (\omega_t ) \neq 0$. Thus $Q$ and $FSW$ indeed map to $\bigoplus_{e\in \mathcal{L}(X,\omega )}\mathbb{Z}$ (rather than the infinite direct product $\Pi_{e \in \mathcal{L}(X,\omega)} \mathbb{Z}$).
\end{remark}

\begin{proof}[Proof of Theorem \ref{theorem:summand}]
By Proposition \ref{proposition:calculation} we have $Q_e ( \mathcal{O}_{L_e} ) \pm 1$ ; by Proposition \ref{proposition:switching} we have $Q_{e^\prime} (\mathcal{O}_{L_e}) = 0$ for all $e^\prime \in \mathcal{L}(X,\omega )$ with $e^\prime \neq e$. Thus, the composition $Q \circ \mathcal{O}$ sends $T_e$ to $\pm T_e$ for every $e \in \mathcal{L}(X,\omega )$. Thus, post-composing $Q$ with the automorphism of $\bigoplus_{e \in \mathcal{L}(X,\omega )}\mathbb{Z}$ which sends $T_e$ to $\pm T_e$ if $Q_e (\mathcal{O}_{L_e}) = \pm 1$, we obtain a new homomorphism $Q^\prime$ so that $Q^\prime \circ \mathcal{O} = \mathrm{Id}$, so $Q^\prime$ provides the desired splitting.
\end{proof}

Suppose further that the Lagrangians are chosen so that $L_e = -L_{-e}$. Then, because the generalised Dehn twist $\gamma_{L_e} \in \pi_1 \mathrm{Diff}_0 (X)$ satisfy $\gamma_{L_e} = - \gamma_{-L_e}$ (Lemma \ref{lemma:gammasymmetries}) the restriction of the homomorphism $\mathcal{O}$ to the anti-invariant subgroup factors through the inclusion in (\ref{SESdiffsymp}) as the homomorphism 
\[
\gamma : \Big( \bigoplus_{e \in \mathcal{L}(X,\omega)}\mathbb{Z}\Big)^- \to \pi_1 \mathrm{Symp}_0 (X,\omega ) \backslash \pi_1 \mathrm{Diff}_0 (X,\omega ) \quad , \quad T_{e} - T_{-e} \mapsto \gamma_{L_e} .
\]
Hence, the projection of $\mathcal{O}$ to $(\pi_0 \mathrm{Symp}_0 (X,\omega))^\mathrm{ab}$ descends along the projection in (\ref{SESsummand}) to give the homomorphism
\[
\tau^2 : \Big( \bigoplus_{e \in [\mathcal{L}]} \mathbb{Z}\Big)^+ \to (\pi_0 \mathrm{Symp}_0 (X, \omega ))^\mathrm{ab}\quad , \quad T_e + T_{-e} \mapsto \tau_{L_e}^2 .
\]

The restriction of $Q$ to the subgroup $\pi_1 \mathrm{Symp}_0 (X,\omega ) \backslash \pi_1 \mathrm{Diff}_0 (X,\omega )$ coincides, by Lemma \ref{lemma:FSW}, with the homomorphism
\begin{align*}
FSW : \pi_1 \mathrm{Symp}_0 (X,\omega ) \backslash \pi_1 \mathrm{Diff}_0 (X,\omega ) \to \bigoplus_{e\in \mathcal{L}(X,\omega)} \mathbb{Z} \quad , \quad 
FSW = \bigoplus_{e\in \mathcal{L}(X,\omega)} FSW_e .
\end{align*}
(Recall that $FSW_e$ vanishes on the image $\pi_1 \mathrm{Symp}_0 (X, \omega ) \to \pi_1 \mathrm{Diff}_0 (X )$ by Lemma \ref{lemma:taubes}).
When $(X, \omega )$ is a symplectic Calabi--Yau surface with $b^+ (X)>1$, then by Lemmas \ref{lemma:FSW}-\ref{lemma:symplecticconjugation} we have $FSW_e = -FSW_{-e}$ and thus $FSW$ takes values in the anti-invariant subgroup. 
From this and (\ref{SESsummand}), the homomorphism $Q$ thus descends to a homomorphism 
 \begin{align*}
q : \pi_0 \mathrm{Symp}_0 (X, \omega ) \to \Big( \bigoplus_{e \in \mathcal{L}(X,\omega)} \mathbb{Z}\Big)^+ \cong \bigoplus_{e \in \mathcal{L}(X,\omega )/\pm} \mathbb{Z} .
 \end{align*}

\begin{proof}[Proof of Theorem \ref{theorem:summandK3}]
In the proof of Theorem \ref{theorem:summand} we established $Q^\prime \circ \mathcal{O} = \mathrm{Id}$, where $Q^\prime$ was obtained by postcomposing $Q$ by a further automorphism which simply swapped the signs of the generators $T_e$. Because now $L_e = - L_{-e}$, then by Proposition \ref{proposition:calculation} we have $Q_e ( \mathcal{O}_{L_e} ) = Q_{-e}(\mathcal{O}_{-L_e} ) = \pm 1$, i.e. the signs match. Thus, we can modify the homomorphisms $FSW$ and $q$ to obtain new homomorphisms $FSW^\prime$ and $q^\prime$, defined by post-composing with the automorphisms sending $T_e - T_{-e} \mapsto \pm ( T_e - T_{-e} )$ if $Q_{e}(\mathcal{O}_{L_e}) = \pm 1$ and $T_e + T_{-e} \mapsto \pm (T_e + T_{-e} )$ if $Q_e (\mathcal{O}_{L_e} ) = \pm 1$, respectively. In particular, $FSW^\prime$ still takes values in the anti-invariant subgroup, and from $Q^\prime \circ \mathcal{O} = \mathrm{Id}$ it follows that $FSW^\prime \circ \gamma = \mathrm{Id}$. By (\ref{SESdiffsymp}) and (\ref{SESsummand}) this implies that $q^\prime \circ \tau^2 = \mathrm{Id}$, and hence $q^\prime$ provides the desired splitting.
\end{proof}

Summarising: for $(X, \omega )$ a symplectic Calabi--Yau surface with $b^+ (X) >1$ there is the following commutative diagram where rows are exact and where the vertical maps are split-injective,
\begin{equation*}
\begin{tikzcd}
\text{ } & \pi_1 \mathrm{Symp}_0 (X, \omega ) \backslash \pi_1 \mathrm{Diff}_0 (X)   \arrow{r} & (\pi_1 \mathcal{S}_0 (X, \omega ))^{\mathrm{ab}}  \arrow{r} & (\pi_0 \mathrm{Symp}_0 (X, \omega ))^{\mathrm{ab}} \arrow{r} & 0 \\
  0  \arrow{r} &  \Big( \bigoplus_{e \in \mathcal{L}(X,\omega)}\mathbb{Z}\Big)^- \arrow{r} \arrow{u}{\gamma} &  \bigoplus_{e \in \mathcal{L}(X,\omega)}\mathbb{Z}  \arrow{r}\arrow{u}{\mathcal{O}} &  \Big( \bigoplus_{e \in \mathcal{L}(X,\omega)}\mathbb{Z}\Big)^+ \arrow{r} \arrow{u}{\tau^2} &  0
\end{tikzcd}.
\end{equation*}
Thus, when $\mathcal{L}(X,\omega )$ is an infinite set this provides a common reason for the infinite-generation of the three groups in the top row.

\newpage

\bibliographystyle{alpha}
\bibliography{main.bib}

\end{document}